\documentclass[12pt]{amsart}
\usepackage{latexsym,amssymb,amsmath,amsthm}
\usepackage{comment,float,graphicx,plain}
\newtheorem{thm}{Theorem}

\theoremstyle{definition}

\def\tcat{{\rm tcat}}
\def\cat{{\rm cat}}

\def\crit{{\rm crit}}
\newcommand{\C}{\mathbb{C}}   
\def\G{\mathcal{G}} \def\A{\mathcal{A}} \def\H{\mathcal{H}} \def\U{\mathcal{U}} 
\def\O{\mathcal{O}} \def\P{\mathcal{P}} \def\B{\mathcal{B}} \def\C{\mathcal{C}}
 \def\dim{\rm{dim}} \def\E{\mathcal{E}}

\title{A notion of graph homeomorphism}
\author{Oliver Knill}
\date{January 12, 2014}
\address{
        Department of Mathematics \\
        Harvard University \\
        Cambridge, MA, 02138
        }

\begin{document}
\maketitle
\begin{abstract}
We introduce a notion of graph homeomorphisms which  uses the concept
of dimension and homotopy for graphs. It preserves the dimension of a subbasis,
cohomology and Euler characteristic. Connectivity and homotopy look as in classical topology.
The Brouwer-Lefshetz fixed point leads to the following discretiszation of the 
Kakutani fixed point theorem: any graph homeomorphism $T$ with nonzero Lefschetz 
number has a nontrivial invariant open set which is fixed by $T$. 
\end{abstract}

\section{The definition}

A classical topology $\O$ on the vertex set $V$ of a finite simple graph $G=(V,E)$ 
is called a {\bf graph topology} if there is a sub-base $\B$ of $\O$ consisting of 
contractible subgraphs such that the intersection of any
two elements in $\B$ satisfying the {\bf dimension assumption} 
$\dim(A \cap B) \geq {\rm min}(\dim(A),\dim(B))$ 
is contractible, and every edge is contained in some $B \in \B$. 
We ask the {\bf nerve graph} $\G$ of $\B$ to be homotopic to $G$, where 
$\G = (\B,\E)$ has edges $\E$ consisting of all pairs $(A,B) \in \B \times \B$ for which
the dimension assumption is satisfied. \\

The {\bf dimension} of $G$ \cite{elemente11,randomgraph} is defined as 
$\dim(G) = \frac{1}{|V|} \sum_{x \in V} (1+\dim(S(x)))$
with the induction assumption that $\dim(\emptyset)=-1$ 
and that $S(x)$ is the {\bf unit sphere graph} of a vertex $x$, the subgraph of $G$
generated by all vertices attached to $x$. For a subgraph $H=(W,F)$ of $G$, define the
{\bf relative topology} $\dim_G(H) = \frac{1}{|W|} \sum_{x \in W} \dim(S(x))$
and especially $\dim(x)=\dim_G(x)=\dim(S(x))$ for $x \in V$. 
In the requirement for $\B$ we have invoked dimensions $\dim(A),\dim(B)$ and not 
relative dimensions $\dim_G(A), \dim_G(B)$. \\

A topology $\O$ is {\bf optimal} if any intersection $\bigcap_{C \in \C} C$ with $\C \subset \B$
is contractible and the {\bf dimension functional }
$(1/|\B|) \sum_{A \in \B}$ $|\dim_G(A)$ $-$ $\dim(A)|$ can not be reduced by
splitting or merging elements in $\B$, or enlarging or shrinking any $B \in \B$ without 
violating the graph topology condition. In order for $\B$ to be optimal we also ask $\B$ to be minimal, 
in the sense that there is no proper subset of $\B$ producing a topology.
To keep things simple, we do not insist on the topology to be optimal.
Note that $\dim(\G)$ and $\dim(G)$ differ in general.
While it is always achievable that $\G$ is isomorphic to $G$, such a topology 
is not always optimal as Figure~(\ref{fine}) indicates. Figure~(\ref{poster}) illustrates
optimal topologies. \\

A graph with a topology $\O$ generated by a sub-base $\B$ is called a {\bf topological graph}.
As for any topology, the topology $\O$ generated by $\B$ has a lattice structure, 
As we will see, there is always a topology on a graph, an example being the {\bf discrete topology} generated by 
star graphs centered at vertices. Often better is the topology generated by the set $\B$ of unit
balls. \\

The topology defines the weighted nerve graph $(\G,\dim)$, where $\dim$ is the function on vertices $\B$ 
given by the dimension of $A \in \B$. The image of $\dim$ on $\B$ is the {\bf dimension spectrum}. The 
average of ${\rm dim}$ is the {\bf topological dimension} of the topological graph. Unlike the dimension
$\dim(G)$ of the graph, the dimension spectrum $\{ {\rm dim}(B) \; | \; B \in \B \}$ and
the topological dimension $(1/|\B|)\sum_{B \in \B} {\rm dim}(B)$ as well as the dimension of the
nerve graph $\dim(\G)$ depend on the choice of the topology. \\

Given two topological graphs $(G,\B,\O)$ and $(H,\C,\P)$, a map $\phi$ from $\O$ to $\P$ is called 
{\bf continuous}, if it induces a graph homomorphism of the
nerve graphs such that ${\rm dim}(\phi(A)) \leq {\rm dim}(A)$ for every $A \in \B$. 
A graph homomorphism $\phi: \G \to \H$ induces a
lattice homomorphism $\phi: \O \to \P$.
If $\phi$ has an inverse which is continuous too, we call it a {\bf graph homeomorphism}. 
A graph homeomorphism for $G$ is a graph isomorphism for the nerve graph $\G$. 
The notion defines an equivalence relation between topological graphs.  \\

This definition provides also an equivalence relation between finite simple graphs: 
two graphs are {\bf equivalent} if each can be equipped with a graph topology
such that the topological graphs are homeomorphic. A {\bf optimal equivalence} is the property
that the two graphs are equivalent with respect to topologies which are both optimal. 
In order to keep the definitions simple, we again do not focus on optimal equivalence. As examples in
classical topology and functional analysis in particular show, it is often useful to work 
with different topologies on the same space and allow also for very weak or very strong topologies.  \\

\begin{figure}
\scalebox{0.12}{\includegraphics{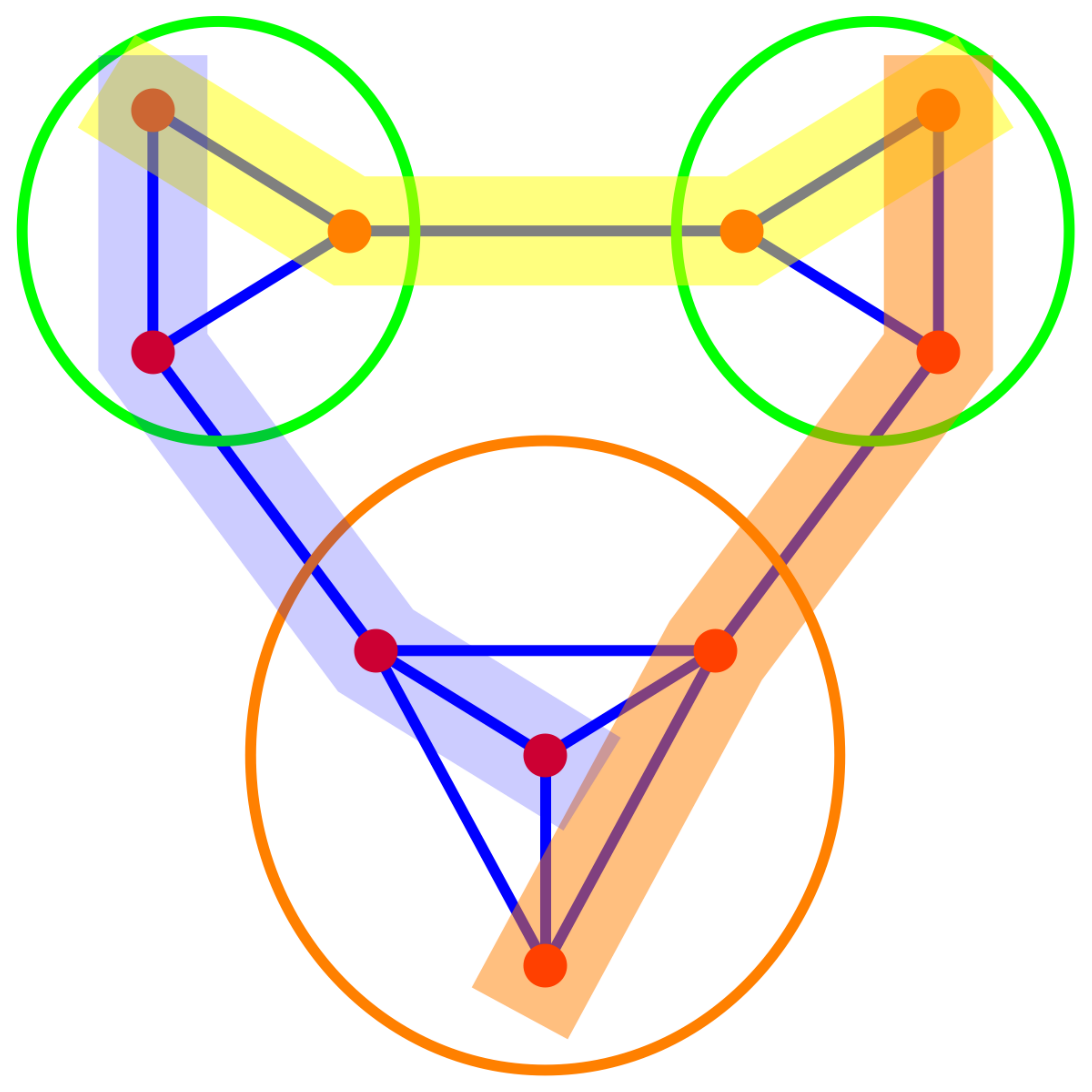}}
\scalebox{0.12}{\includegraphics{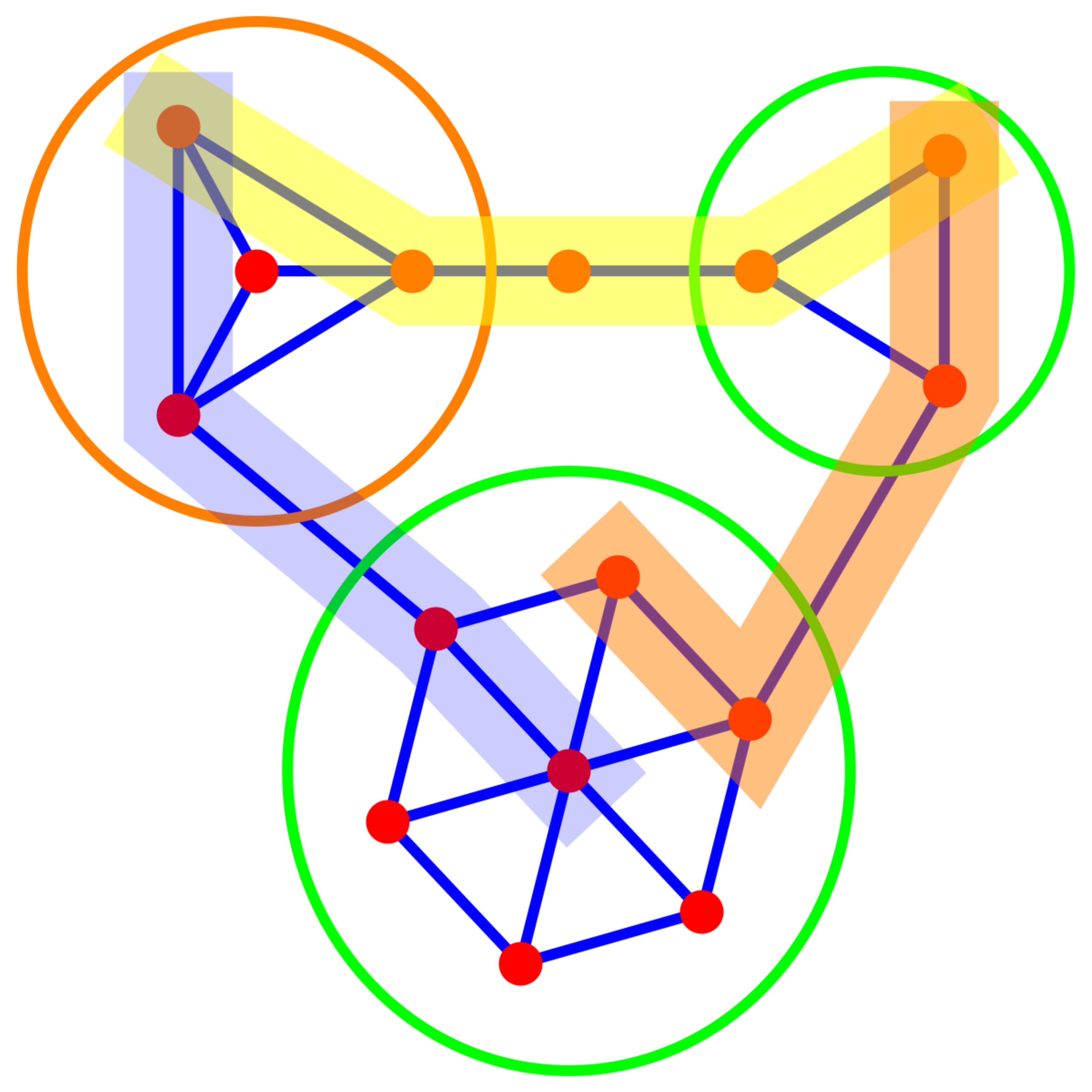}}
\caption{
Two homeomorphic graphs $H,G$ of order 10 and 15 are each equipped 
with an optimal subbase consisting of 6 sets. The dimension spectrum 
for both is $\{1,2,1,2,1,3 \; \}$ so that the topological dimension is $10/6=1.25$. 
We have ${\rm dim}(H)=131/60=2.183 \dots $ and
${\rm dim}(G) = 15/7 = 2.143 \dots$. The nerve graph $\G$
is the cyclic graph $\G=C_6$ with Lebesgue covering dimension $1$. 
When the nerve graph is seen as a subgraph of $G$ it appears as a
{\bf deformation retract} of $G$. 
\label{poster}
}
\end{figure}

Why is this interesting? We want graphs to be deformable in a {\bf rubber geometry} type fashion
and have basic topological properties like connectivity, dimension, Euler characteristic or homotopy
class preserved by the deformation. The wish list contains that all 
noncontractible cyclic graphs should be homeomorphic and that the octahedron
and icosahedron should be homeomorphic. If we want topologies $\O$ in the sense of classical point set
topology, a difficulty is that $\O$ often has different connectivity features
because finite topologies often have many sets which are both open and closed. This difficulty is bypassed by enhancing the topology using a 
sub-basis $\B$ of $\O$ of contractible sets in which dimension plays a crucial role: it is used to see 
which basis elements are linked and require this link structure called "nerve" to be homotopic to the graph 
itself. Dimension is crucial also when defining "continuity" because dimension should not increase under 
a continuous map. The definitions are constructive: we can start with the topology generated by star graphs
which generates the discrete topology and then modify the elements of the subbasis 
so that the dimension of the basis elements approximates the dimensions as embedded in the graph. 
While the proposed graph topology works for arbitrary finite simple graphs, it 
is inspired from constructions for manifolds, where the subbasis $\B$ 
is related to a \v{C}ech cover and the nerve graph corresponds to the nerve graph of the classical cover. 
Furthermore, for graphs without triangles, the homeomorphisms coincide with classical homeomorphisms 
in the sense of topological graph theory in which embedding questions of graphs
in continuum topological spaces plays an important role. An other motivation for a pointless approach 
are fixed point theorems for set-valued maps which are important in applications like game theory. This was
our entry point to this topic. Instead of set-valued maps, we can look directly at automorphisms of the 
lattice given by the topology and forget about the points.  That the notion is natural can be seen also from the fact
that - as mentioned below - the classical notion of homotopy using continuous deformations 
of maps works verbatim for graphs: there are continuous maps $f:H \to G, g: G \to H$ in the pointless topology
sense defined here such that $f \circ g$ and $g \circ f$ are homotopic to the identity.
The classical formulation of homotopy obviously is based on topology. 

\section{Results}

Ivashchenko homotopy \cite{I94} is based on earlier constructs put forward in \cite{Whitehead}. 
With a simplification \cite{CYY}, it works with Lusternik-Schnirelmann
and Morse theory \cite{josellisknill}. The definition is inductive: the one-point graph $K_1$ is {\bf contractible}.
A {\bf homotopy extension} of $G$ is obtained by selecting a contractible subgraph $H$ and making a pyramid
extension over $H$, which so adds an other vertex $z$ building a {\bf cone}.
The reverse step is to take a vertex $z$ with contractible unit sphere 
$S(z)$ and remove $z$ together with all connections from $z$. 
Two graphs are {\bf homotopic} if one can get from one to the other by a sequence of homotopy steps. A graph 
homotopic to the one point graph $K_1$ is called {\bf contractible}. Examples like the ``dunce hat" show
that this can not always be done by homotopy reductions alone. The space might first have to be thickened in order
to be contracted later to a point. 

\begin{figure}[H]
\scalebox{0.35}{\includegraphics{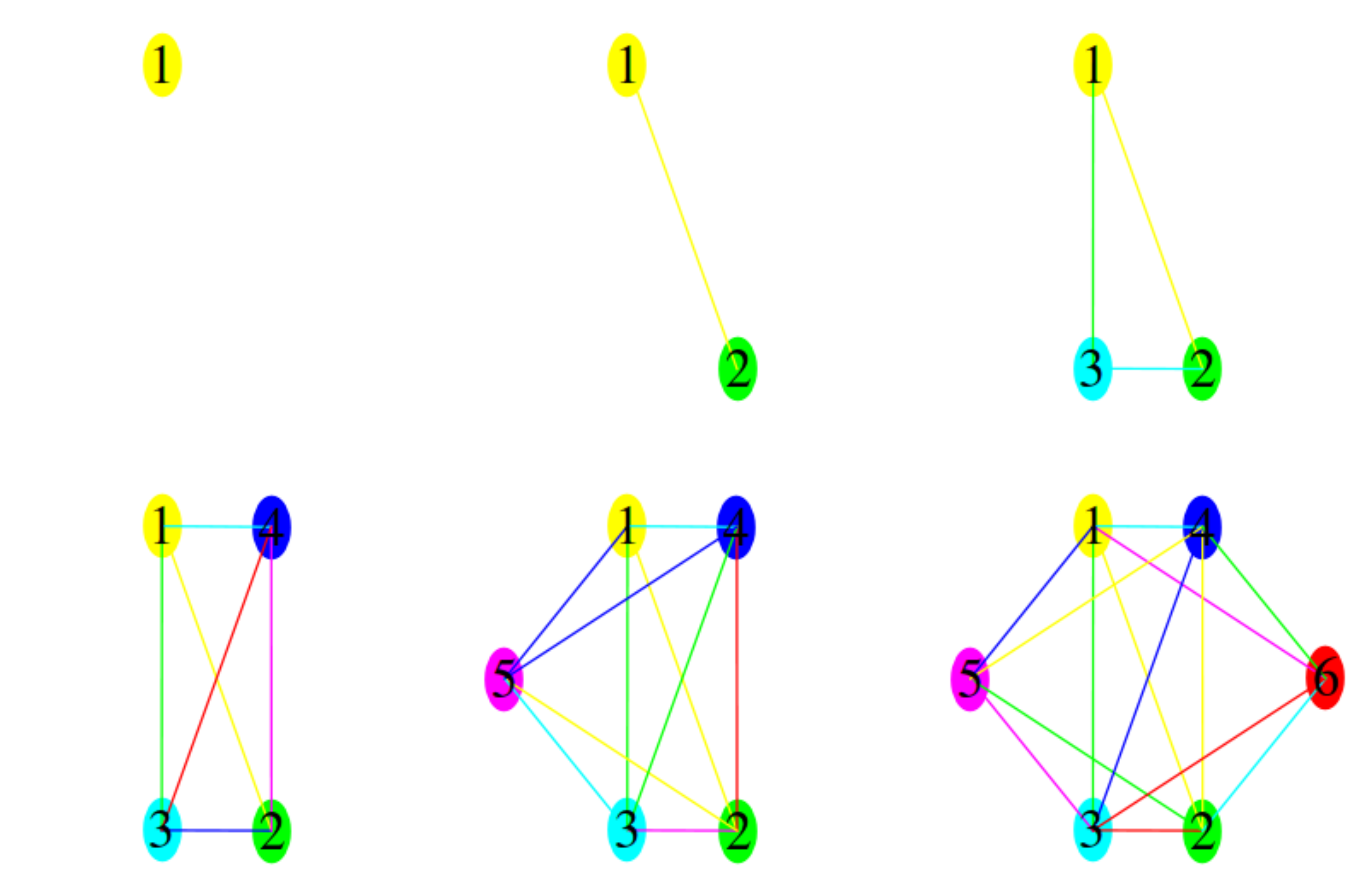}}
\caption{
Ivashkenko homotopy steps illustrate how a contractible graph is built
up from pyramid constructions, which builds cones over smaller contractible
subgraphs. 
\label{homotopy}
}
\end{figure}

Morse theory illustrates why homotopy is natural: given an injective function $f$ on $V$, we can look at the 
filtration $\{ f \leq c \; \}$ of the graph and define the {\bf index} ${\rm i}_f(x) = 1-\chi(S^-(x))$, where 
$S^-(x) = \{ y \in G \; | \; y \in S(x), f(y) \leq f(x) \; \}$. Start with $c={\rm min}(f)=f(x)$ such that $x_0$
is the minimum implying ${\rm i}_f(x_0)=1$. As $c$ increases, more and more vertices 
are added. If $S^-(x)$ is contractible, then $x$ is regular point and the addition a homotopy extension stop and 
${\rm ind}(x)=0$. If $S^-(x)$ is not contractible and in particular if the index $i_f(x)$ is not zero, 
then $x$ is a critical point. Because
Euler characteristic satisfies $\chi(A \cup B) = \chi(A) + \chi(B) - \chi(A \cap B)$, 
the sum of all indices is the Euler characteristic of $G$. We have just proven
the {\bf Poincar\'e-Hopf} theorem 
$\sum_{x \in V} {\bf i}_f(x) = \chi(G)$ \cite{poincarehopf}.
It turns out \cite{indexexpectation} that averaging the index over all possible functions using the product
topology gives the {\bf curvature} $K(x)= \sum_{k=0}^{\infty} (-1)^k V_{k-1}(x)/(k+1)$, where
$V_k(x)$ be the number of $K_{k+1}$ subgraphs of $S(x)$ and $V_{-1}(x)=1$. This leads to an other proof
of the {\bf Gauss-Bonnet-Chern} theorem $\sum_{x \in V} K(x) = \chi(G)$ \cite{cherngaussbonnet}. 
These results are true for any finite simple graph \cite{knillcalculus}. For {\bf geometric graphs} of dimension $d$
graphs for which every unit sphere is a Reeb graph (a geometric graph of dimension $d-1$ which admits an injective
function with exactly two {\bf critical points}, points where $S(x) \cap \{ y \; | \; f(y)<f(x) \; \})$ is not
contractible), then the curvature is zero for odd-dimensional geometric graphs \cite{indexformula}.\\

\begin{figure}
\scalebox{0.29}{\includegraphics{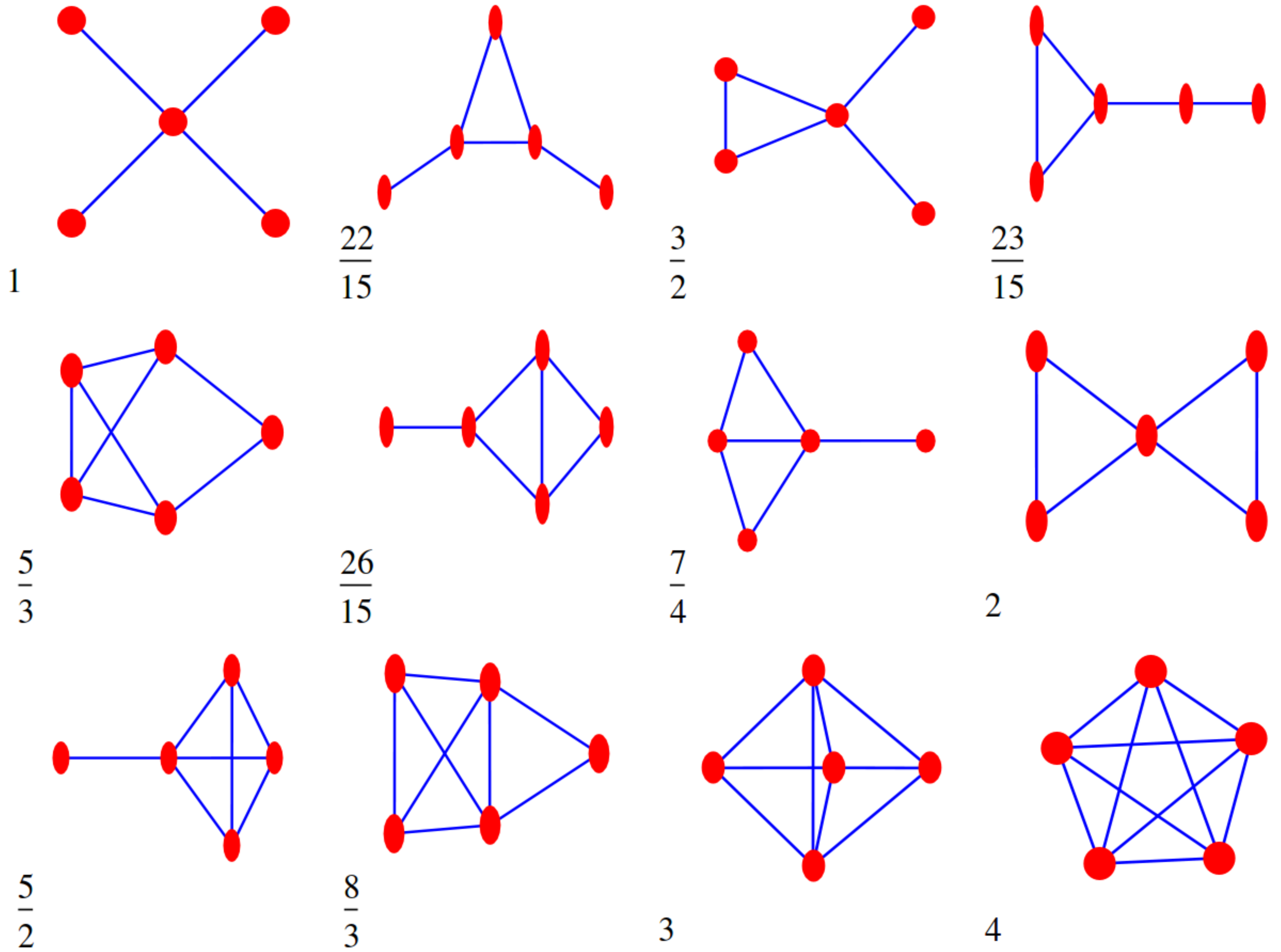}}
\caption{
The dimension functional takes 12 different values on 
the set of 728 connected graphs with $5$ vertices. Here is a
choice of 12 representatives picked for each dimension value. 
\label{dimension}
}
\end{figure}

{\bf Graph cohomology} uses the set $\G_k$ of $K_{k+1}$ subgraphs of $G$ called {\bf cliques}. 
An orientation on each maximal cliques induce an orientations on sub-cliques.
Define the {\bf exterior derivative }
$d_kf(x) = \sum_{k=0}^n f(x_0, \dots, \hat{x}_k, \dots, x_n)$
on the $v_k$-dimensional vector space $\Omega_k$ of all functions
on $\G_k$ which are {\bf alternating} in the sense that
$f(\pi x) = (-1)^{\pi} f(x)$, where $(-1)^{\pi}$ is the sign of the 
permutation $\pi$ of the coordinates $(x_0, \dots ,x_n)=x$ of $x \in \G_k$. The orientation fixes a
still ambiguous sign of $f$. While the linear map $d_k: \Omega_k \to \Omega_{k+1}$
depends on the orientation, the {\bf cohomology groups} $H^k(G) = {\rm ker}(d_{k})/{\rm im}(d_{k-1})$ 
are vector spaces which do not depend on the choice of orientations. It corresponds to a choice of 
the basis.  Hodge theory allows to realize $H^k(G)$ as ${\rm ker}(L_k)$, 
where $L_k$ is the Laplacian $L=(d+d^*)^2=D^2$ restricted to $\Omega_k$. Its dimension 
is the $k$'th Betti number $b_k = {\rm dim}(H^k(G))$. In calculus lingo, $d_0$ is the {\bf gradient}, 
$d_1$ the {\bf curl} and $d_0^*$ the {\bf divergence} and $L_0=d_0^* d_0$ is the {\bf scalar Laplacian}. 
The matrix $D=d+d^*$ is the {\bf Dirac operator} of the graph. It is unique up to orthogonal conjugacy
given by the orientation choice. Many results from the continuum hold verbatim \cite{knillmckeansinger}.
Graph cohomology is by definition equivalent to simplicial cohomology and formally equivalent to any 
discrete adaptation of de Rham cohomology. In particular, if $v_k = \dim(\Omega_k)$ is the cardinality of $\G_k$, then 
the {\bf Euler-Poincar\'e formula} $\sum_{k=0}^{\infty} (-1)^k v_k = \sum_{k=0}^{\infty} (-1)^k b_k$ holds. 
All this works for a general finite simple graph and that no geometric assumptions whatsoever is needed.
Only for {\bf Stokes theorem}, we want to require that the boundary $\delta G$ is a graph. Stokes tells
that if $H$ is a subgraph of $G$ which is a union of $k$-dimensional simplices which have compatible orientation
and $f \in \Omega_{k-1}$ then $\sum_{x \in H} df(x) = \sum_{x \in \delta H} f(x)$. The formula holds
due to cancellations at intersecting simplices. When graphs are embedded in smooth manifolds, one can be
led to the {\bf classical Stokes theorem} after introducing the standard 
calculus machinery based on the concept of "limit". \\

As usual, lets call a graph $G=(V,E)$ {\bf path connected} if for any two vertices $x,y$, there is a path 
$x=x_0,x_1,\dots,x_n=y$ with $(x_{i},x_{i+1}) \in E$ which connects $x$ with $y$. 
Path connectedness is what traditionally is understood with {\bf connected} in graph theory. Lets call a topological
graph $(G,\O)$ to be {\bf connected} if $\B$ can not be written as a union of two sets $\B_1,\B_2$ which have
no common intersection. 
This notion of connectedness is equivalent to the one usually given if a topology has a subbasis
given by connected sets. Lets denote with{\bf $1$-homeomorphic} the classical notion of homeomorphism in topological graph theory:
a graph $H$ is $1$-homeomorphic to $G$ if it can be deformed to $G$ by applying 
or reversing {\bf barycentric subdivision steps} of edges. Finally, lets call a graph {\bf $K_3$-free}, if it contains 
no triangles. 

\begin{thm} Every graph has an optimal graph topology. \label{1} \end{thm}
\begin{thm} Homeomorphisms preserve the dimension spectrum.     \label{2} \end{thm}
\begin{thm} Homeomorphic graphs are homotopic.                  \label{3} \end{thm}
\begin{thm} Homeomorphic graphs have the same cohomology. \label{4} \end{thm}
\begin{thm} Homeomorphic graphs have the same $\chi$. \label{5} \end{thm}
\begin{thm} Connected and path connected is always equivalent.    \label{6} \end{thm}
\begin{thm} $1$-homeomorphic $K_3$-free graphs are homeomorphic.              \label{7} \end{thm}

{\bf Remarks.} \\
{\bf 1)} The definitions have been chosen so that the proofs are immediate. \\
{\bf 2)} As in the continuum, the curvature, indices,
the cluster coefficient or the average length are not topological invariants. \\
{\bf 3)} Theorem~(\ref{2}) essentially tells that combinatorial cohomology 
on a graph agrees with \v{C}ech cohomology defined by the topology. \\
{\bf 4)} The assumptions imply that an optimal topology $\O$ has a basis which 
consists of contractible sets, where the notion of basis is the classical notion
as used in set theoretical topology.  \\

\begin{figure}
\scalebox{0.14}{\includegraphics{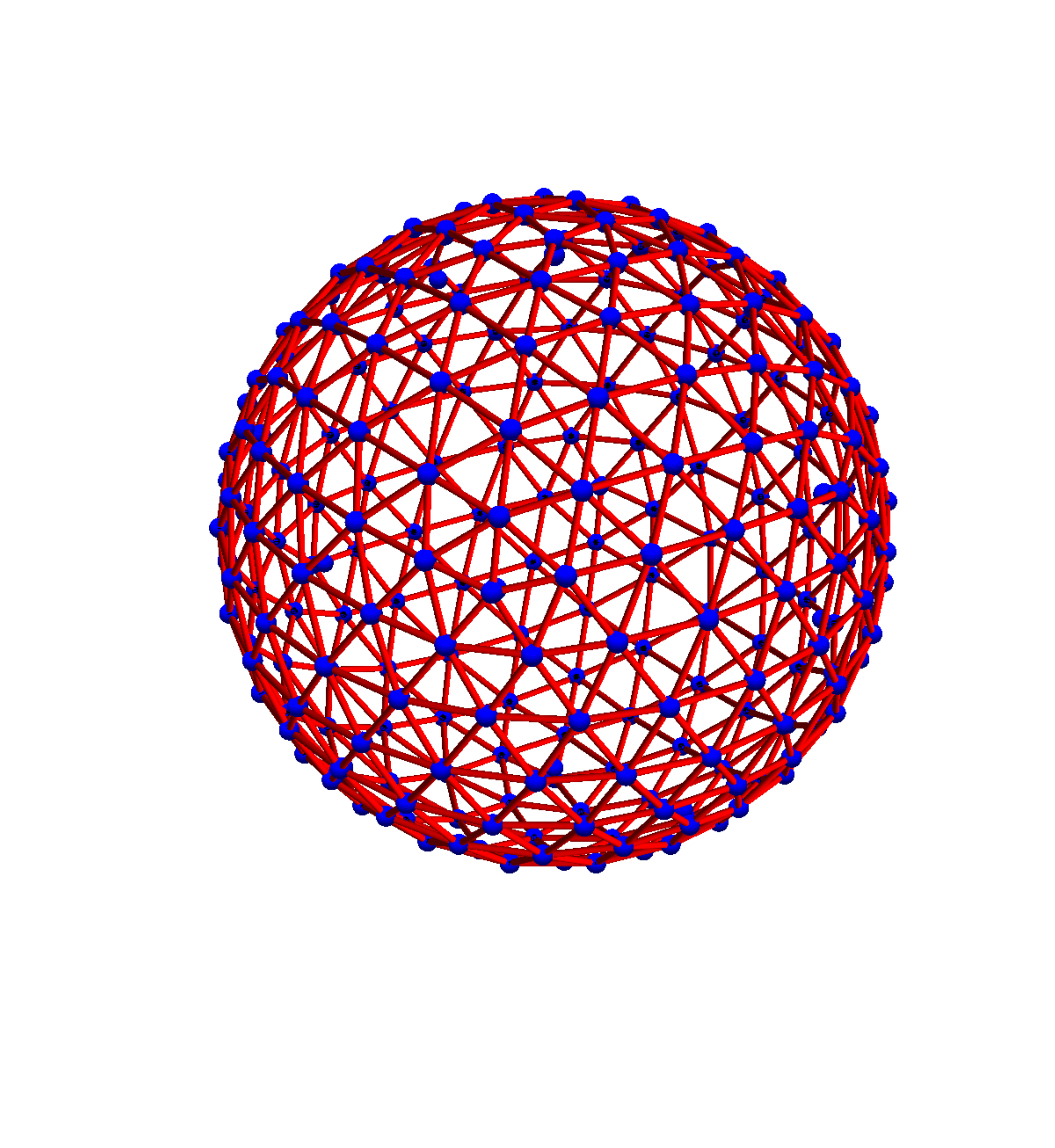}}
\scalebox{0.14}{\includegraphics{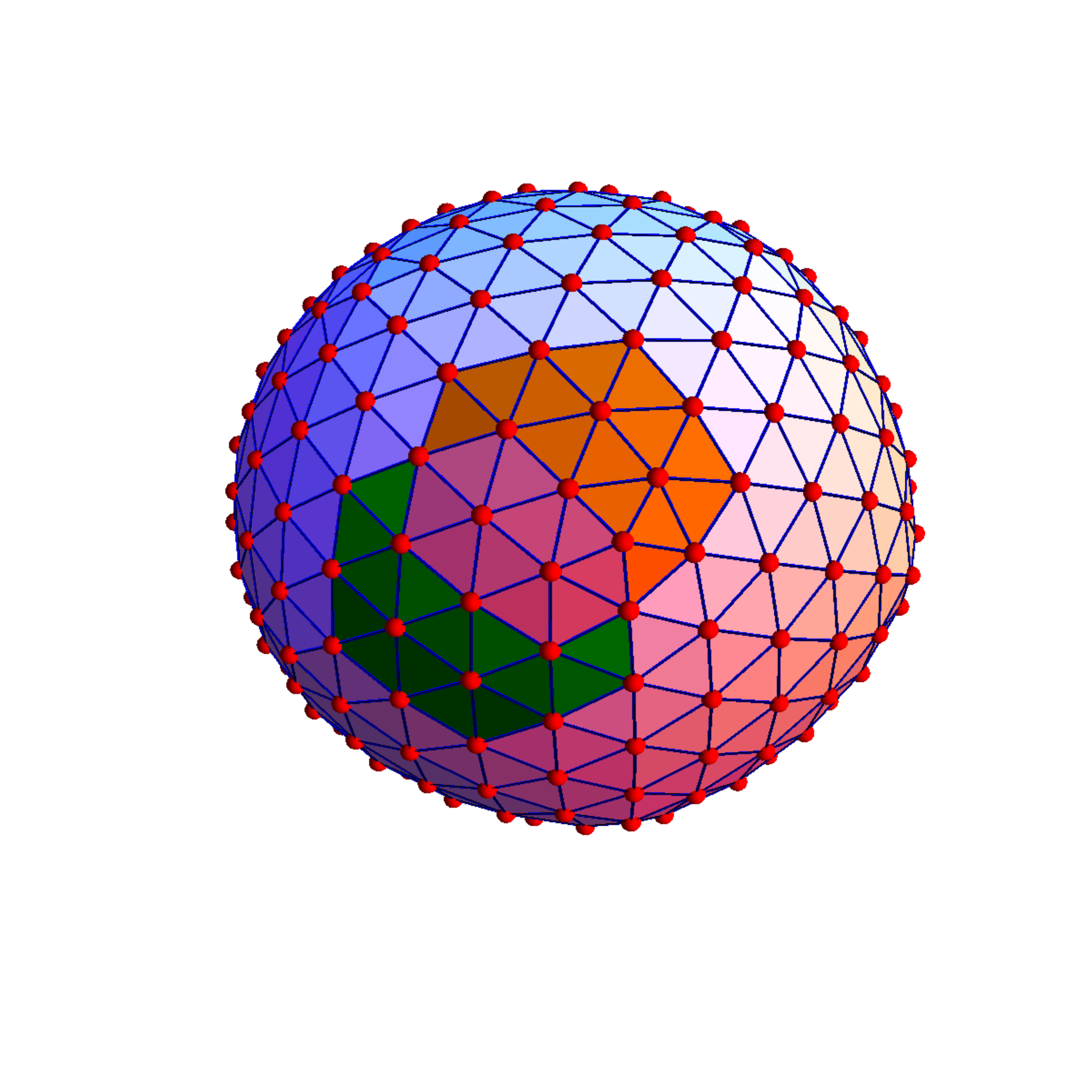}}
\caption{
The displayed two-dimensional graph $G$ with 252 vertices, 750 edges and 500 triangles is 
homeomorphic to an icosahedron. We can find a set $\B_1$ consisting of 12 open sets
which lead to the icosahedron $\G_1$ as the nerve graph on which ${\rm dim}$ is constant
$2$. An other topology takes $\B_2$ as the set of unit balls which are wheel graphs $W_5$ or $W_6$. 
The nerve of $\B_2$ is $G=\G_2$ itself and the dimension again constant $2$. A third topology 
in which we take the star graphs $\B_3$ centered at vertices also has $G=\G_3$ as the nerve graph
but the dimension function is constant $1$. Even so the nerve graphs $\G_2$ and $\G_3$ are the
same, the weighted graphs $(\G_2,{\rm dim}),(\G_3,{\rm dim})$ are different. The set $\B_4$
consisting of all geodesic balls of radius $1/2$ leads to the discrete topology too but the nerve 
graph is not homotopic to $G$, it has no vertices and Euler characteristic is $252$ and the 
dimension is constant $0$. $\B_4$ is not a graph topology. Finally, there is a topology $\B_5$
with $6$ elements which makes the graph homeomorphic to the octahedron. 
\label{fullerene}
}
\end{figure}

\section{Proofs}

{\bf Proof of \ref{1}:} \\
Any finite simple graph has a topology. It is the topology generated by 
{\bf star graphs at a vertex $x$}: this is the smallest graph which contains all edges attached to $x$. 
The nerve graph of this topology is the graph itself: each element is one-dimensional as is the intersection
so that the intersection assumption holds. Two star graphs with centers of geodesic distance $2$ are not 
connected because their intersection is zero-dimensional only. This proves existence. 
(The unit ball topology is often
natural too, especially in the case when $G$ is a triangularization of a manifold.)
To get optimality, start with a topology and increase, split 
or remove elements as long as all the topology conditions are satisfied as long as we decreases the dimension 
functional.  Once we can no more increase, we have an optimal topology. 
It might be a local extremum of the dimension functional only. Since there are only finitely many topologies,
we certainly also could get a global extremum, even so it might be costly to find it. \\

\begin{figure}[H]
\scalebox{0.12}{\includegraphics{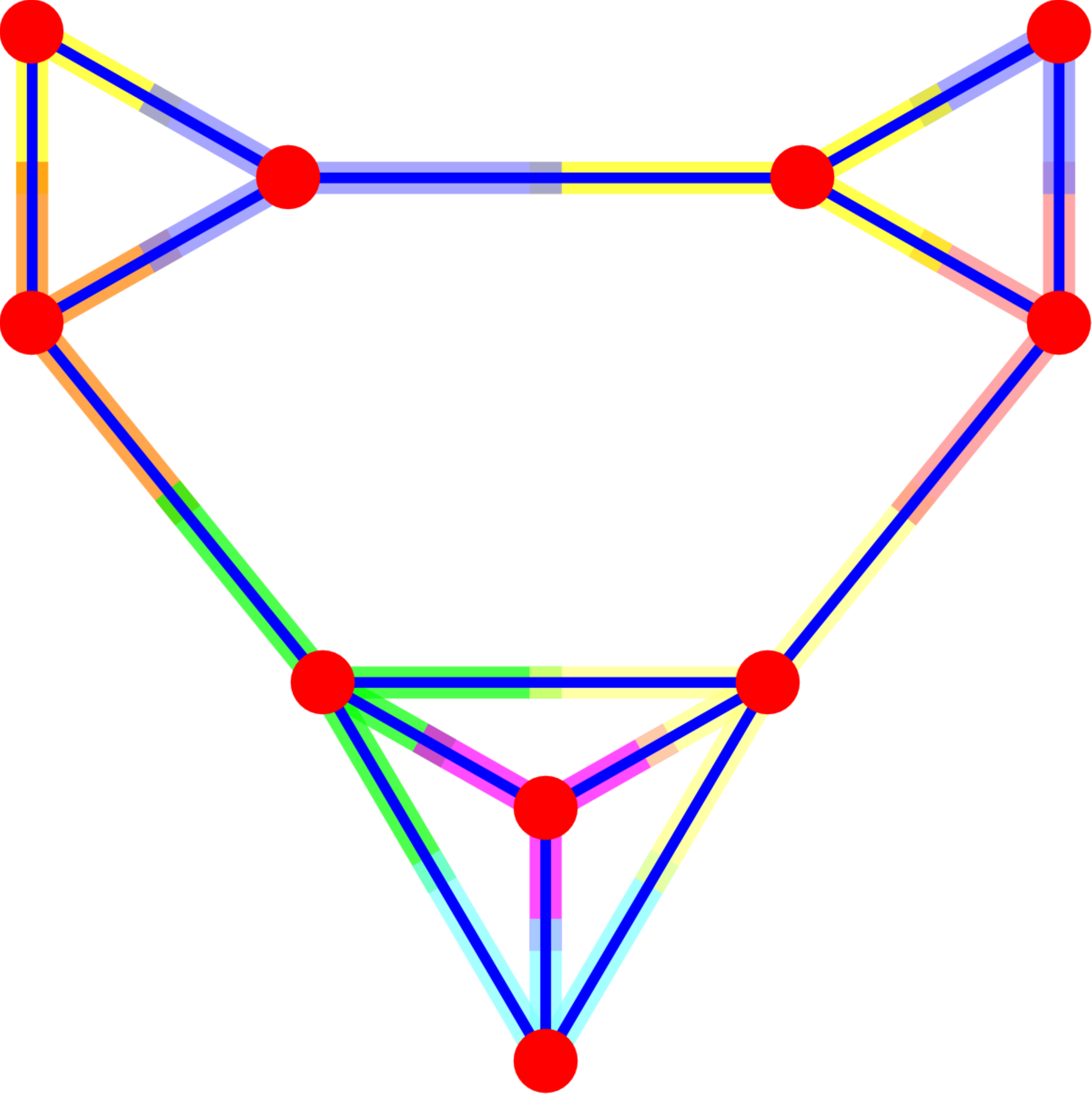}}
\caption{
The finest topology on a graph consists of star graphs centered at the
vertices. Its nerve is the graph itself. The finest topology is rarely optimal, has topological
dimension $1$ and always exists. An optimal topology for this graph is given in Figure~(\ref{poster}).
\label{fine}
}
\end{figure}

{\bf Proof of \ref{2}:} \\
Let $H$ and $G$ be two graphs which are homeomorphic with respect to a topology on $H$ generated by a subbase 
$\A$ and a subbase $\B$ for $G$. There is a lattice isomorphism between $\H$ and $\G$
and the corresponding basis elements $A \in \A$ and $B=\phi(A) \in \B$ 
have the same dimension, so that by definition, the dimension spectrum is the same. \\

{\bf Proof of \ref{3}:} \\
By definition, the graph $G$ is homotopic to the nerve graph. Since the two nerve graphs are
isomorphic, and being homotopic is an equivalence relation, the 
two homeomorphic graphs are homotopic too. \\

{\bf Remark.} The assumption of being homotopic to the nerve graph is natural: we get the nerve graph
by collapsing contractible graphs $B \in \B$ to a star graph. Each of these deformations is a 
homotopy. After having collapsed every node, we end up with the nerve graph. We only have
to makes sure for example that the nerve graph does not contain additional triangles. In
the $K_3$-free graph $C_4$ for example, we can not have three open sets in $\B$ intersecting 
each other in sets of dimension $1$.  \\

{\bf Proof of \ref{4}:} \\
Cohomology is a homotopy invariant \cite{I94} because each homotopy step is: it is straightforward
to extend cocycles and coboundaries to the extended graph and to check that the cohomology groups
do not change. 
The statement follows from the previous one. We see that the graph cohomology without
topology is the same than the cohomology with topology, which corresponds to \v{C}ech  
in the continuum. \\

{\bf Proof of \ref{5}:} \\
Because of the Euler-Poincar\'e formula, the Euler characteristic can be expressed
in cohomological terms alone $\chi(G) = \sum_{k=0}^{\infty} (-1)^k v_k = \sum_{k=0}^{\infty} (-1)^k {\rm dim}(H^k(G))$.
The result follows now from Theorem~(\ref{4}). Alternatively, this can also be checked directly 
for the combinatorial definition $\sum_{k=0}^{\infty} (-1)^k v_k$ of Euler characteristic: if we add
a new vertex $z$ over a contractible subset $Z$ of $V$, then because $S(z)=Z$ is contractible
and $B(z)$ is contractible as every unit ball is, then 
$\chi(G \cup_Z \{z\}) = \chi(G) + \chi(B(z)) - \chi(S(z)) = \chi(G) + 1-1=0$. \\

{\bf Proof of \ref{6}:} \\
Let $(G,\O)$ be the topology generated by the subbasis $\B$. 
If $G$ is not path-connected, there are two maximal subgraphs $G_1,G_2$ 
for which there is no path connecting a vertex from $G_1$ with a vertex in $G_2$ and such 
that $G_1 \cup G_2$ is $G$. 
By assumption, every edge $e \in G_1$ is contained in an open contractible set $U_e \in \B$ such
that $\B_1 = \{ U_e \; | \; e \in G_1 \}$ is a subbase of the graph $G_1$. 
and $\B_2 = \{ U_e \; | \; e \in \G_2 \}$ is a subbase of $G_2$. Both are nonempty and intersections
of $B_1 \in \B_1, B_2 \in \B_2$ are empty. 
Conversely, assume that $G$ is not connected with respect to some graph topology $(\B,\O)$.
This means that $\B$ can be split into two disjoint sets $\B_1 \cup \B_2$ for which all intersections 
$A \in \B_1$ with $B \in \B_2$ are empty. Let $G_i$ denote the subgraph 
generated by edges in $\B_i$. If there would exist a path from $G_1$ to $G_2$, then there would exist
$x \in G_1$ and a vertex $y \in G_2$ such that 
$e=(x,y)$ is an edge. By assumption, there would now be an open set $U$ containing $e$. 
But this $U$ has to belong either to $\B_1$ or to $\B_2$. This contraction shows that the 
existence of a path from $A$ to $B$ is not possible. \\

{\bf Proof of \ref{7}:} \\
We only have to verify this for a single refinement-step of an edge in which we add or remove 
an additional refinement vertex. We do not have to change the number of elements in $\B$ 
generating the topology $\O$: the new point can absorbed in each open set which contains the edge. 
{\it A refinement step in which a new vertex is put in the middle of a single edge
is a homotopy deformation.}
{\bf Proof.} Take an edge $e=(x,y)$. Do a pyramid construction over $(x,y)$ using a new vertex $z$.
Now remove the old edge. This can be done by homotopy steps
because $S(x) \cap S(y)$ is contractible. (see \cite{CYY}).  \\
For the reverse step, when removing a point, we might have to modify the topology first and
take a rougher topology. For example, lets look at the line graph $L_4$ with four vertices equipped with 
the topology generated by $\B = \{ (1,2,3),(2,3,4) \}$. This topological graph has the nerve $K_2$. 
Removing the vertex $2$ would produce $\C = \{ (1,3),(3,4) \}$ which has a disconnected nerve graph. 
But if we take the rougher topology $\B = \{ (1,2,3,4) \}$ then this becomes $\{ (1,3,4) \}$ after removing
the vertex and the homotopy reduction is continuous. \\

\begin{figure}
\scalebox{0.14}{\includegraphics{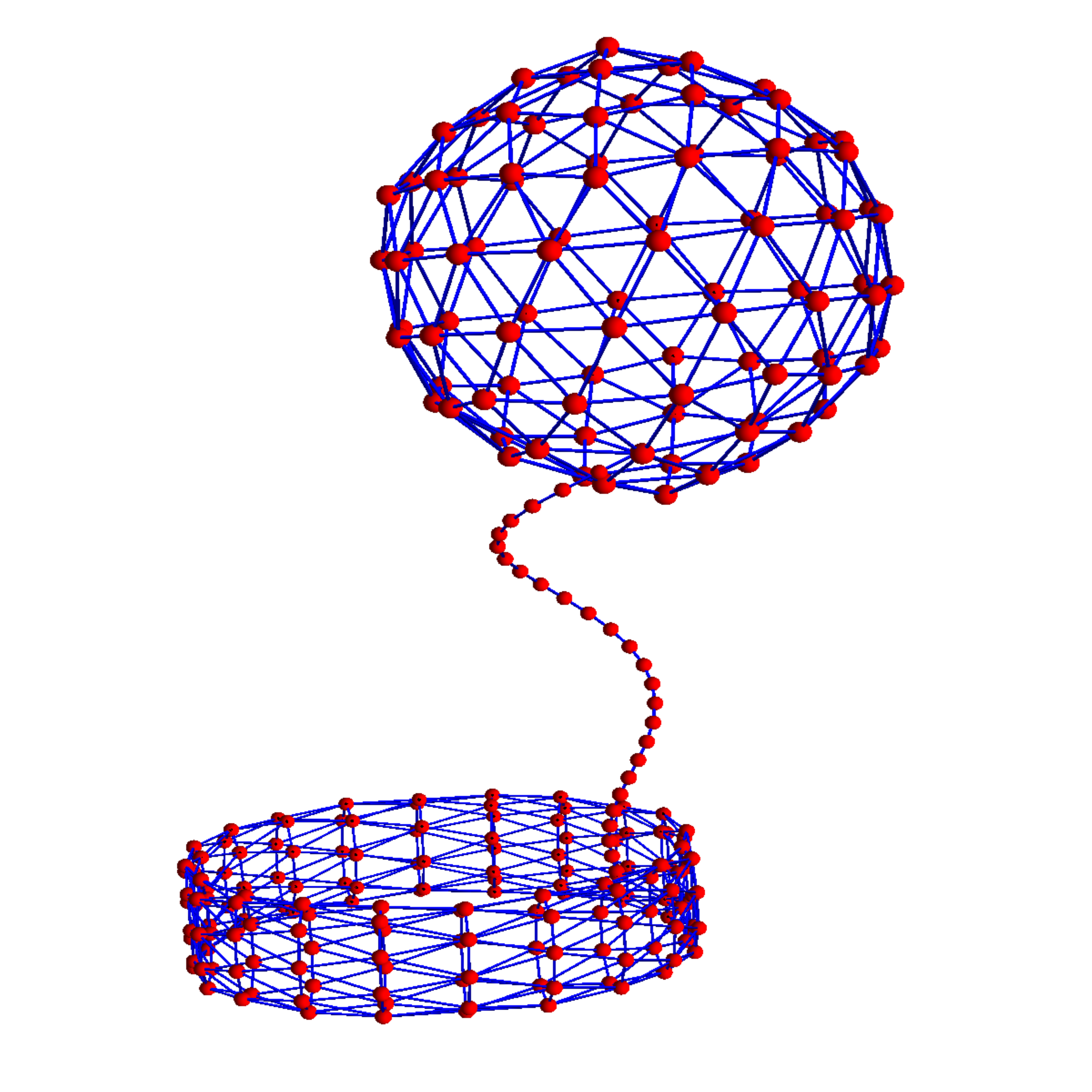}}
\scalebox{0.14}{\includegraphics{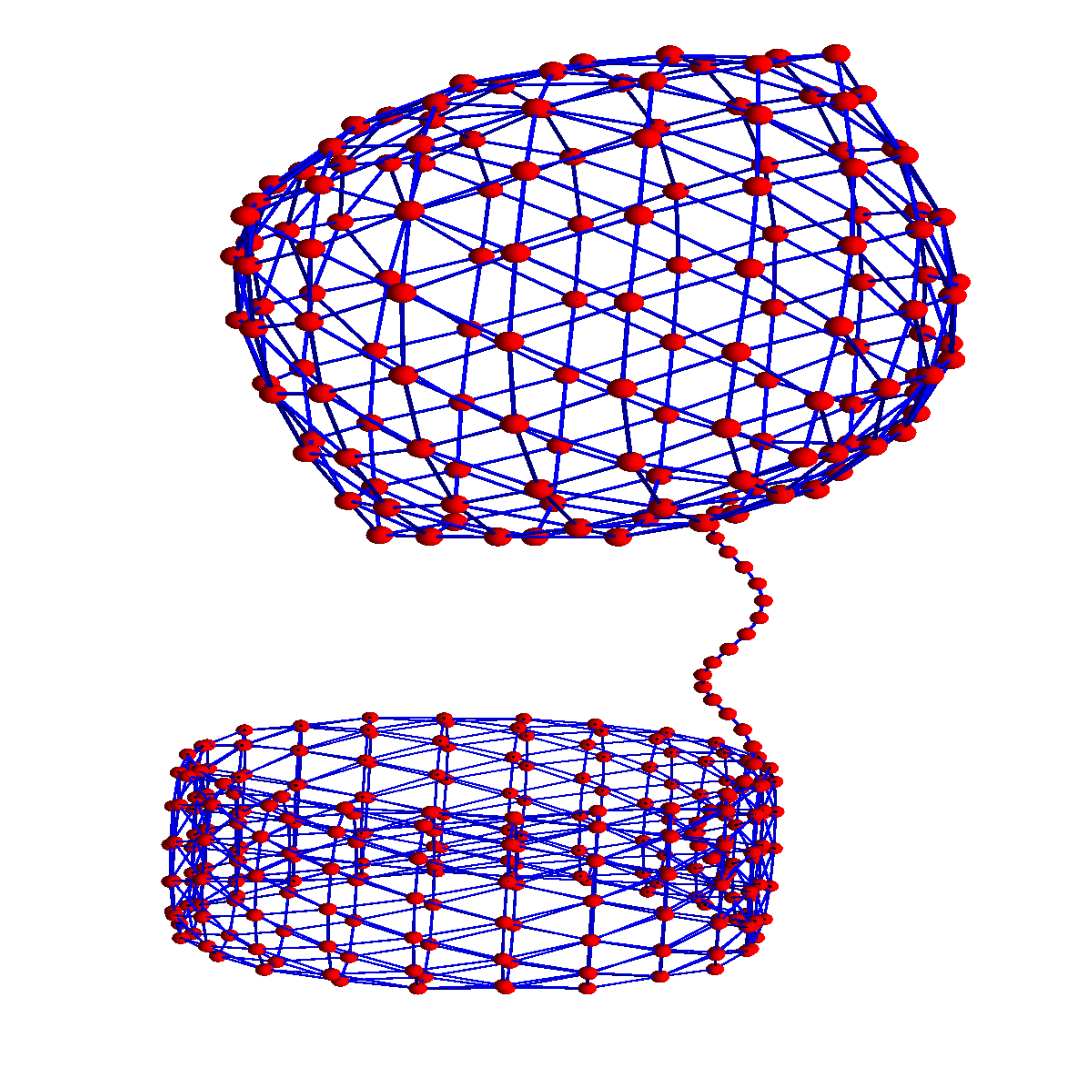}}
\caption{
Two homeomorphic graphs $H,G$. Both are discretisations of $S^2 \dot{\cup} [0,1] \dot{\cup} T^2$,
where $\dot{\cup}$ is disjoint union of two topological spaces with one point identified. 
The homeomorphism can be achieved with a sub-base consisting of 16 two-dimensional patches 
for the torus, a one-dimensional patch for the connection and 20 open balls for the sphere. 
\label{flower}
}
\end{figure}

\section{Examples}

To illustrate the notion with examples, lets introduce some more notion: 
a subgraph $K$ of $G=(V,E)$ is called {\bf dimension homogeneous} if ${\rm dim}_K(x)$ is the same for all $x \in V$. 
For example, every star graph within a graph $G$ is dimension homogeneous. 
A subgraph $K$ of $G$ is {\bf dimension-maximal} if ${\rm dim_K}(x) = {\rm dim}_G(x)$ for every $x \in V(K)$. 
A single point $K_1$ in a triangle is not dimension maximal because it has dimension $0$ by itself but dimension $2$
as a point in the triangle. A triangle is dimension maximal in an octahedron because both in the triangle as well
as in the octahedron, every point has dimension $2$. Since it can be difficult to construct for any finite simple
graph a dimension maximal basis, we don't require it in the definition.  \\

{\bf 1)} Given any finite simple graph, we can take the subbasis $\B$ of all star graphs centered at 
vertices together with the sets $\{x_i\}$ of all isolated vertices $x_i$.
All the properties for a graph topology are satisfied: the elements are contractible, the intersection with each other 
are $K_2$ graphs which are contractible and the dimension of each $B \in \B$ is $1$ and the dimension of
the intersection of two neighboring basis elements is $1$. Two star graphs of points of distance $2$
have a vertex in the intersection but the dimension assumption prevents this from counting as a link
in the nerve graph. We call the topology generated by this subbase the {\bf finest topology} on a graph.
It has the property that the nerve graph is the graph itself. Two graphs dressed with the finest topology are 
homeomorphic if and only if they are graph isomorphic. The fine topology on a graph reflects what is often
understood with a graph, a one-dimensional structure. As pointed out before, this is not what we consider a good
topology in general. Graphs are more universal and carry topologies which make them behave like higher dimensional 
spaces in the continuum. \\

{\bf 2)} Any contractible graph carries the {\bf indiscrete topology} = trivial topology 
generated by a cover $\B = \{ V \}$ which consists of one element only. 
The nerve graph of this topology is the graph $K_1$ and the covering dimension is zero. 
Two contractible graphs of the same dimension equipped with the indiscrete topology are 
homeomorphic as topological graphs. Any two trees are homeomorphic with respect to the indiscrete
topology. By the wqy, trees can be characterized as {\bf uniformly one-dimensional, contractible} graphs
because one-dimensional graphs are determined by the classical notion of homeomorphism in which the 
genus $g$ is the only invariant. \\

{\bf 3)} The {\bf cycle graph} $C_6$ has a topology with the $6$ elements 
$\B = \{$ $(1,2,3)$, $(2,3,4)$, $(3,4,5)$, $(4,5,6)$, $(5,6,1)$, $(6,1,2)$ $\} \;$.
We can not take $\B_0 = \{$ $(1,2,3,4)$, $(3,4,5,6)$, $(5,6,1,2)$ $\}$ because its nerve is not homotopic.
We can not take $\B_1 = \{$ $(1,2)$, $(2,3)$, $(3,4)$, $(4,5)$, $(5,6)$, $(6,1)$ $\}$ because the nerve
is zero-dimensional, not homotopic and also the intersection dimension assumption fails. 
The basis elements in $\B$ are dimension maximal and dimension homogeneous.
Two cyclic graphs $C_n,C_m$ are homeomorphic if $n,m \geq 4$:
to illustrate this with $C_4$ and $C_5$, take
$\B = \{ \{1,2,3 \}, \{2,3,4 \}$, $\{3,4,1 \}, \{1,2,3 \} \}$ for $C_4$ and 
$\C = \{ \{1,2,3,4 \}, \{3,4,5\}$, $\{4,5,1 \}, \{5,1,2 \} \}$ for $C_5$. For $\B$ on $C_4$,
there are sets which intersect in a non-contractible way but we have assumed this not to count. The nerve
graph of the topology generated by $\B$ is $C_4$ itself. \\

{\bf 4)} Make a pyramid construction over an edge $(1,2)$ of a {\bf cycle graph} $C_4$. This is a homotopy step. 
The new graph $G$ has now a 5'th vertex $5$ and the new dimension is $22/15$ like the bull graph. 
The topology generated by $\B = \{(1,2,5)$, $(5,2,3)$, $(2,3,4)$, $(4,1,5)$  $\}$
is optimal. Its dimension spectrum is $\{ 2,1,1,1 \}$ and the topological dimension is the average $5/4$.
This example shows that the topological dimension is not the same than the inductive dimension. The 
topological dimension depends on the topology.  The example also illustrates that $C_4$ and $G$ are not 
homeomorphic even so they are homotopic. 
Finally, lets look at a subbasis $\B_1 = \{$ $(1,2,3)$, $(2,3,4)$, $(3,4,1,5)$, $(1,5,2)$ $\}$. It produces a topology
but not an optimal topology. \\

{\bf 5)} For an {\bf octahedron}, we can take $\B$ as the set of unit balls. The two unit balls of antipodes intersect
in a circular graph but the nerve graph is the octahedron itself because the dimension assumption prevents antipodal 
points to be connected. The set $\B$ does indeed define a topology. 
Also the {\bf icosahedron}  has a 
topology generated by the $20$ unit balls of radius $1$.  More generally, for any fine enough 
triangularization of the two-dimensional sphere, we can take for $\B$ a set of unit balls. 
When taking the set of balls of radius $2$ as the cover for the octahedron we see that the 
Icosahedron and Octahedron are homeomorphic with respect to natural optimal topologies. \\

{\bf 6)} A {\bf sun graph} $G=S_{1,1,1}$ over a triangle has a topology which consists of 4 sets. We can
take the triangle $(1,2,3)$ and the sets $(1,2,3,4)$, $(1,2,3,5)$, $(2,2,3,6)$. 
The nerve graph is the star graph $S_3$. This sub-basis is dimension maximal but 
not dimension homogeneous. Since that sun graph is contractible, we can also take the indiscrete
topology on $G$.  \\

{\bf 7)}  Take a $a$-dimensional simplex and connect it with a line graph with $n$ vertices 
$b$-dimensional simplex. The graph has $a+b+n$ vertices with $a$ vertices of dimension $a$
and $b$ vertices of dimension $b$ and $n-2$ vertices of dimension $1$ and
one vertex of dimension $1+a(a-1)/(a+1)$ and one vertex of dimension $1+b(b-1)/(b+1)$.
The dimension of such a {\bf dumbbell graph} $G_{a,b,n}$ therefore is
${\rm dim}(G_{a,b,n}) = (a^2 + b^2  + n + a(a-1)/(a+1) + b(b-1)/(b+1))/(a+b+n)$. 
Since ${\rm dim} G_{3,4,3}={\rm dim} G_{3,7,15}=319/100$, 
these two graphs are homeomorphic with the indiscrete topology. 
We can not find topologies on on these two graphs which are dimension homogeneous and
for which the graphs are homogeneous.  \\

{\bf 8)} Any {\bf sun graph} $S_{a_1, \dots, a_n}$ with $n \geq 4$ obtained by taking a cyclic graph and attaching
line graphs of length $a_i$ at the vertex $x_i$ is strongly homeomorphic to $C_n$. 
The reason is that every point of such a graph has dimension $1$. The topology is illustrated in 
Figure~(\ref{example1}) and is optimal. Since any two graphs $C_n$ are homeomorphic, all sun graphs are
homeomorphic with respect to this topology. The equivalence class is the topological circle. We can 
also find other topologies, for which the nerve graph is again a sun graph. This is in particular the
case for the discrete topology generated by star graphs attached to vertices. \\

{\bf 9)} Any two {\bf wheel graphs} $W_n$ with $n \geq 4$ are homeomorphic with respect to the 
trivial indiscrete topology. More generally, any two contractible graphs for which every point is two-dimensional 
are homeomorphic. This includes the triangle $K_3$. The equivalence class is the topological disc. 
More natural topologies are the topologies generated by open balls in $W_n$. In that case the nerve
graph is again $W_n$. \\

\begin{figure}[H]
\scalebox{0.14}{\includegraphics{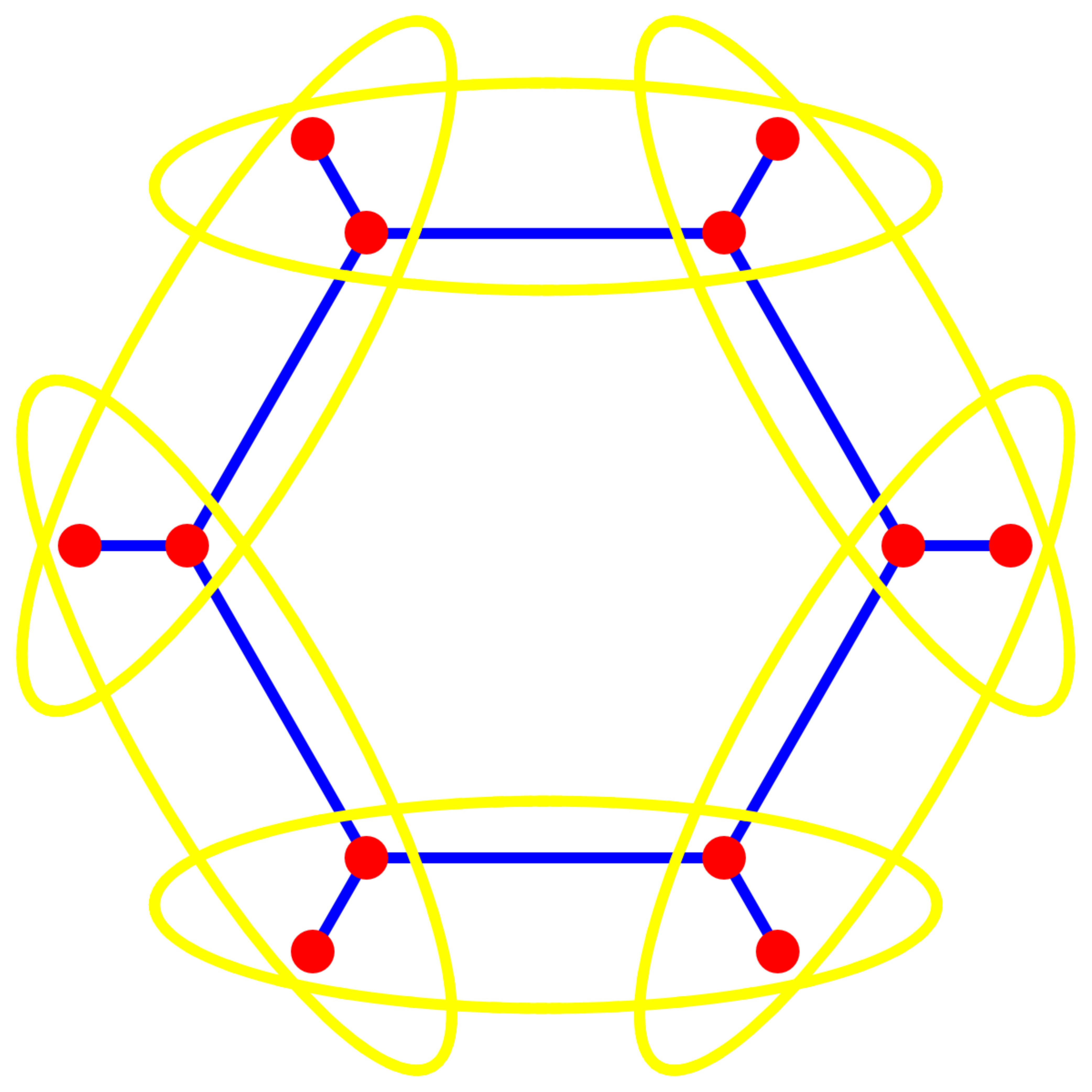}}
\scalebox{0.14}{\includegraphics{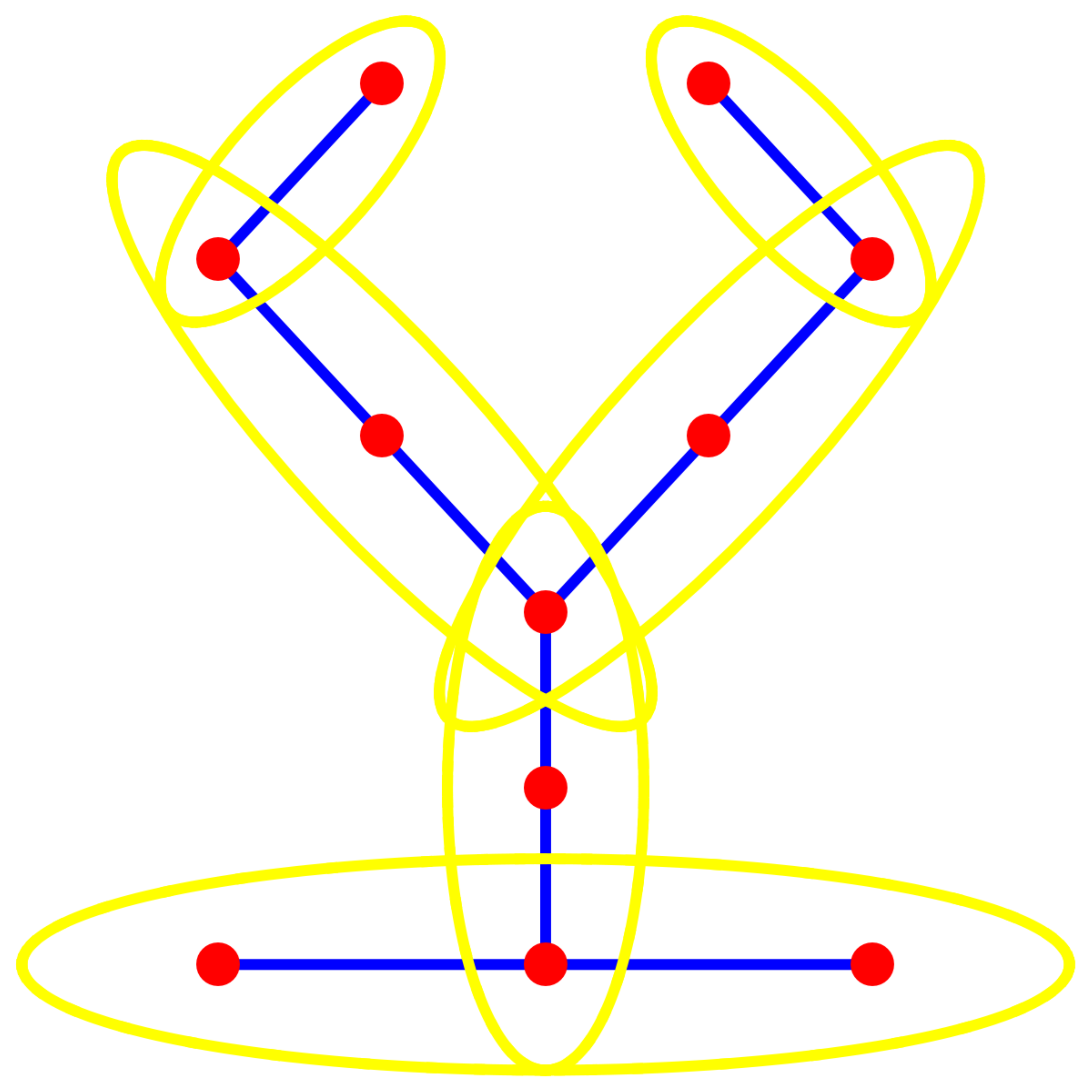}}
\caption{
The left graph is a one-dimensional sun graph $G=S_{1,1,1,1,1,1}$ with a topology $\B$ 
rendering it to be homeomorphic to $C_{6}$, where the later is equipped with the topology generated by 
unit balls $\B$. The right graph is an example of a tree. There are 6 sets drawn, but
this is not a sub basis because intersections are zero-dimensional, but we can merge them to
get a subbase with three elements $B_1,B_2,B_3$.
The two graphs are one-dimensional. There is no topology which makes the second graph $G_2$ 
homeomorphic to the first graph $G_1$ because $\chi(G_1)=0$ and $\chi(G_2)=1$.
\label{example1}
}
\end{figure}

{\bf 10)} Start with any connected finite simple graph $H=(V,E)$ and subdivide every edge with a vertex.
The new graph $G$ with $|V|+|E|$ vertices and $2|E|$ edges is uniformly one-dimensional. 
The Euler characteristic is $\chi(G) = |V|-|E| = 1-b_1$. If it is simply connected ($b_1=0$), then it
is a tree. Any two trees are homeomorphic: because the dimension is uniform $1$, we can go with 
the indiscrete topology. The indiscrete topology on a tree is too weak however. Better and more
natural is the topology generated by a subbase $\B$ consisting of star graphs. With respect to 
this topology, two trees are homeomorphic if and only if they are $1$-homeomorphic. \\

{\bf 11)} Take an octahedron and connect two opposite vertices $a,b$. This new graph is a contractible 
three-dimensional graph of Euler characteristic $1$ ($=6-13+12-4$). The intersection of 
two unit balls $B_1(a),B_1(b)$ is the graph itself. The graph has the indiscrete topology
as an optimal topology. \\

{\bf 12)} Any two connected {\bf trees} are strongly equivalent: since they are contractible and uniformly 
of dimension $1$, we can chose the fine topology generated by star graphs centered at vertices.
This is an optimal topology. It is also dimension faithful: ts nerve graph 
is the same tree homeomorphic to $G$. By extending the paths at the star graphs, we can get
topologies for trees 1-homeomorphic to $G$. For trees, the topology generated by unit 
balls of radius $1$ is the discrete topology. \\

{\bf 13)} Lets look at some smaller concrete graphs. Among the 24 graphs under consideration 
there are $4$ graphs with Euler characteristic $\chi=0$: the {\bf cycle}, the {\bf hole}, 
the {\bf house} and the {\bf sun}.
The cycle and sun graph are homeomorphic with optimal topologies.
There is no way to have more relations among those graphs because of dimension constraints. 
The house has some $2$-dimensional component and the hole is uniformly $2$-dimensional. 
Then there are graphs which by Euler characteristic alone are topologically distinguished
from the others: the {\bf prism} with $\chi=2$ is a discrete sphere, the {\bf utility graph} 
with $\chi=-3$, the {\bf snub cube} with $\chi=-4$, the {\bf Petersen graph} with $\chi=-5$, 
the {\bf dihedral graph} with $\chi=-6$ and the {\bf snub octahedron} with $\chi=-10$. 
Then, there are two graphs which by dimension alone are distinguished: 
the complete {\bf hyper-tetrahedron} $K_5$ is uniformly $4$-dimensional and not equivalent
to anything else. The {\bf tetrahedron} is uniformly $3$-dimensional and not
equivalent to any thing else. The {\bf lollipop} has a $3$-dimensional and $1$-dimensional
component and is distinguished from anything else. 
The {\bf kite}, the {\bf gem}, the {\bf gate}, the {\bf wheel}, and {\bf Hex} 
all are homeomorphic and form the equivalence class of a $2$-dimensional topological disc. 
This is also true for the {\bf fly}, even so in a bit unnatural way: we have to take the weak 
topology with one set, the unit ball of the center. This is not a geometric graph since the unit sphere
of the central point has Euler characteristic 2 and not 0 as demanded for the interior
of $2$-dimensional geometric graphs. It is also not a geometric graph with boundary: the
later class has at every point a sphere of Euler characteristic $0$ (interior) or $1$ 
(boundary). The {\bf fork} and {\bf star} are homeomorphic with respect to a weak optimal topology
(the discrete topology). What remains is the {\bf cricket}, the {\bf dart} and the {\bf bull}. 
They all have $2-$ and $1$-dimensional components. 
There are weak but optimal topologies which render these three graphs homeomorphic. 

\parbox{15cm}{
\parbox{3.6cm}{\scalebox{0.061}{\includegraphics{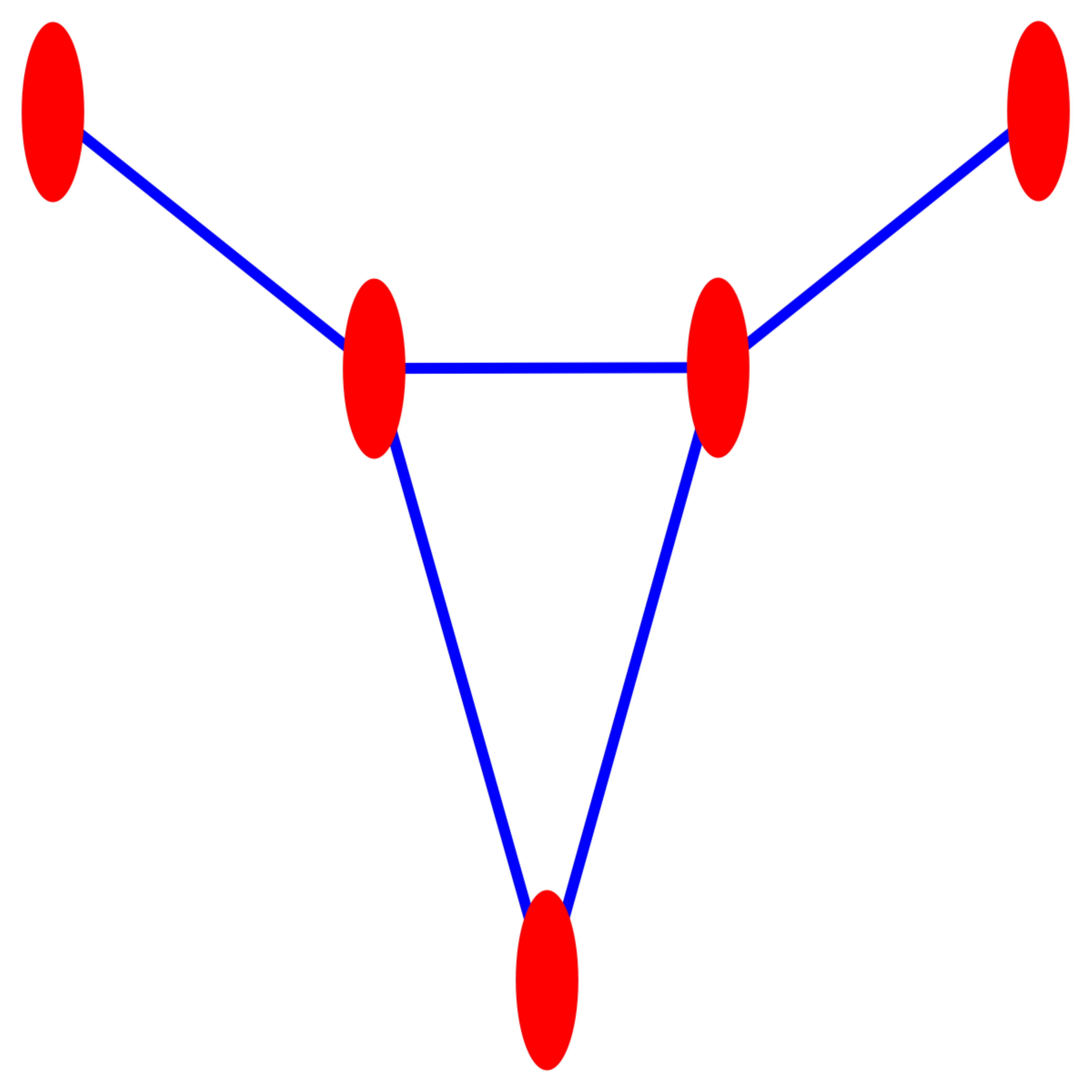}}}
\parbox{3.6cm}{\scalebox{0.061}{\includegraphics{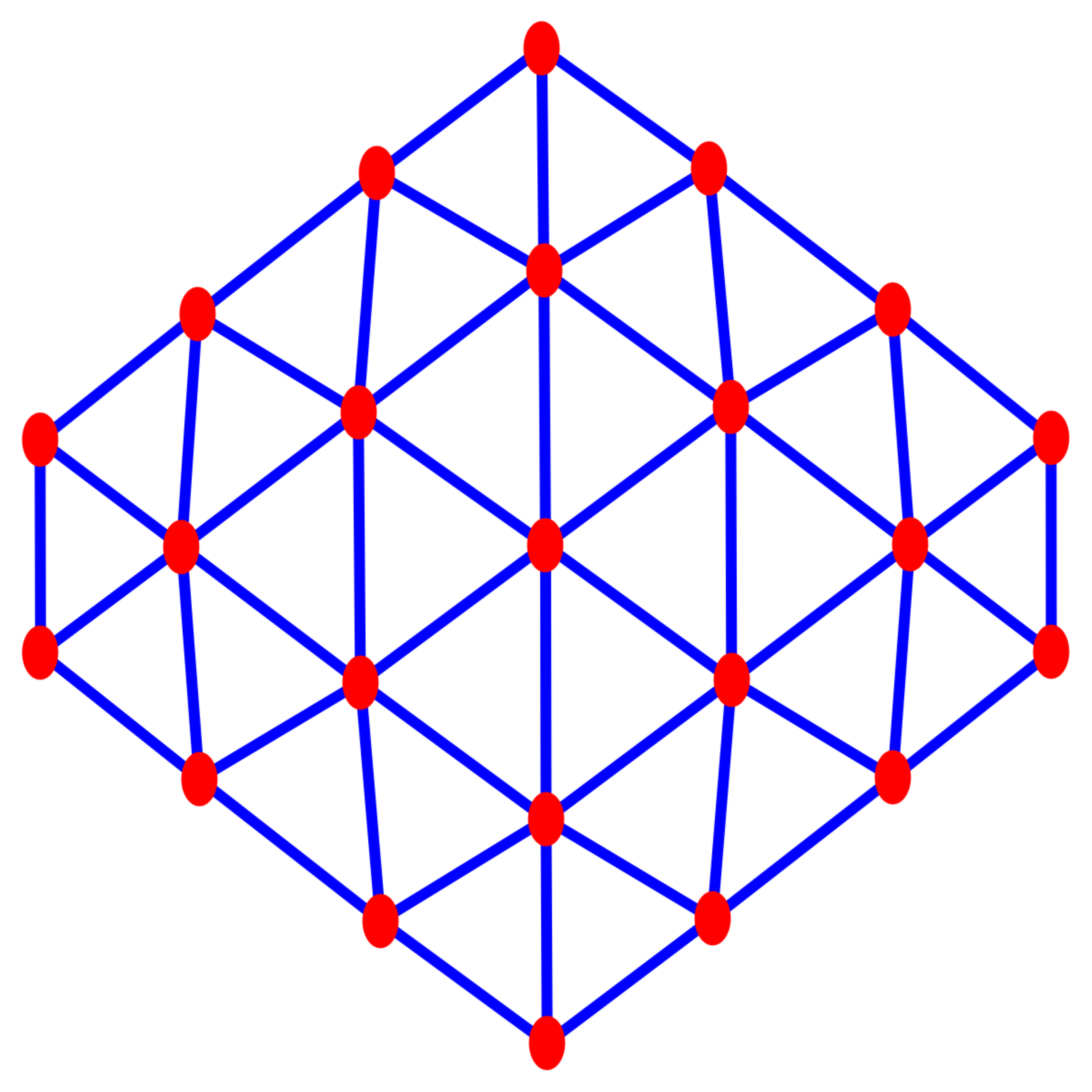}}}
\parbox{3.6cm}{\scalebox{0.061}{\includegraphics{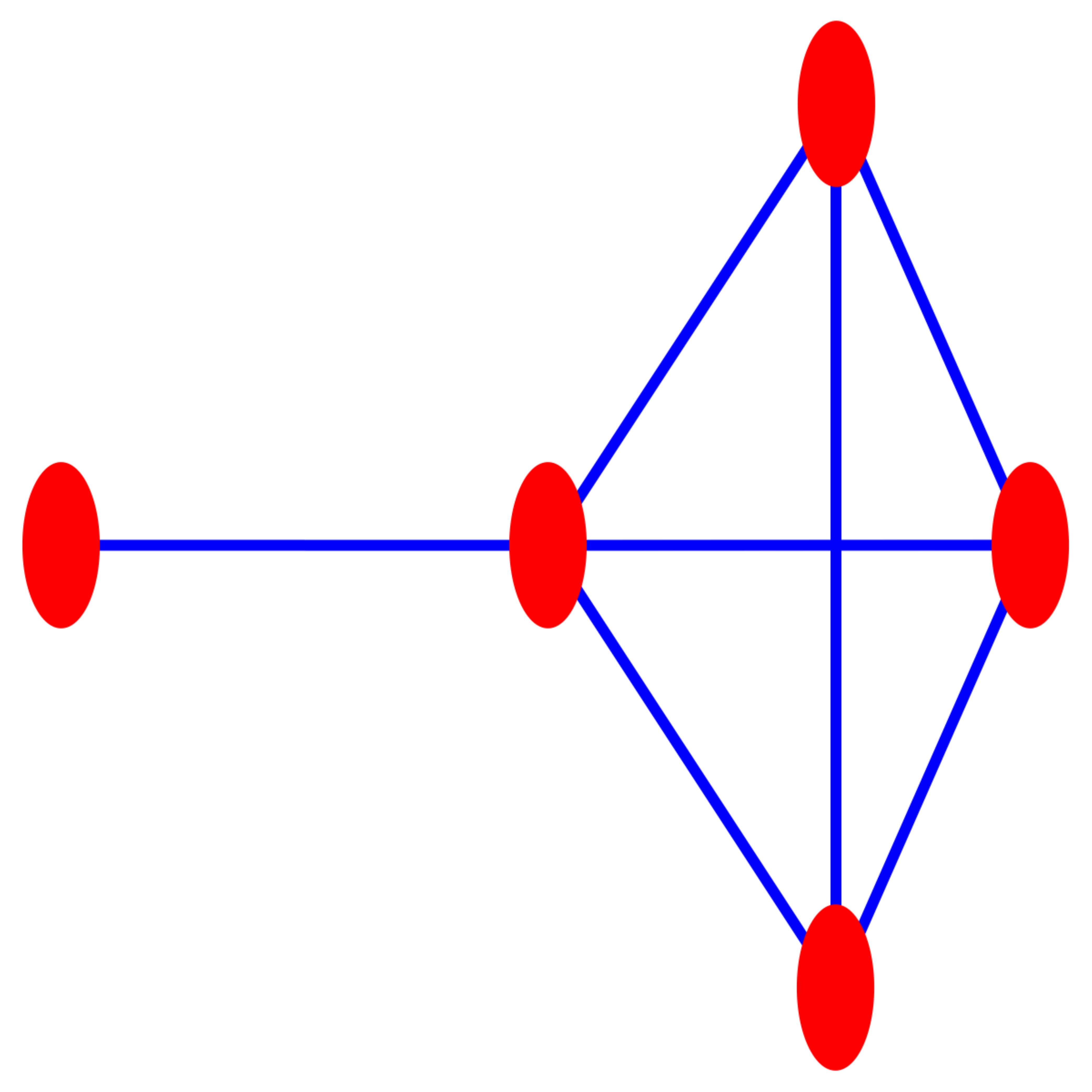}}}
\parbox{3.6cm}{\scalebox{0.061}{\includegraphics{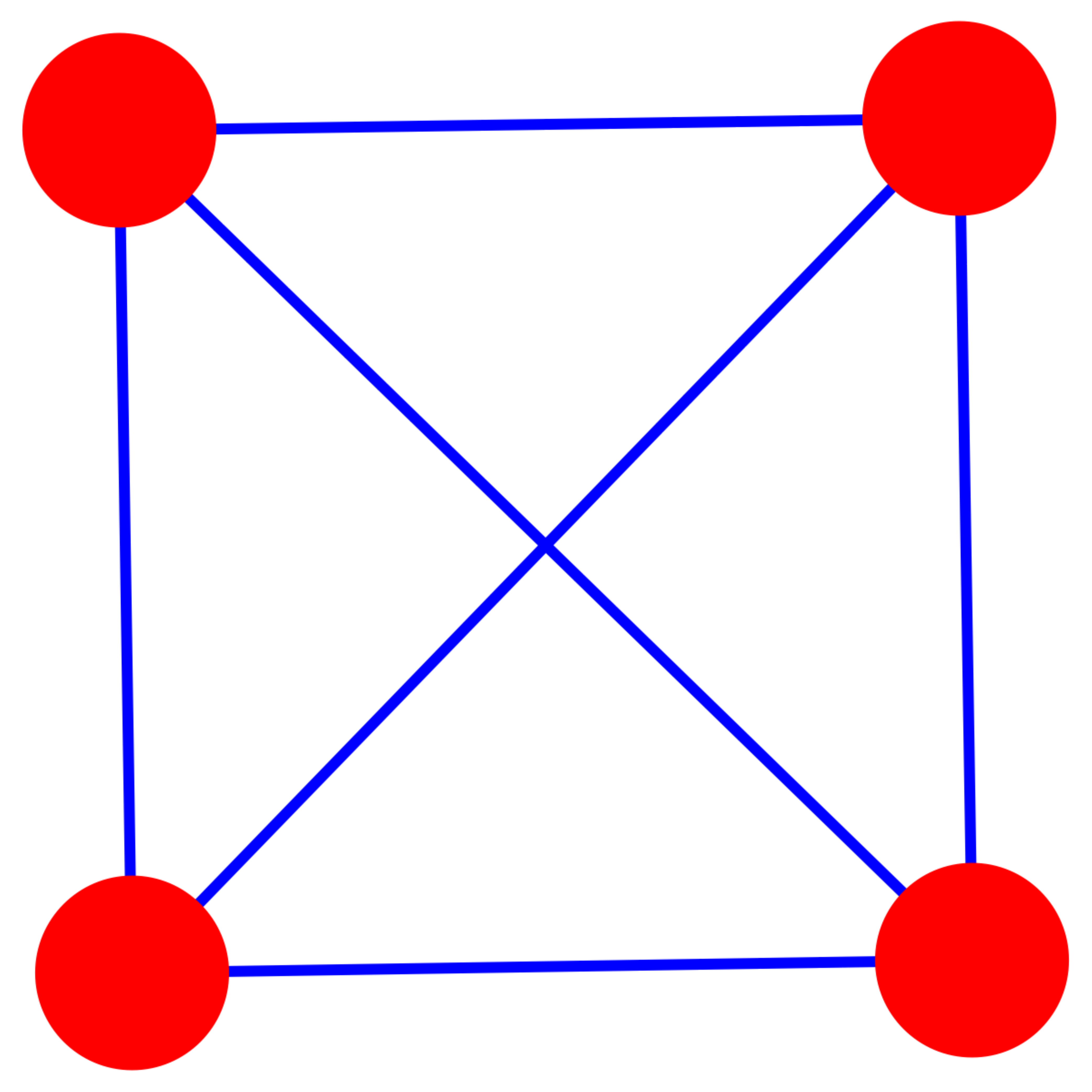}}} }
\parbox{15cm}{
\parbox{3.6cm}{ 1) Bull} \parbox{3.6cm}{ 2) Hex region} \parbox{3.6cm}{ 3) Lollipop} \parbox{3.6cm}{ 4) Tetrahedron} }
\parbox{15cm}{
\parbox{3.6cm}{\scalebox{0.061}{\includegraphics{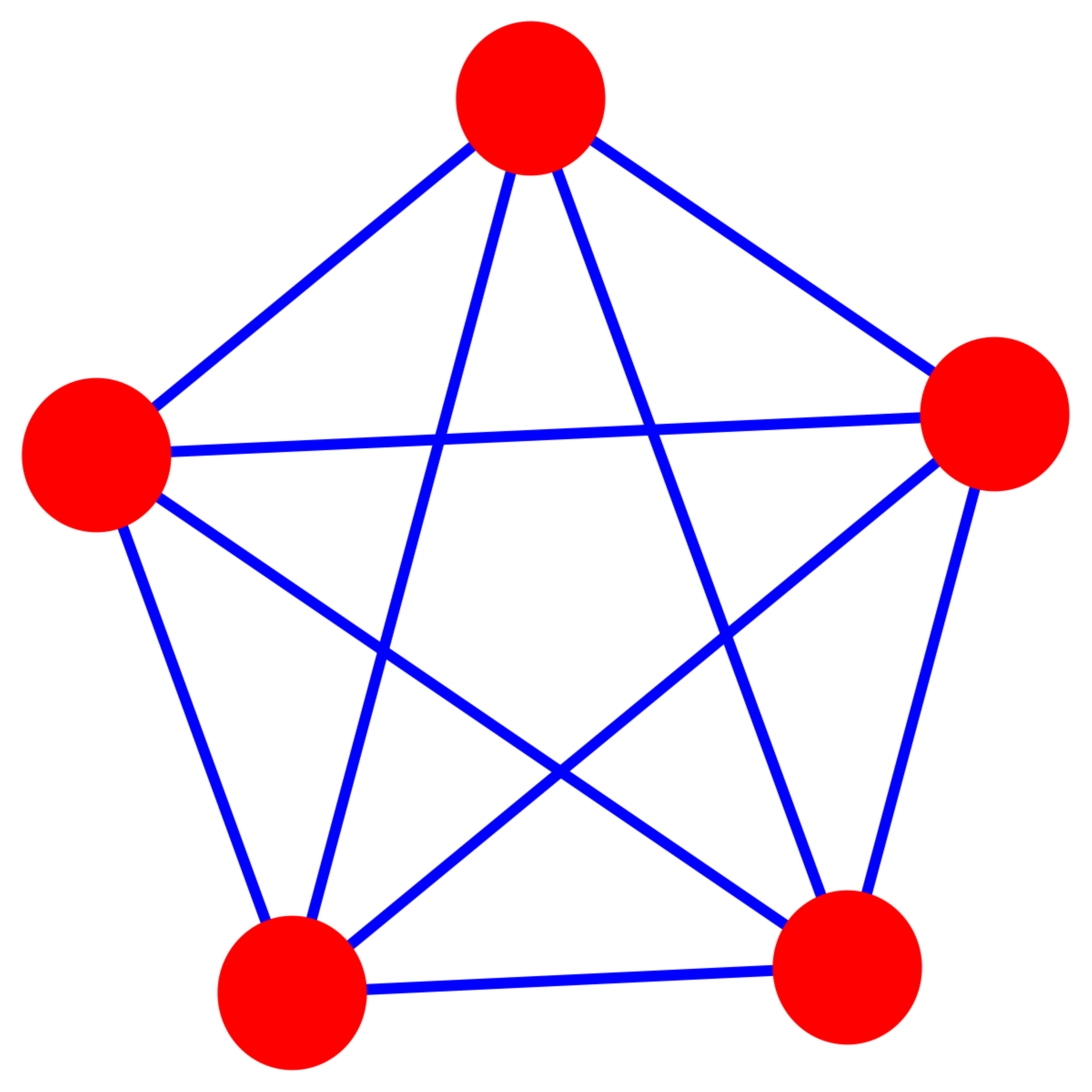}}}
\parbox{3.6cm}{\scalebox{0.061}{\includegraphics{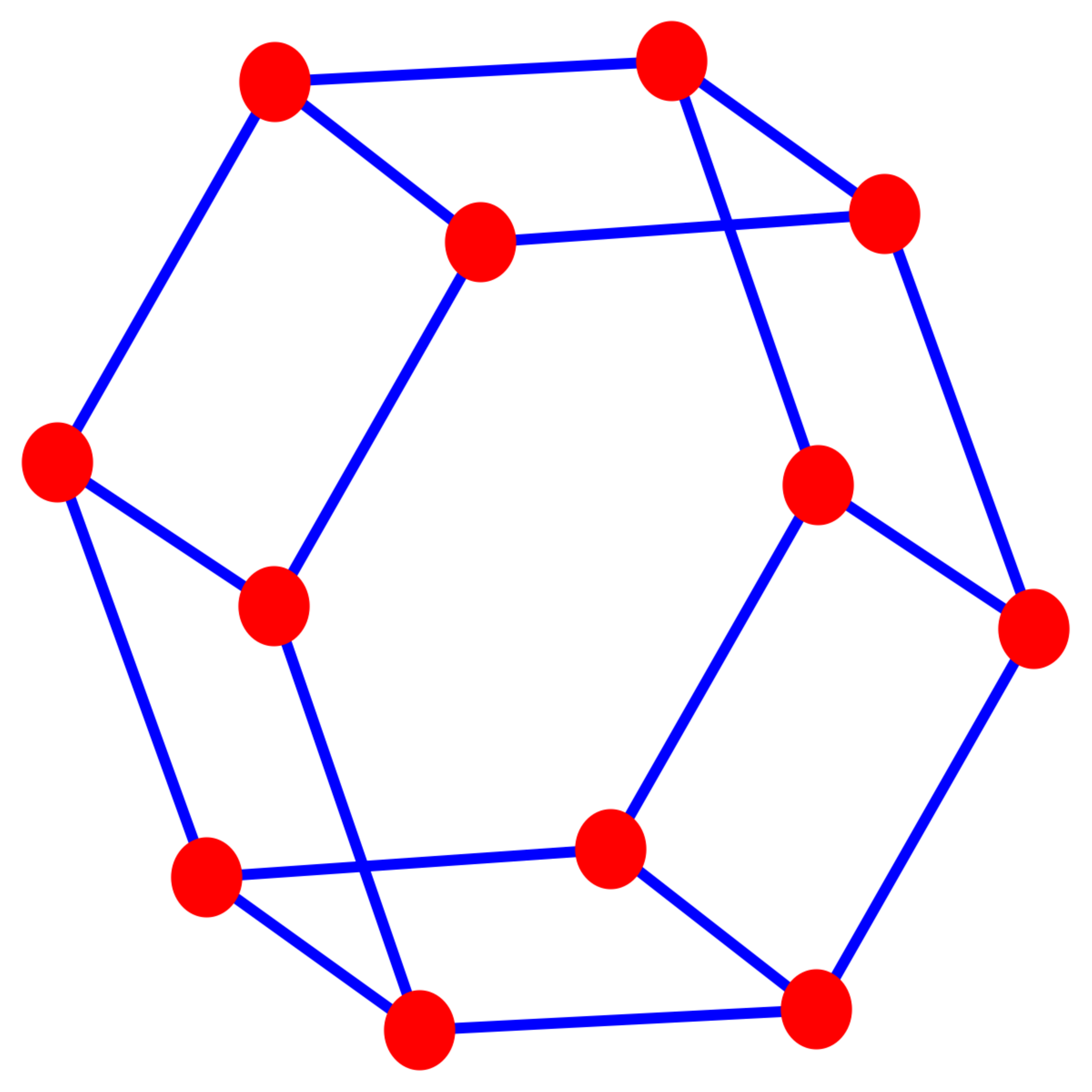}}}
\parbox{3.6cm}{\scalebox{0.061}{\includegraphics{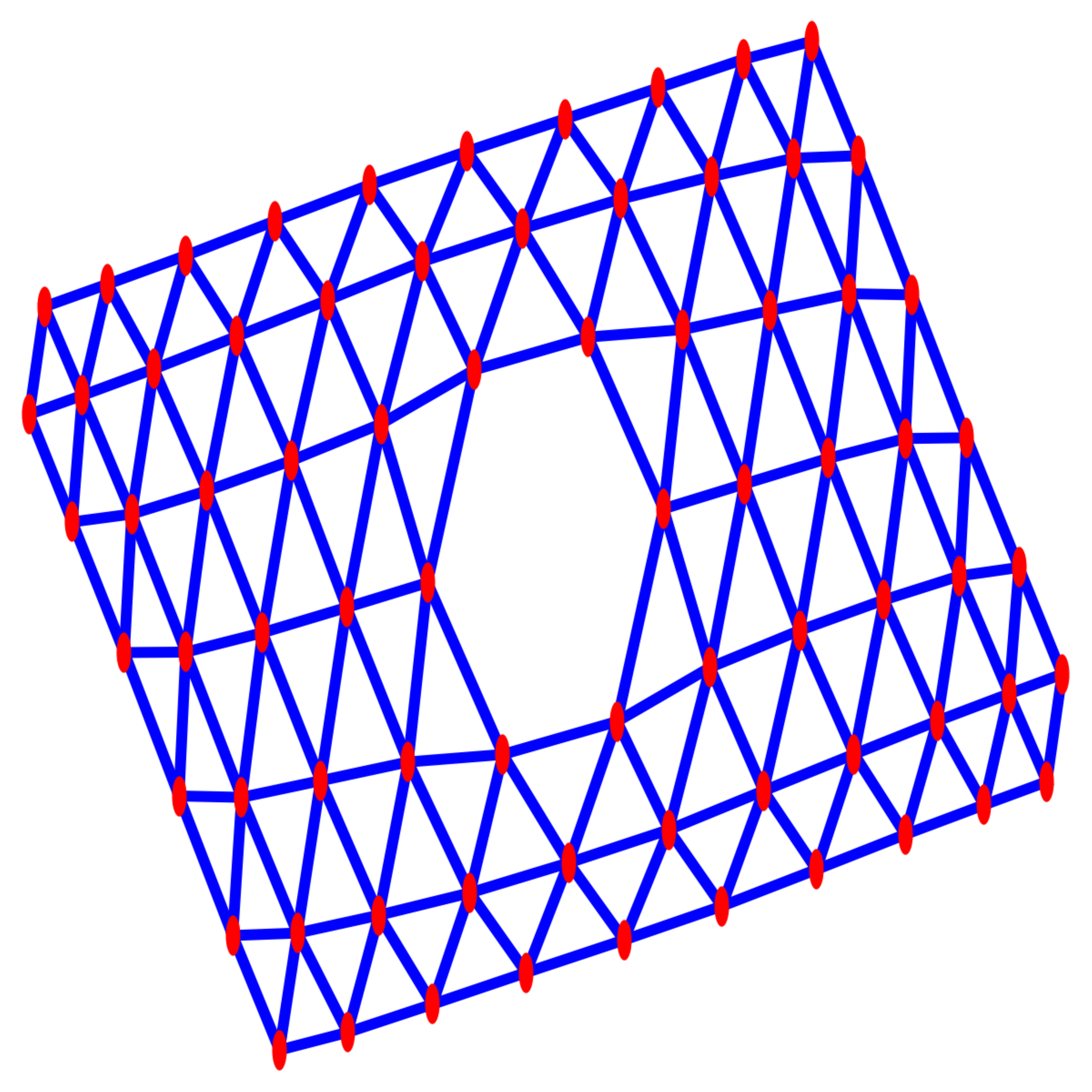}}}
\parbox{3.6cm}{\scalebox{0.061}{\includegraphics{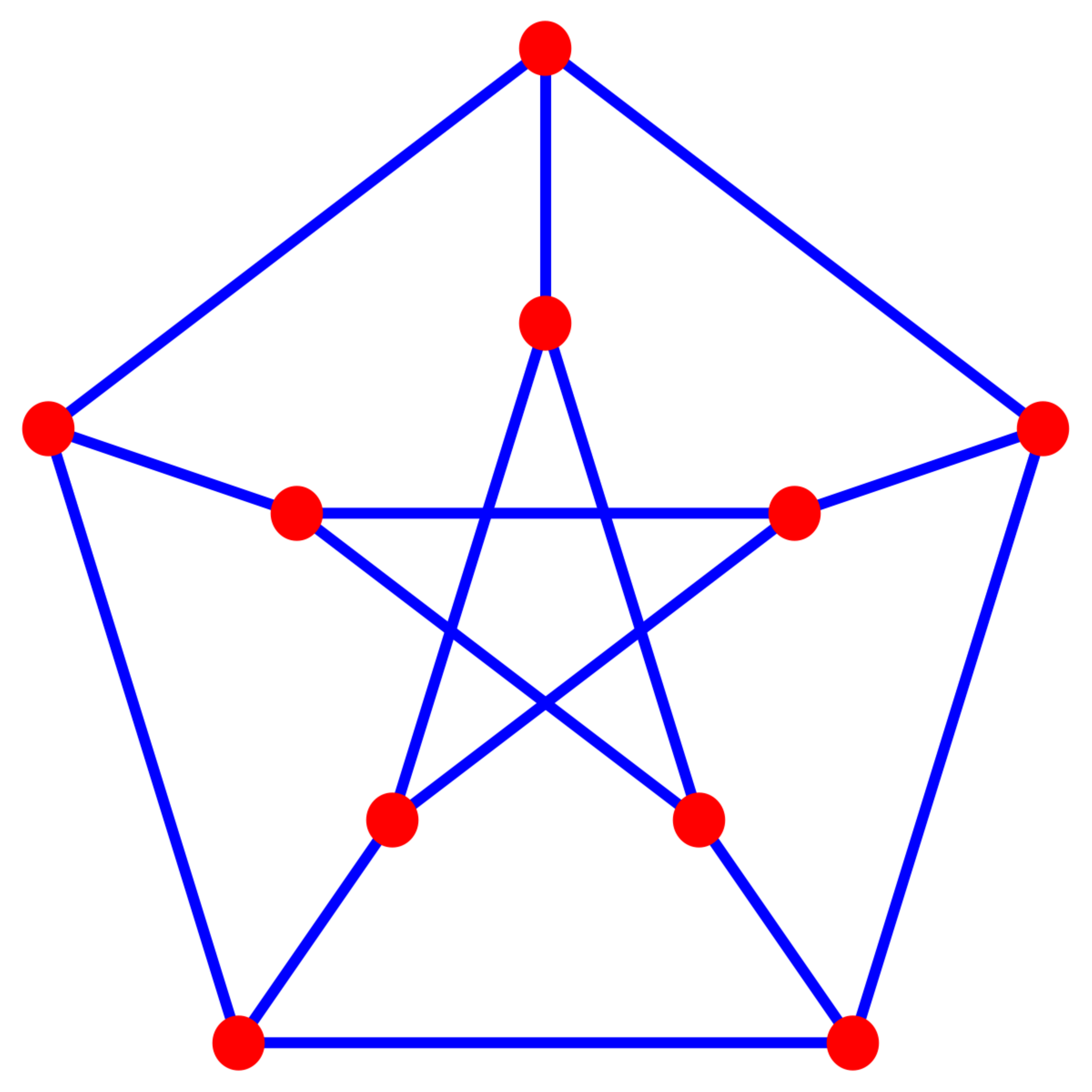}}} }
\parbox{15cm}{
\parbox{3.6cm}{ 5) Complete}      \parbox{3.6cm}{ 6) Dihedral} \parbox{3.6cm}{ 7) Hole}  \parbox{3.6cm}{ 8) Petersen}     }
\parbox{15cm}{
\parbox{3.6cm}{\scalebox{0.061}{\includegraphics{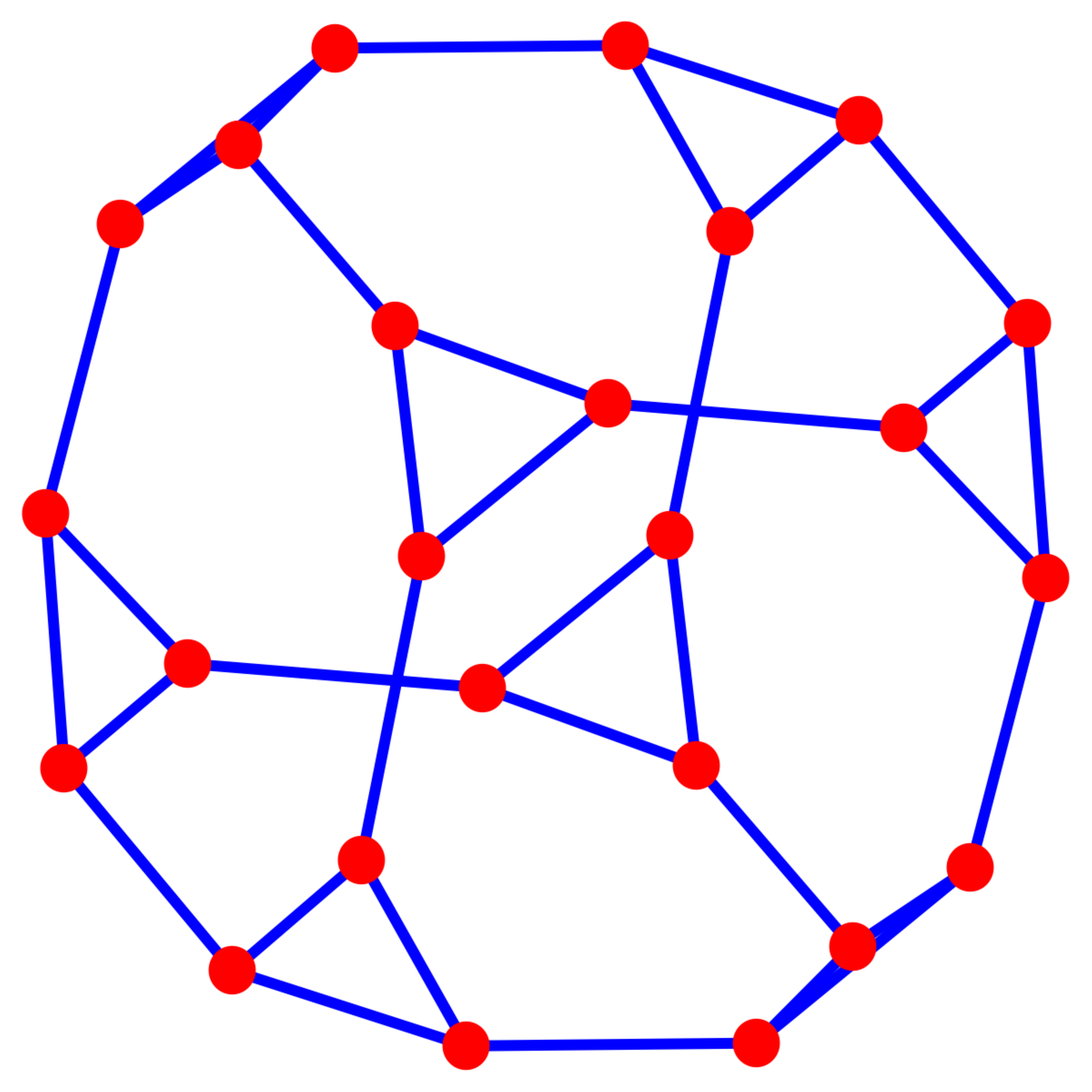}}}
\parbox{3.6cm}{\scalebox{0.061}{\includegraphics{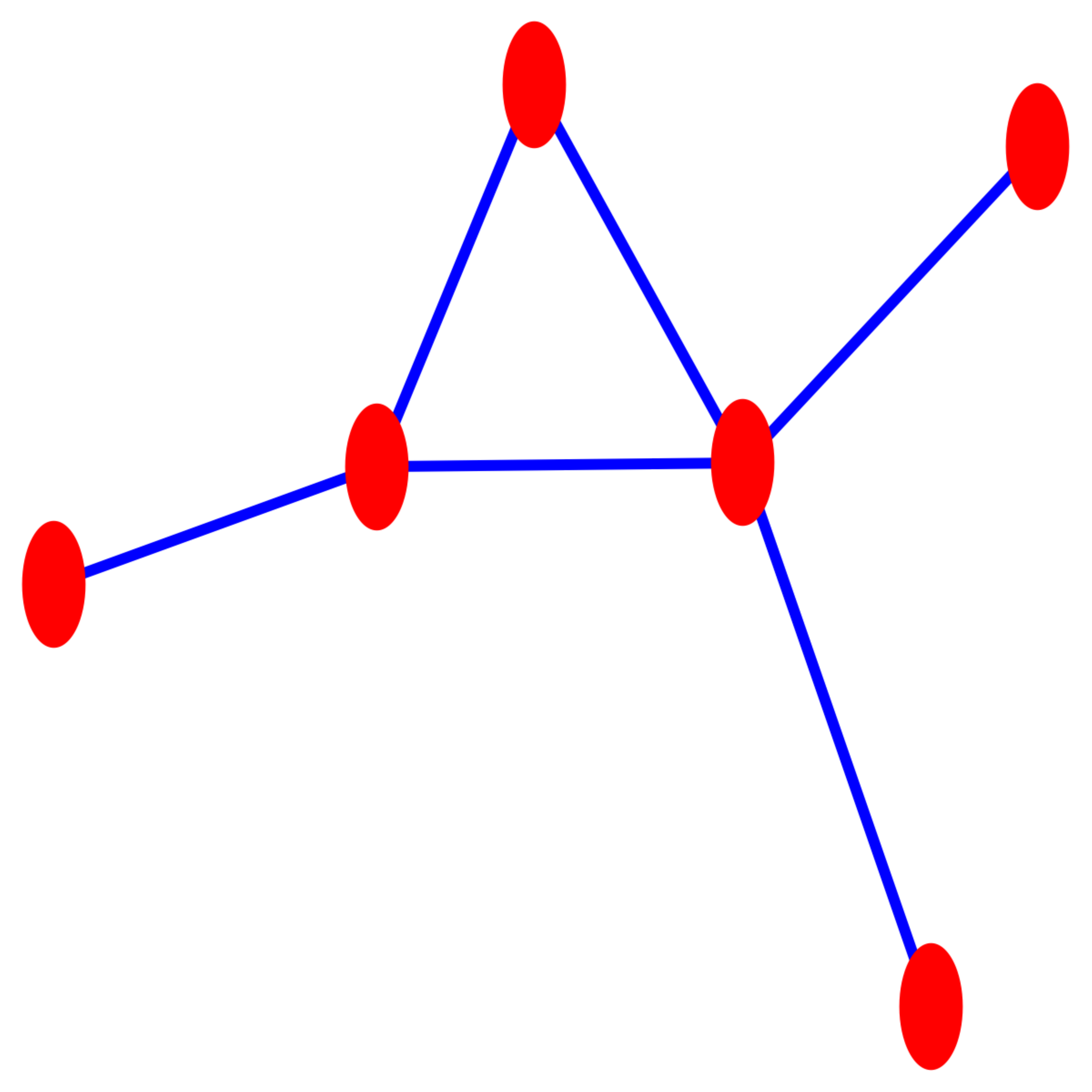}}}
\parbox{3.6cm}{\scalebox{0.061}{\includegraphics{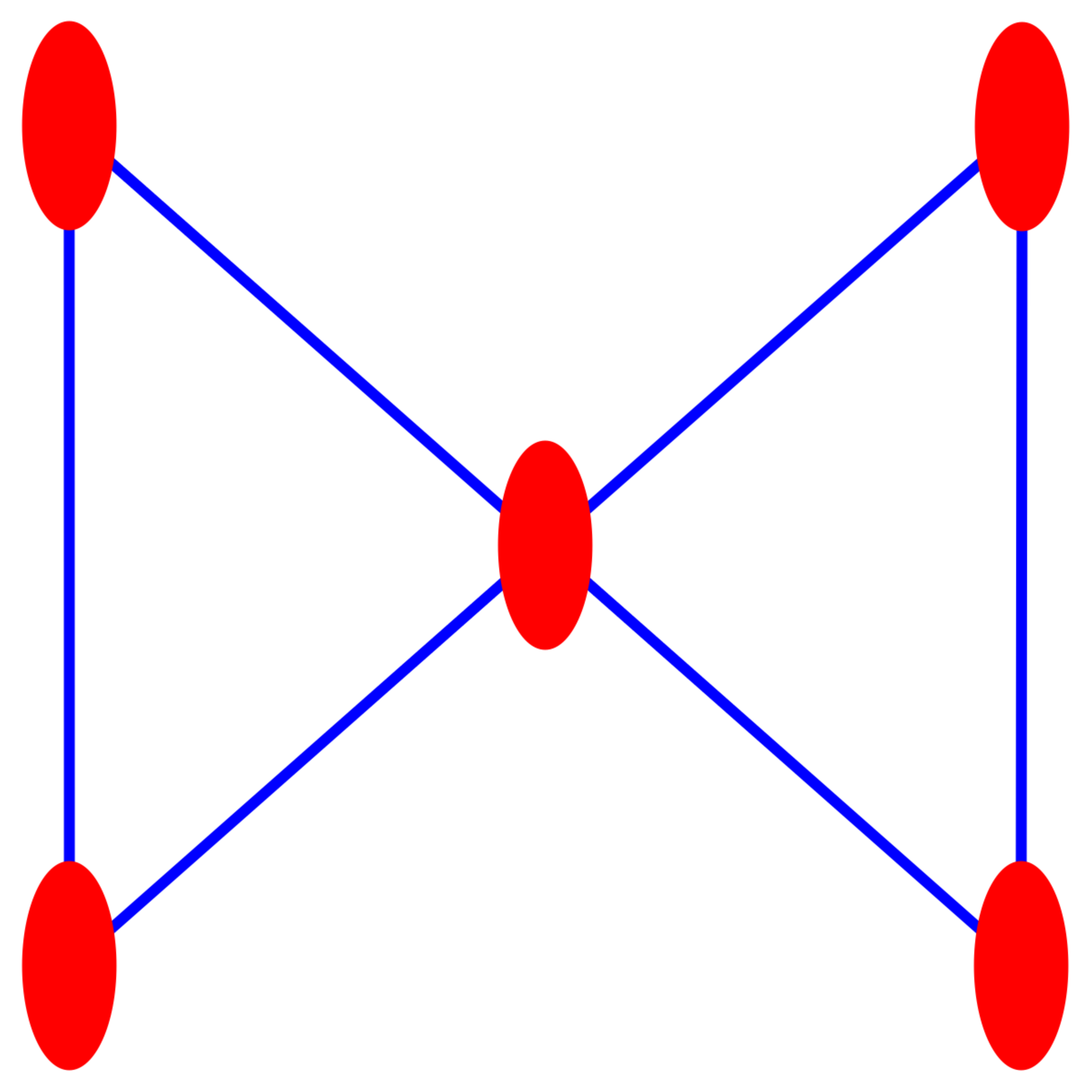}}}
\parbox{3.6cm}{\scalebox{0.061}{\includegraphics{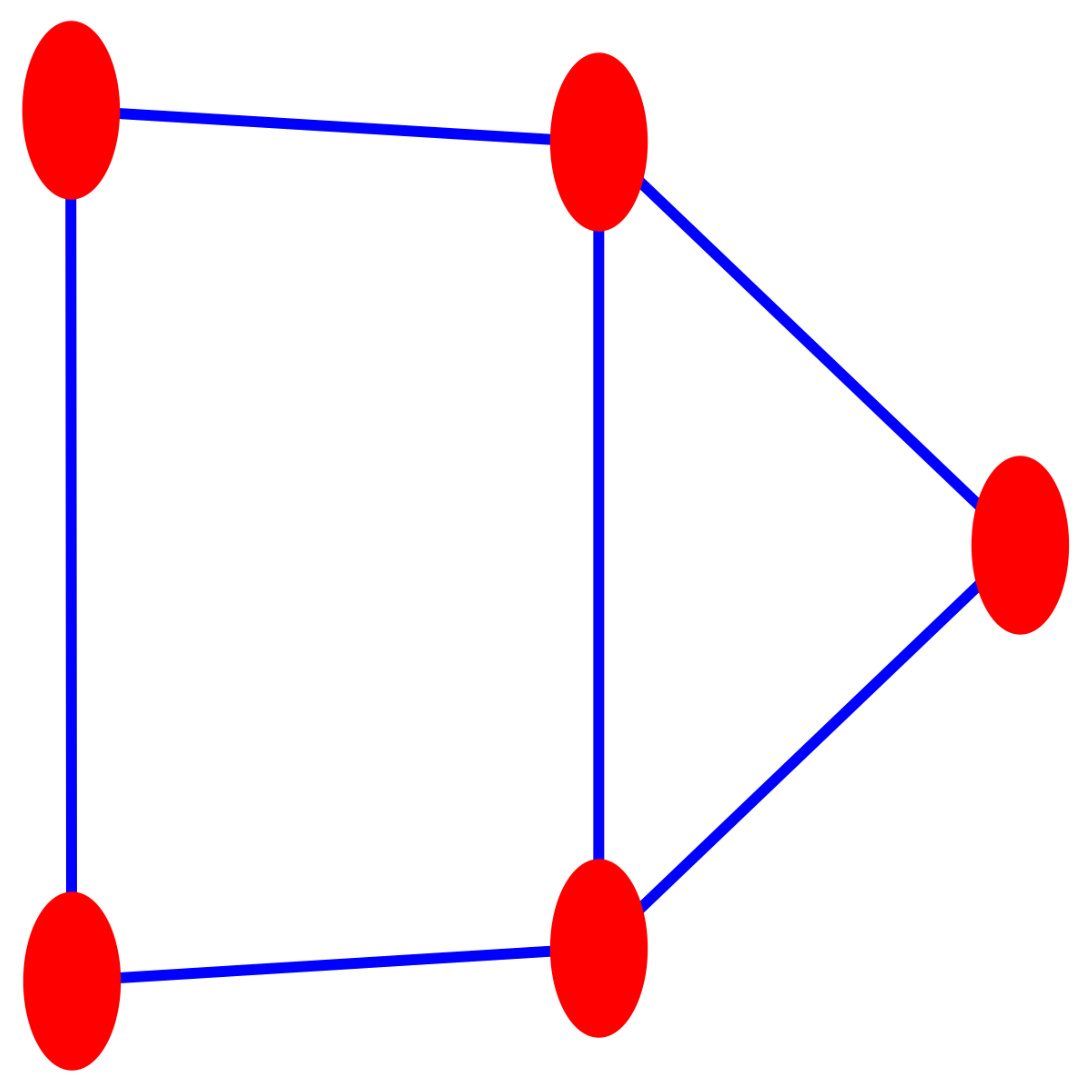}}} }
\parbox{15cm}{
\parbox{3.6cm}{ 9) Snub Cube} \parbox{3.6cm}{ 10) Cricket} \parbox{3.6cm}{ 11) Fly}    \parbox{3.6cm}{ 12) House} }
\parbox{15cm}{
\parbox{3.6cm}{\scalebox{0.061}{\includegraphics{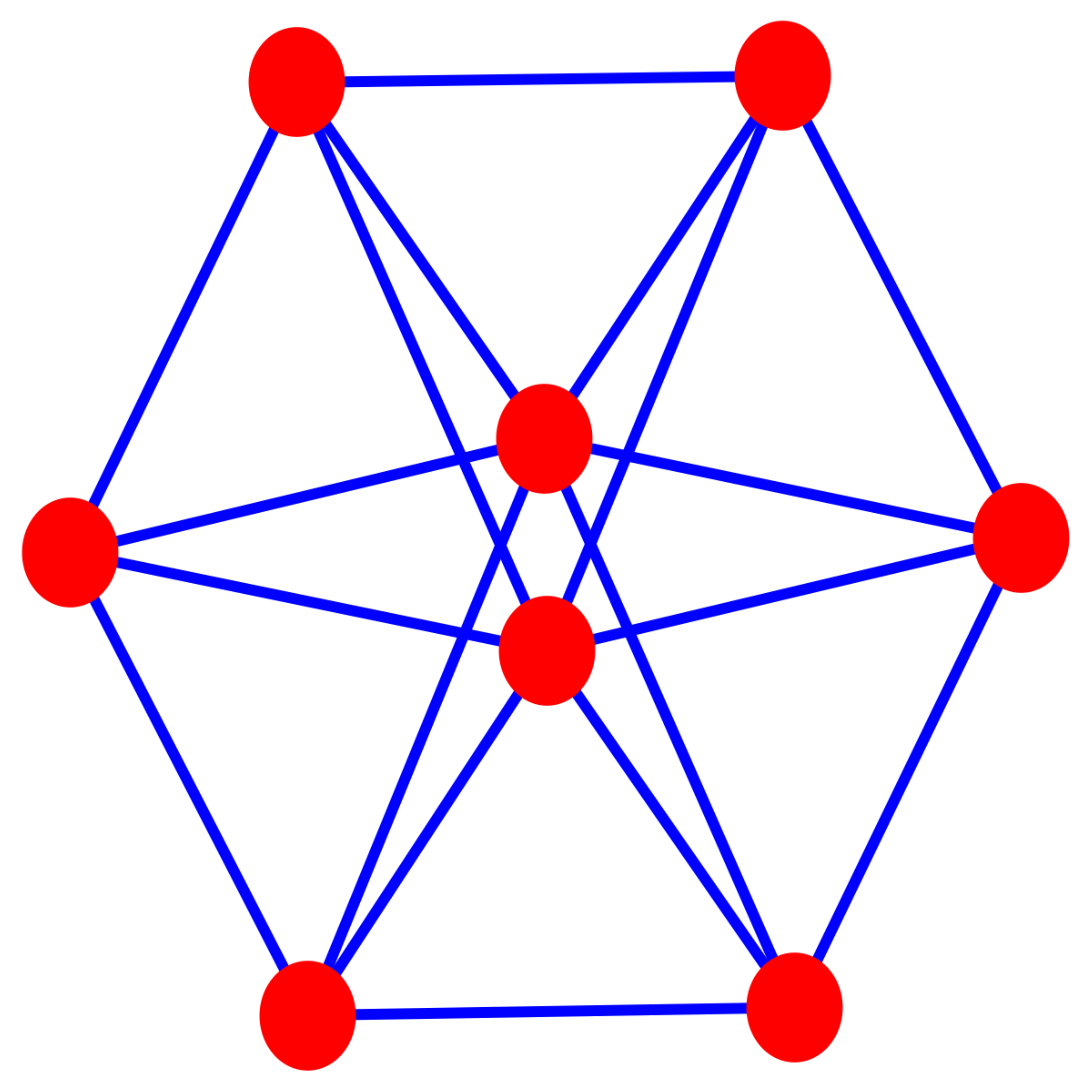}}}
\parbox{3.6cm}{\scalebox{0.061}{\includegraphics{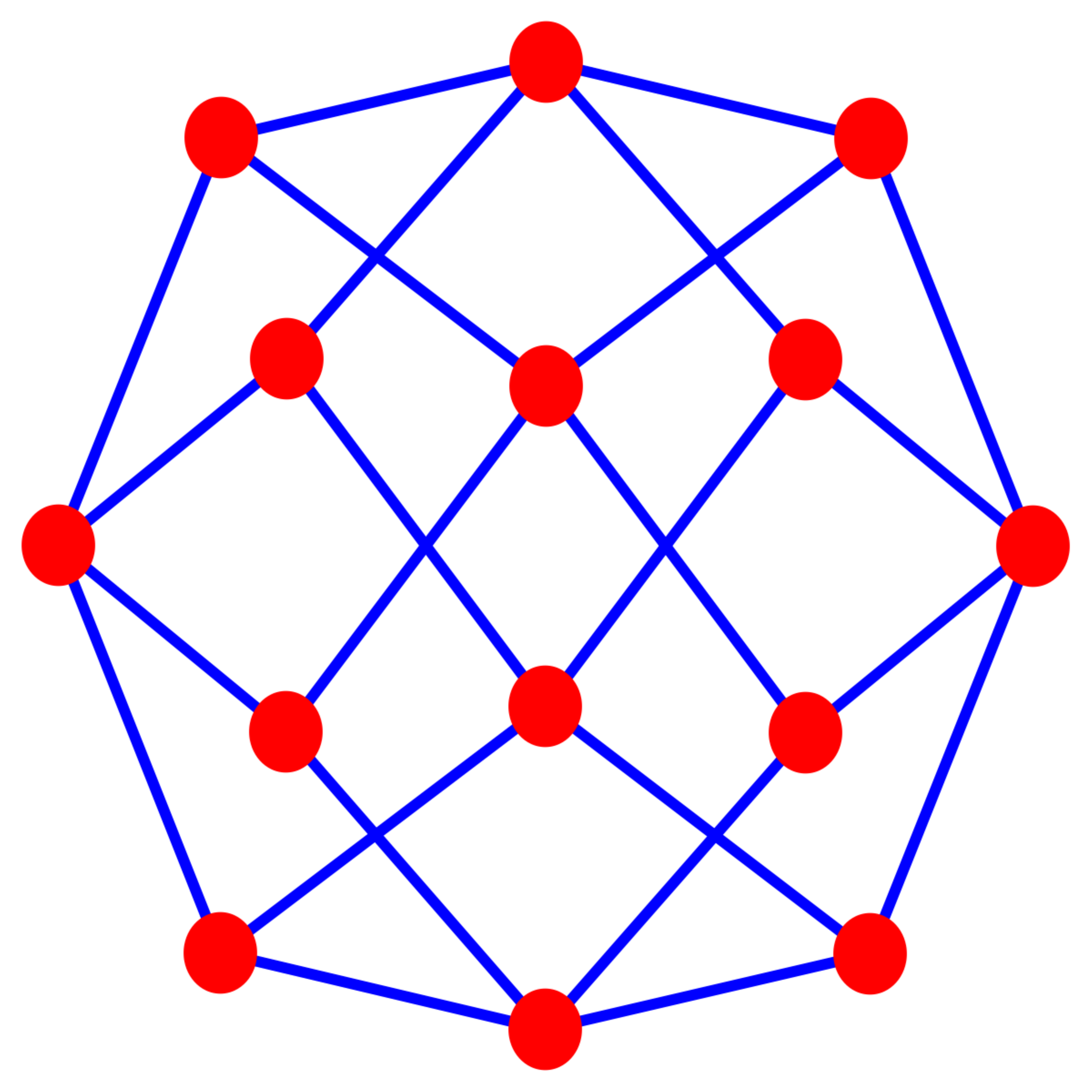}}}
\parbox{3.6cm}{\scalebox{0.061}{\includegraphics{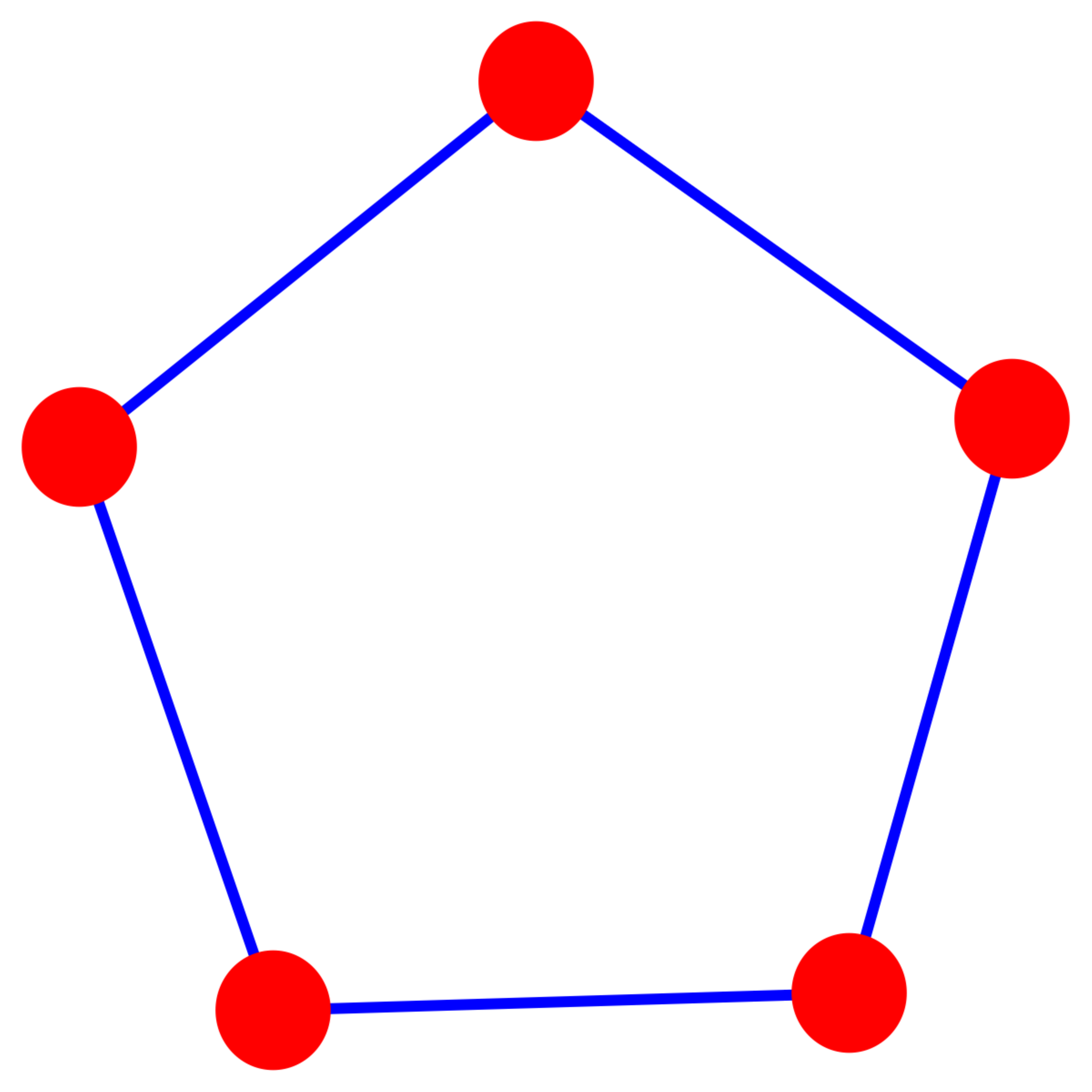}}}
\parbox{3.6cm}{\scalebox{0.061}{\includegraphics{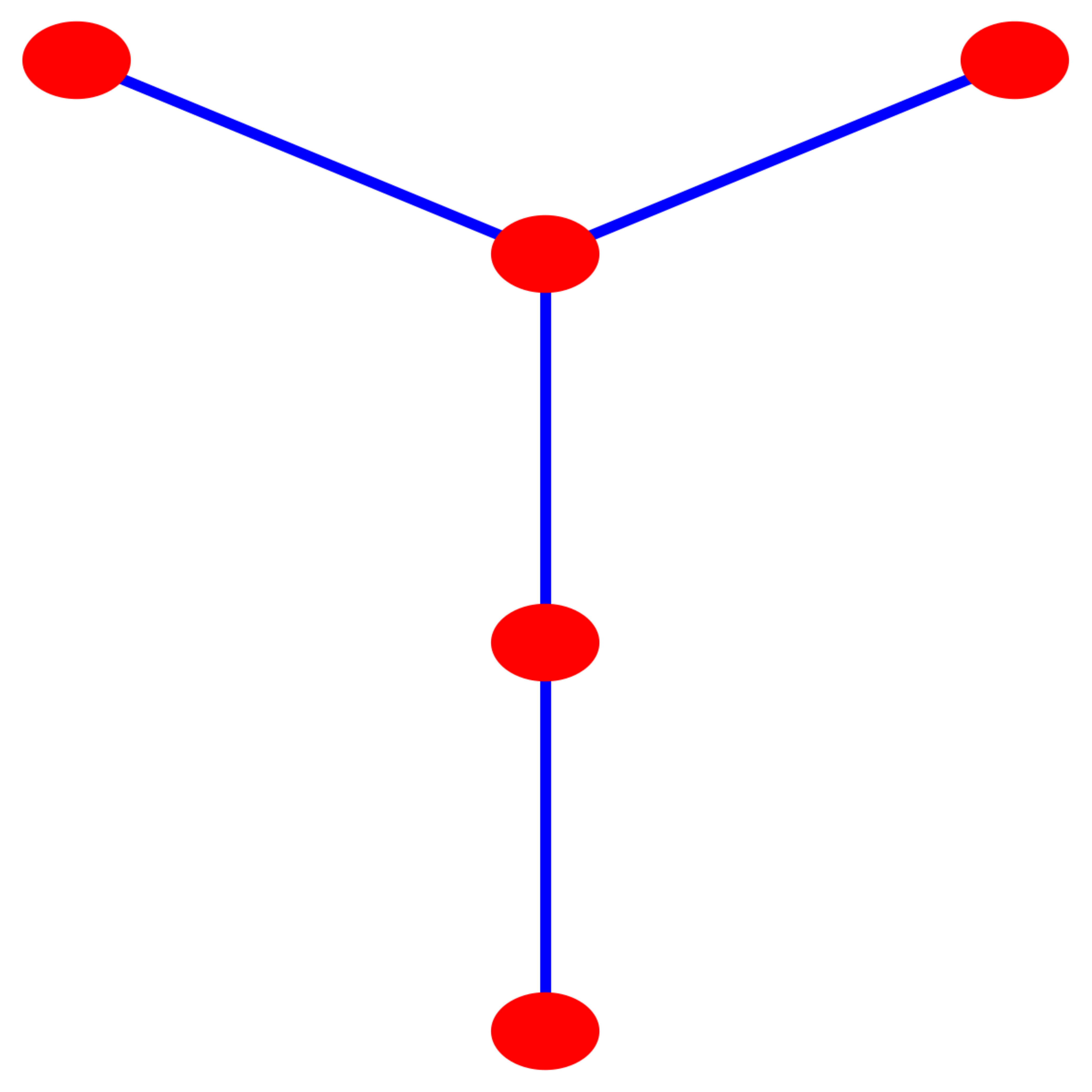}}} }
\parbox{15cm}{
\parbox{3.6cm}{ 13) Prism} \parbox{3.6cm}{ 14) Snub Oct} \parbox{3.6cm}{ 15) Cycle} \parbox{3.6cm}{ 16) Fork} }
\parbox{15cm}{
\parbox{3.6cm}{\scalebox{0.061}{\includegraphics{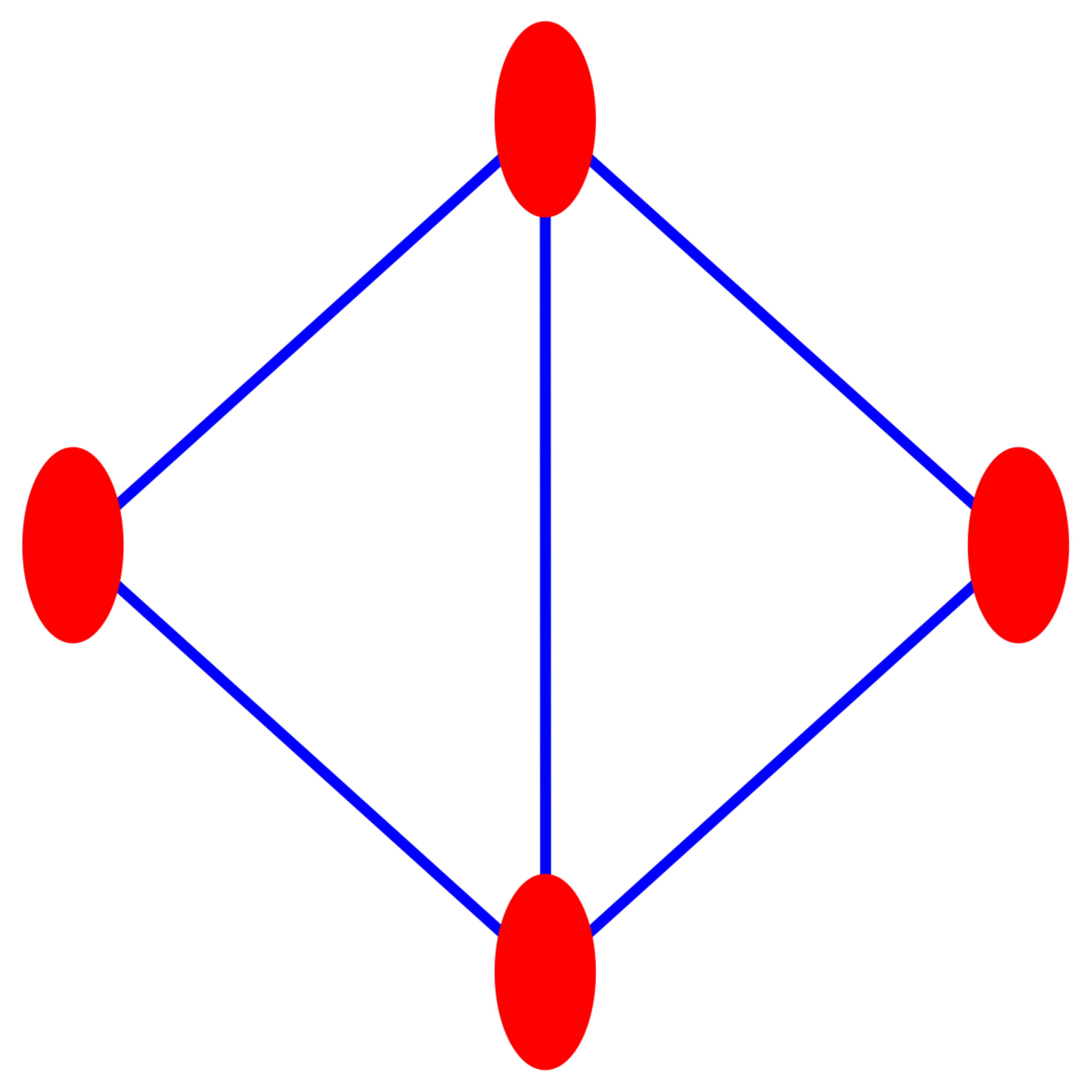}}}
\parbox{3.6cm}{\scalebox{0.061}{\includegraphics{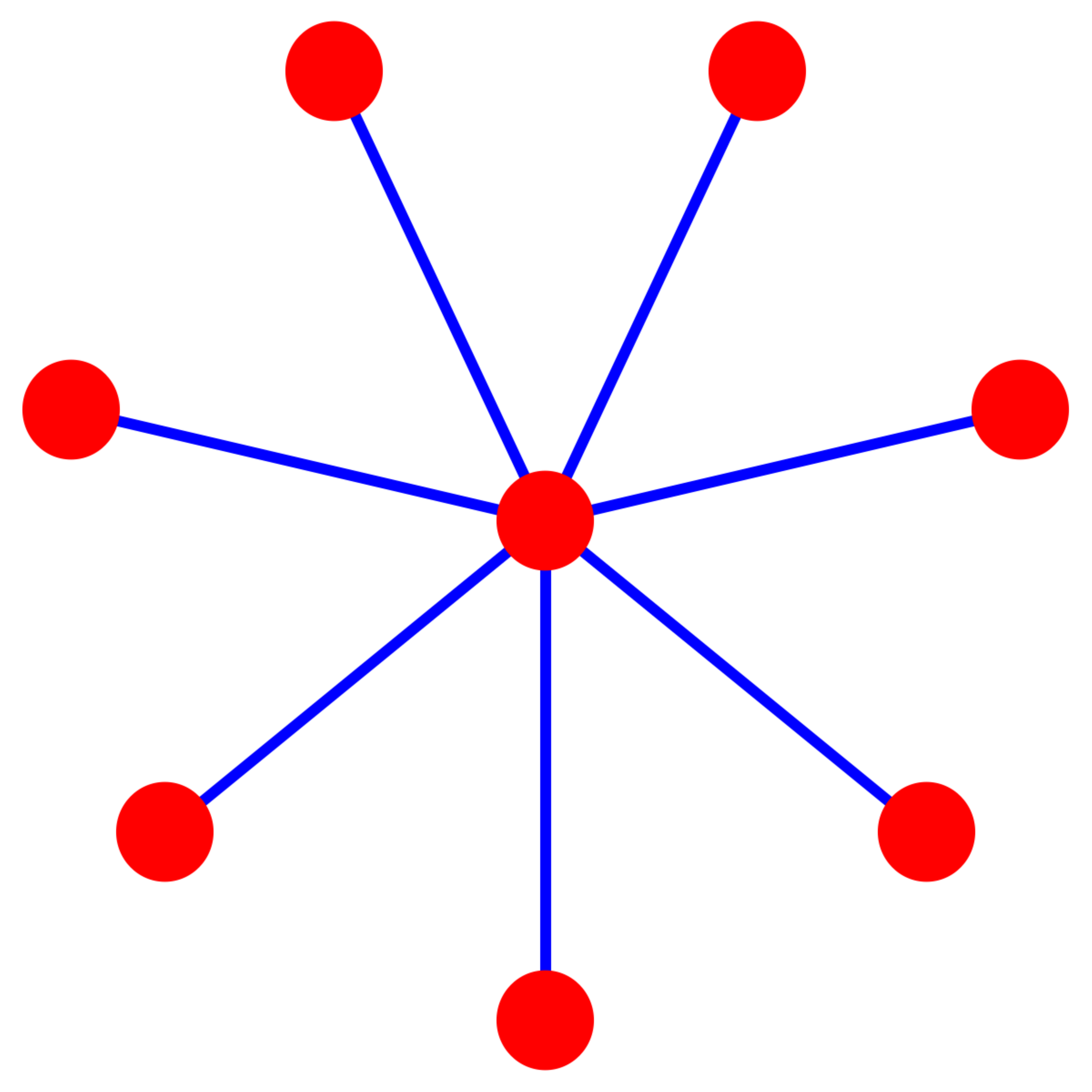}}}
\parbox{3.6cm}{\scalebox{0.061}{\includegraphics{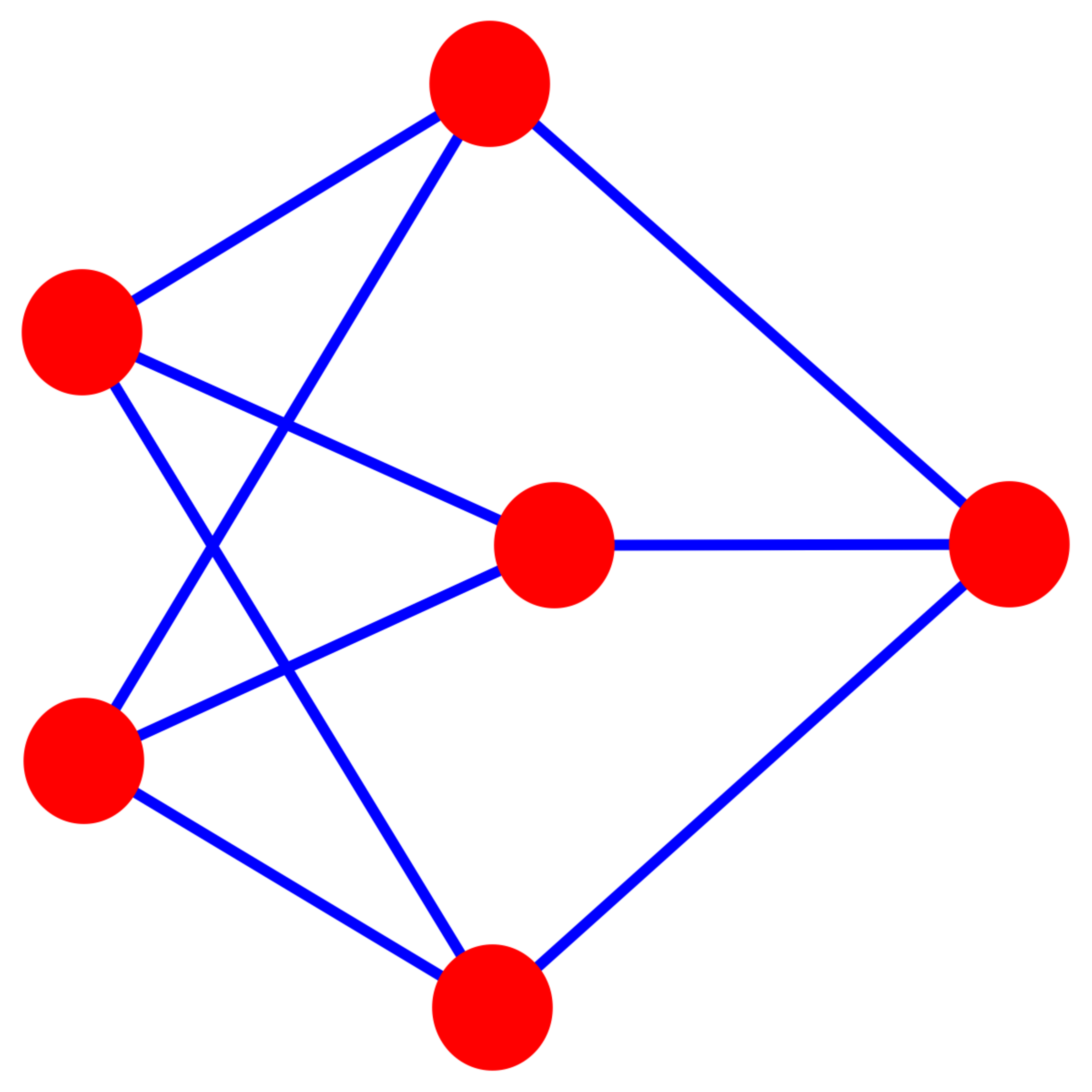}}}
\parbox{3.6cm}{\scalebox{0.061}{\includegraphics{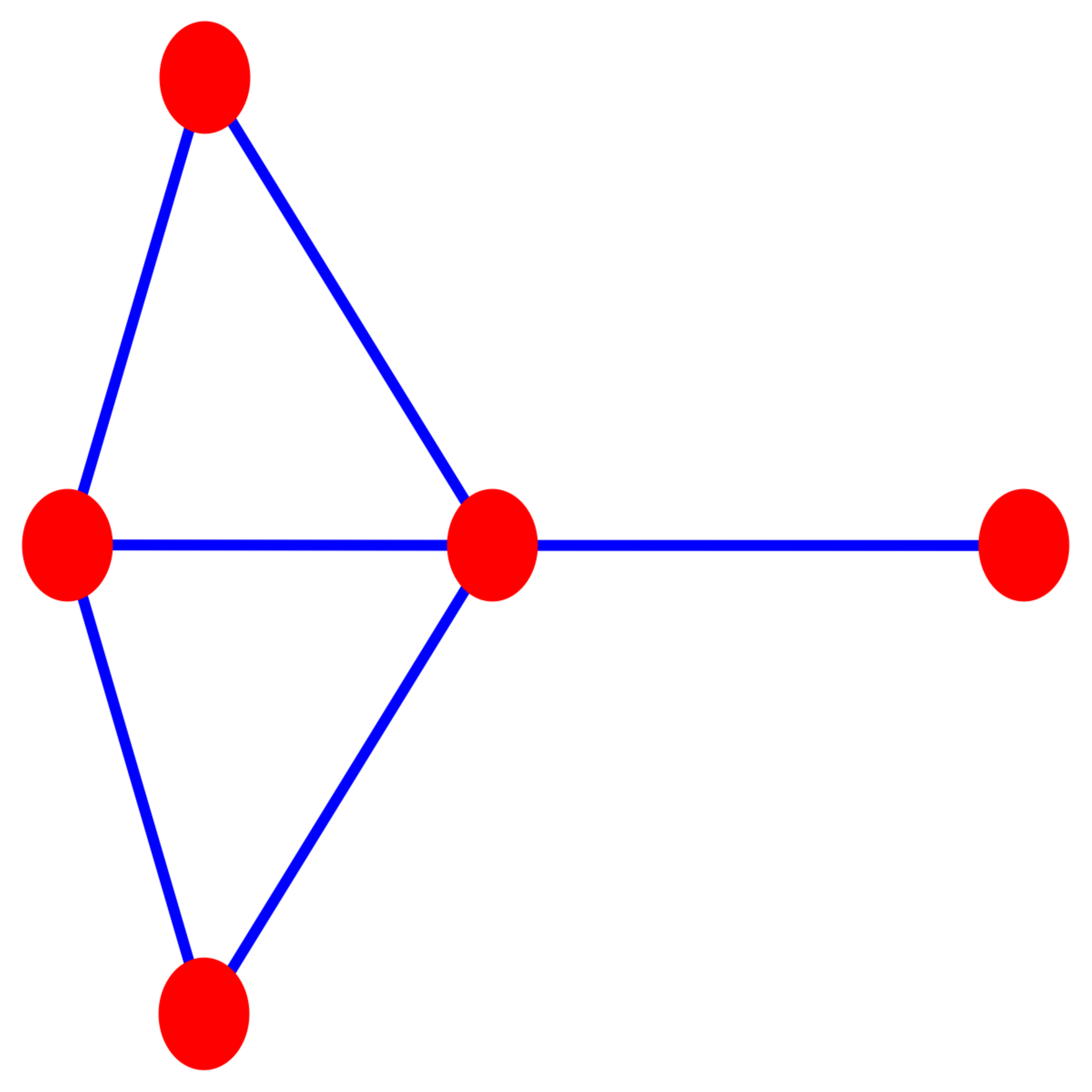}}} }
\parbox{15cm}{
\parbox{3.6cm}{ 17) Kite} \parbox{3.6cm}{ 18) Star} \parbox{3.6cm}{ 19) Utility} \parbox{3.6cm}{ 20) Dart} }
\parbox{15cm}{
\parbox{3.6cm}{\scalebox{0.061}{\includegraphics{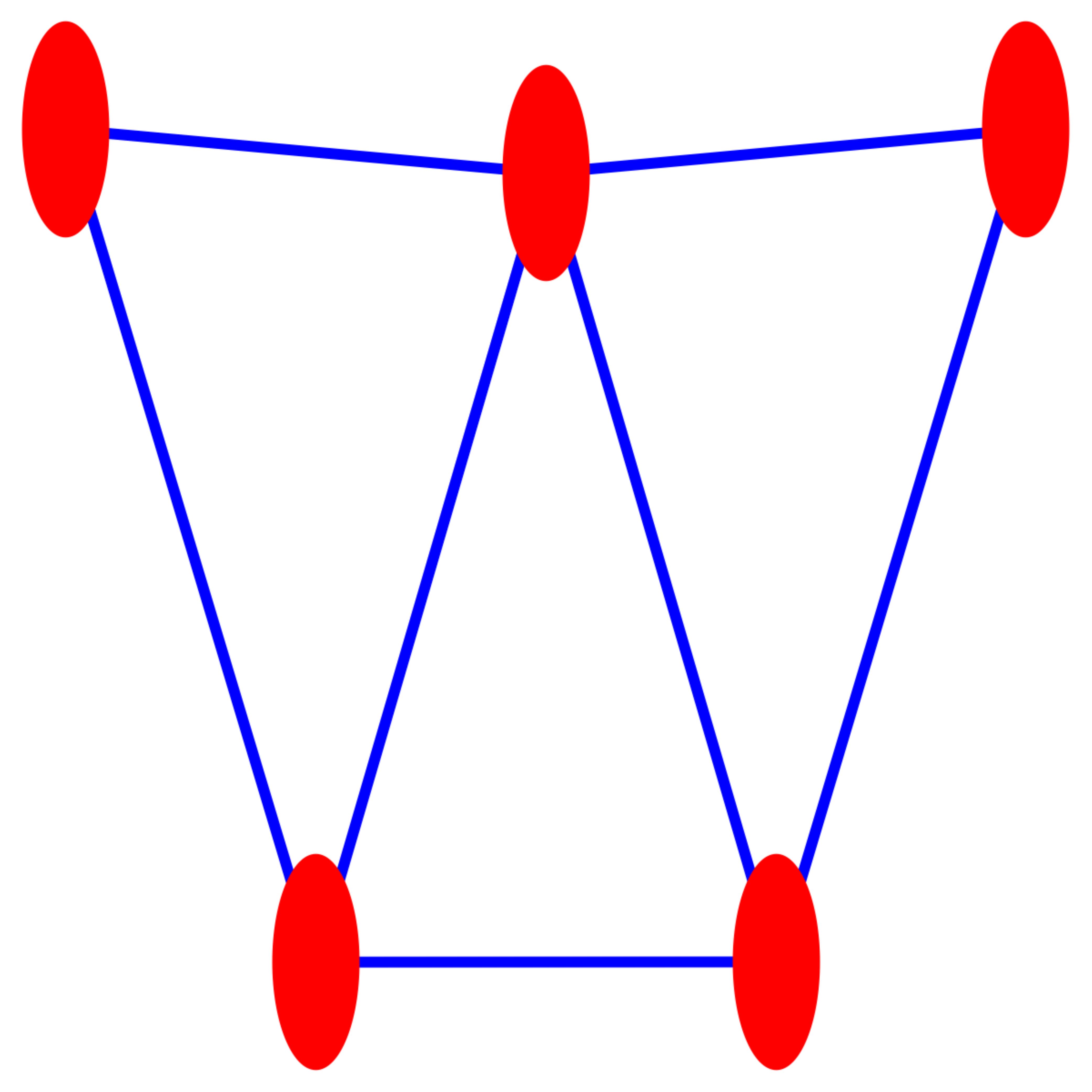}}}
\parbox{3.6cm}{\scalebox{0.061}{\includegraphics{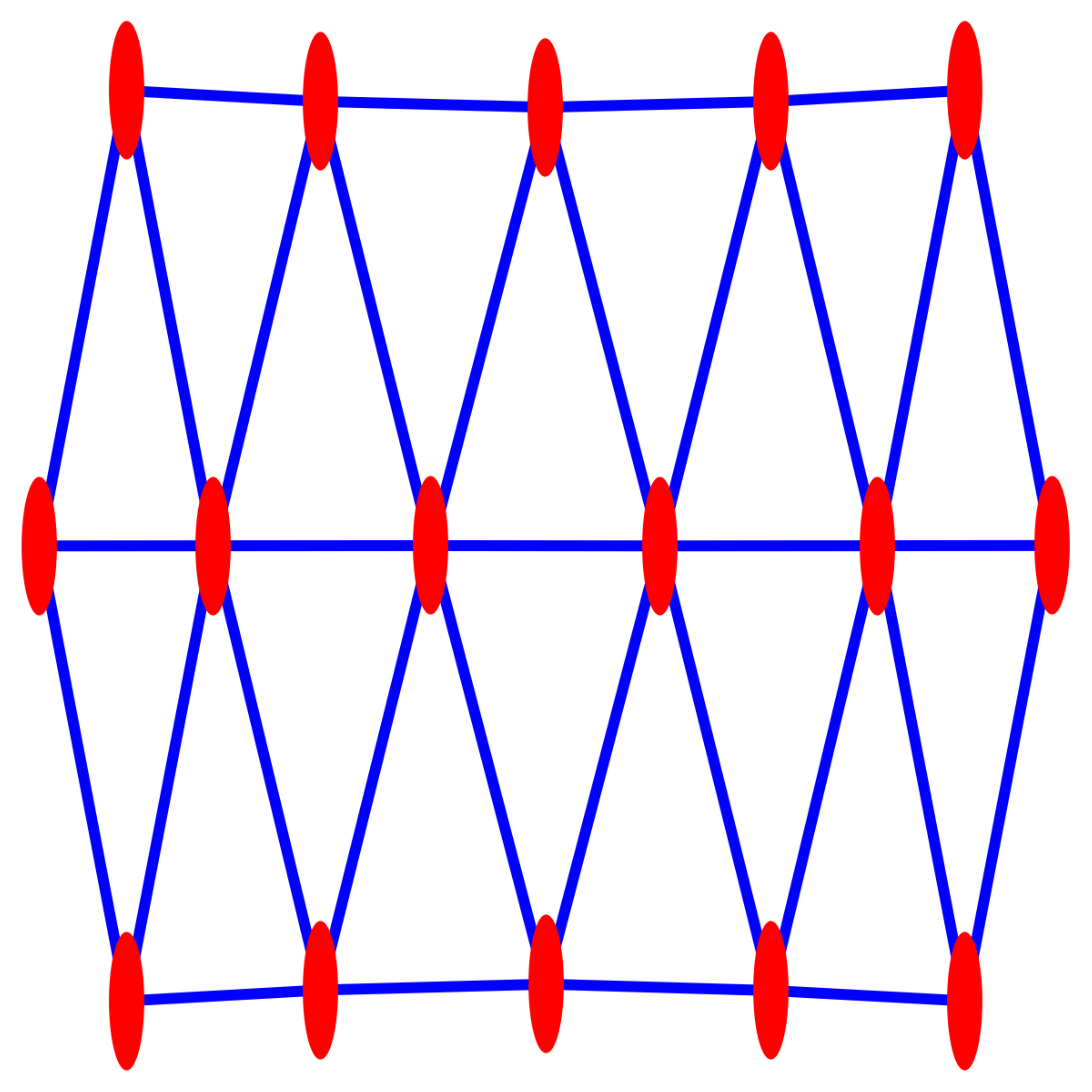}}}
\parbox{3.6cm}{\scalebox{0.061}{\includegraphics{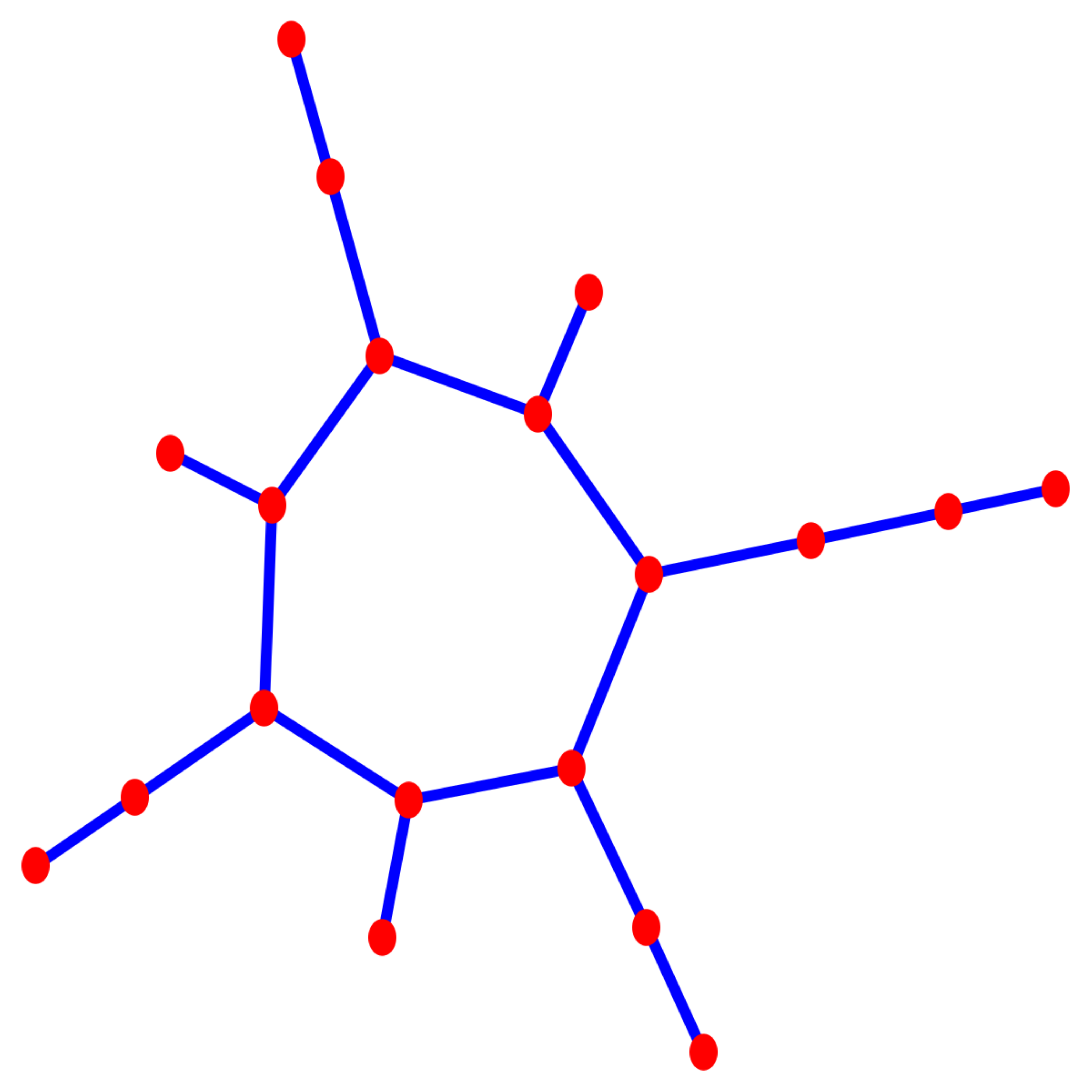}}}
\parbox{3.6cm}{\scalebox{0.061}{\includegraphics{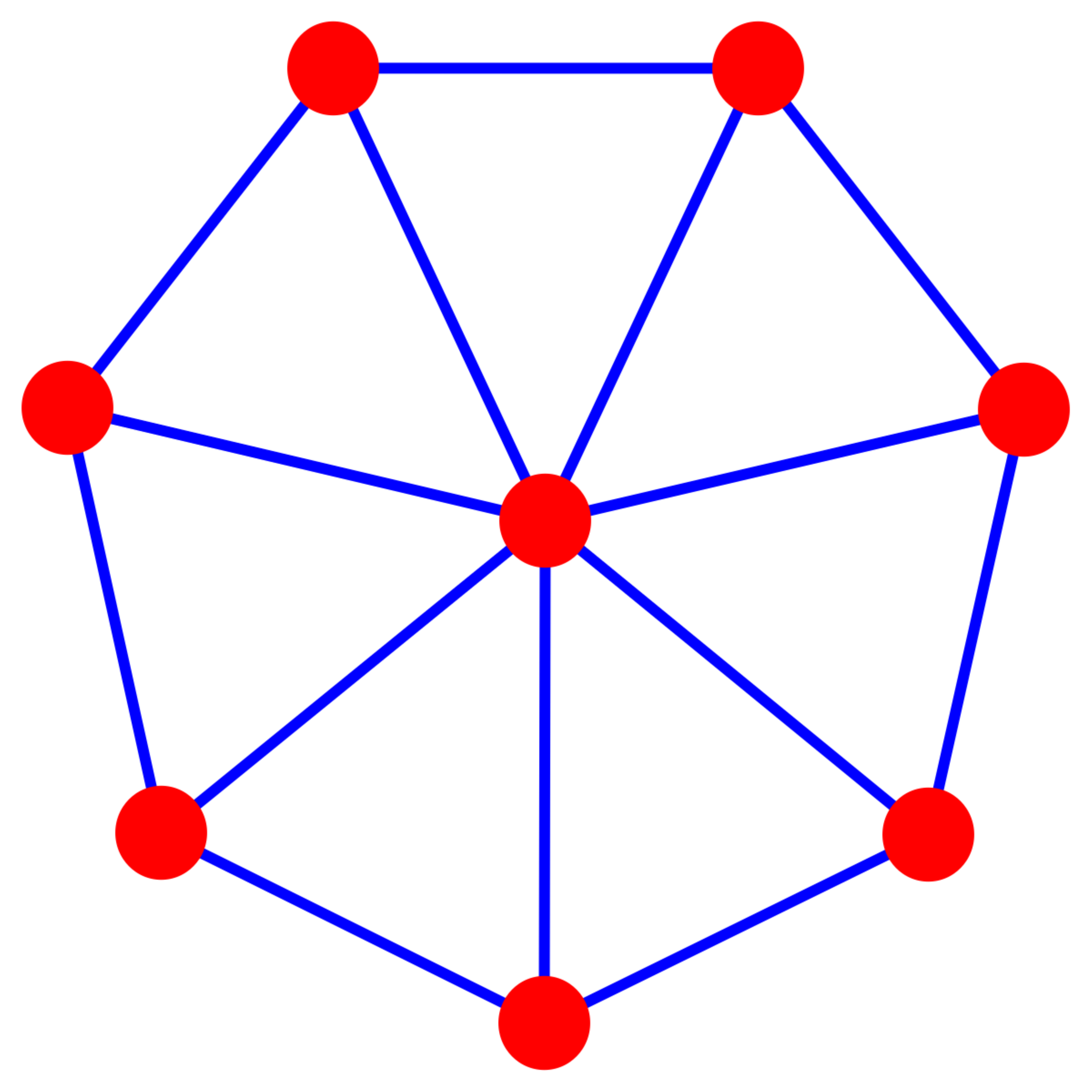}}} }
\parbox{15cm}{
\parbox{3.6cm}{ 21) Gem} \parbox{3.6cm}{ 22) Gate} \parbox{3.6cm}{ 23) Sun} \parbox{3.6cm}{ 24) Wheel} }

\section{Fixed points}

The fixed point theorem for graph endomorphisms \cite{brouwergraph} generalizes
to graph homeomorphisms. 
Any homeomorphism $T:(G,\O) \to (G,\O)$ defines an automorphism of the nerve graph and so 
induces linear maps $T_k$ on the cohomology groups $H^k(G)$ which are finite dimensional 
vector spaces. The {\bf Lefschetz number} is defined as usual as
$L(T) = \sum_{k=0}^{\infty} (-1)^k {\rm tr}(T_k)$.

\begin{thm}
Every graph homeomorphism $T$ with nonzero Lefschetz number $L(T)$ 
has a fixed subgraph consisting 
of a union of elements in $\B$ which all mutually intersect. 
\label{DiscreteKakutani}
\end{thm}

\begin{proof}
By \cite{brouwergraph}, the graph automorphism on the nerve 
graph has a fixed simplex. 
\end{proof}

{\bf Example.} \\
Take a sun graph $C_{(1,2,3,0)}$. Its automorphism group is trivial as 
for most sun graphs. Take the topology given by the graph generated by the 
union of two adjacent rays. The nerve graph is $C_4$. Lets take the reflection
over the diagonal. This induces a reflection on the nerve graph. We have a union
of three base elements which is invariant. This corresponds to an edge which 
stays invariant under a reflection at the nerve graph. \\

{\bf Remarks. } \\
{\bf 1)} It would be nice to have direct proofs of discrete versions of other fixed point 
theorems like the {\bf Poincar\'e-Birkhoff} fixed point theorem \cite{BrNe75}. 
The classical theorem itself implies that there is a discrete version. The point is
to give a purely discrete proof.  Given a topology on an annular graph and assume we have a 
graph homeomorphism $T$ for which any homotopically nontrivial circular chain in $\B$  intersects 
its image and that the boundary components in $\B$ rotate in opposite directions. Then
there is an invariant contractible open set for $T$. \\
{\bf 2)} If the topology is optimal then the invariant set 
established in \ref{DiscreteKakutani} is a union of contractible
sets with contractible essential intersections and therefore contractible. There
is a bound on the diameter then given by $(d+1) M$, where $d$ is the dimension of the 
nerve graph and $M$ is the maximal diameter of sets in $\B$. 

\section{Euclidean space}

{\bf A)} The construction of graph homeomorphism provokes the question whether there is
an analogue constructions in the continuum which is based on homotopy and dimension.
Indeed, a similar notion of dimension allows to characterize the {\bf category of compact manifolds} 
in the more general category of {\bf compact metric spaces}. We will sketch that the notions of 
{\bf "contractibility" and "dimension"} allow to characterize {\bf locally Euclidean spaces}: \\

Lets call a compact metric space $(X,d)$ a {\bf geometric space of dimension $k$},
if there exists $\epsilon>0$ such that for all $0<r<\epsilon$, the unit sphere $S_r(x)$ is a $(k-1)$-dimensional
{\bf Reeb sphere}. To define what a {\rm Reeb sphere} is in the context of metric spaces, 
we need a notion of {\bf Morse function} for metric spaces.
As in the graph case, this notion depends on a predefined notion of {\bf contractibility} in $X$. Lets use 
the standard notion of {\bf contractible = homotopic to a point}. Given a continuous real-valued function $f$ on 
$X$, we call $x$ a {\bf critical point} if there are arbitrary small $r>0$ for which $S_r(x) \cap \{ f(y)=f(x) \; \}$ 
is either empty or not contractible. A continuous real-valued function on $(X,d)$ is called a 
{\bf Morse function} on $X$ if it has only finitely many critical points. A compact metric space is a {\bf Reeb sphere}, 
if the minimal number of critical points among all Morse functions is $2$. A Morse function always has the minimum
$x$ as a critical point. Reeb spheres are spaces for which only one other critical point exists. In the manifold
case, the second critical point is the maximum $y$ with index $1-\chi(X)$ 
so that $\chi(X)=1+(-1)^k$ if $X$ is a $k$-dimensional Reeb sphere. Note however that we have a notion of Euler 
characteristic only a posteriori after showing the metric space is a topological manifold. For the standard
Cantor set for example, the Euler characteristic is not defined and indeed the spheres $S_r(x)$ can fail to be 
Reeb spheres for arbitrary small $r>0$. \\

{\bf Any 1D connected geometric metric space is a classical circle.} \\
Proof: a zero-dimensional Reeb sphere consists of two points. Since $S_r(x)$ consists of two points $a_x(r),b_x(r)$
and because the distance function is continuous, we have two continuous curves $r \to a_x(r)$ and $r \to a_y(r)$ in $X$ 
which cover a neighborhood of $x$. This shows that a small ball in $X$ is an open interval.  
For every $x$ we have such a ball $B(x)$. They form a cover. By compactness of $X$, there is a finite 
subcover. They produce an atlas for a topological $1$-manifold. 
Each connected component is a boundary-free $1$-dimensional connected manifold 
which must be a circle. We have seen that any one-dimensional {\bf connected} one-dimensional geometric metric space is 
a circle.  \\

{\bf Any 2D geometric metric space is a topological two manifold.} \\
Proof. We have just established that every small enough
sphere $S_r(x)$ is a one-dimensional Reeb sphere and so a connected one-dimensional circle. 
We have now a polar coordinate description of a neighborhood of a point. This open cover $\{ B(x) \; | \; x \in X \}$
produces an atlas for the topological manifold, a compact metric space which is locally Euclidean. The coordinate changes
on the intersection is continuous
The classical Reeb theorem \cite{Mil63} (which requires to check the conditions
for smooth functions only) 
assures that two-dimensional geometric spheres are classical topological spheres.  \\

We can now continue like this and see that $3$-dimensional geometric spaces are topological three manifolds. 
Again, by invoking the classical Reeb theorem, we see that three-dimensional geometric spheres must be topological three spheres.
We inductively assure that a $d$-dimensional geometric space with integer $d$
must be a compact topological manifold of dimension $d$ and that a $d$-dimensional Reeb sphere is homeomorphic to 
a classical $d$-dimensional sphere. \\

{\bf B)} Lets see how the fixed point theorem implies a variant of the classical Lefschetz-Brouwer fixed point theorem
in a more general setting. 
Let $(X,d)$ be a compact metric space and let $T:X \to X$ be a homeomorphism. Let $\B$ be a finite cover
on $X$ of maximal diameter $\epsilon$. It generates a finite topology $\O$ ion $X$ which defines a graph for
which we can look at the cohomology. 
We assume that the cover is {\bf good} in the sense that the cohomology is finite and does not change for 
$\epsilon \to 0$. Since the new cover $T(\B)$ generates a different topology, $T$ does not produce a homeomorphism
of $\O$ yet, but we can take the permutation of $\O$ which is closest in the supremum topology. 
This modification has a fixed contractible set by the fixed point theorem for graphs \cite{brouwergraph}.
If $X$ is finite-dimensional, then there is a bound on the diameter $M(d+1) \epsilon$ of this fixed contractible set. 
For every $\epsilon=1/n$, there is a point $x_n$ which satisfies $d(T(x_n),x_n) \leq C\epsilon$. An accumulation point
of $x_n$ as $n \to \infty$ is a fixed point of $T$. We have sketched:  \\

Assume $(X,d)$ is a metric space and $T$ is a homeomorphism. Assume that there are arbitrarily fine
covers $\B_n$ for which the \v{C}ech cohomology $H^k(X)$ of the generated topology $\O_n$ 
is the same and $H^k(X)=0$ for large enough $k$. Assume further that
sufficiently fine approximations $T_n$ of $T$ in the uniform topology produce graph automorphisms 
for which the Lefshetz number is constant and nonzero for every fine enough cover $\B_n$. 
Then $T$ has a fixed point. \\

{\bf C)} A {\bf good triangularization} of a topological manifold 
$M$ is a geometric graph $G$ of dimension $d$ such that the topological
manifold $N$ constructed from $G$ obtained by filling any discrete unit ball with an Euclidean unit ball and 
introducing obvious coordinate transition maps is homeomorphic to $M$. 
Any good triangularization defines now a \v{C}ech cover $\U$ of the manifold. 
We expect that any two good triangularizations are homotopic and that there is a third, possibly 
finer triangularization which is homeomorphic to both. 

\begin{figure}[H]
\scalebox{0.15}{\includegraphics{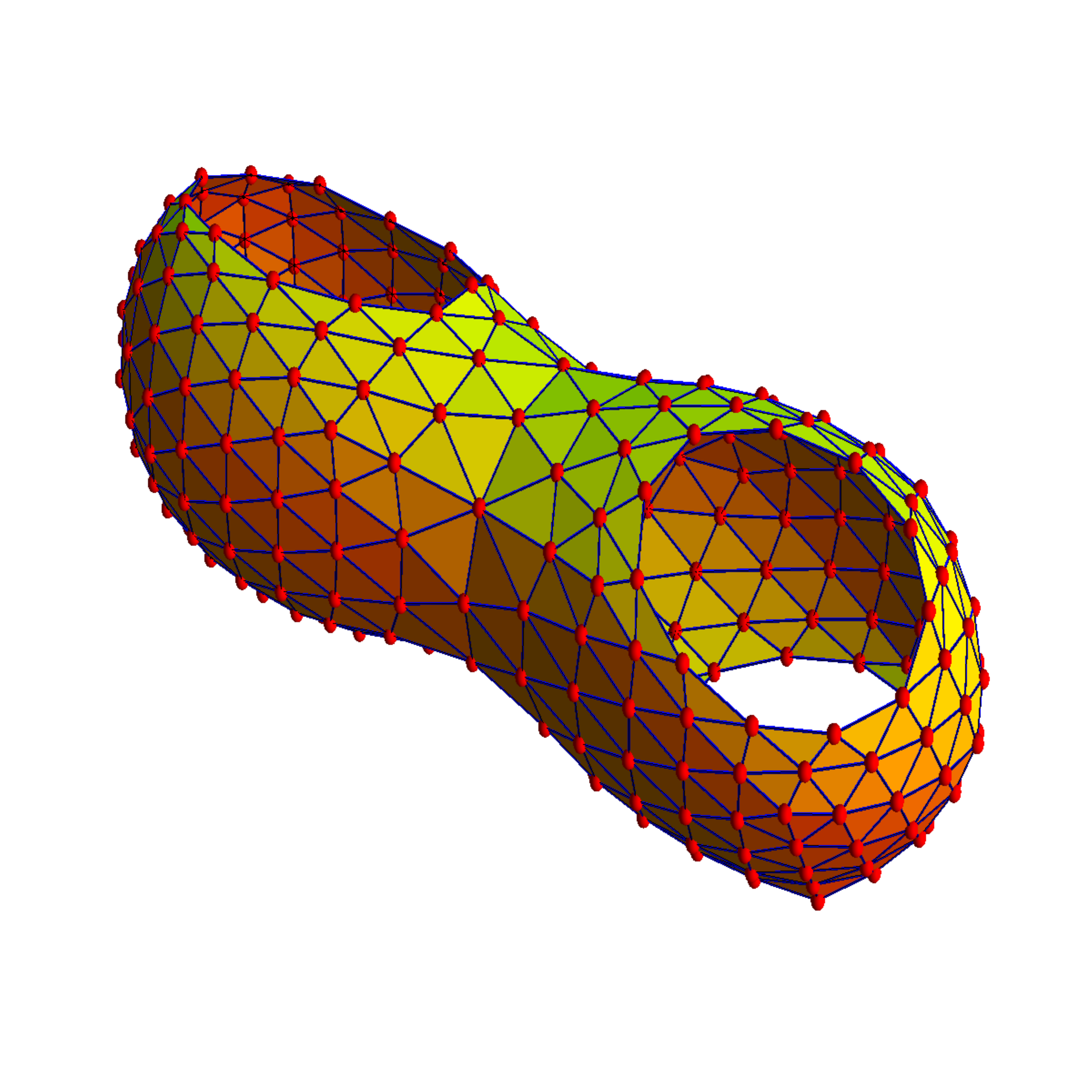}}
\caption{
An example of a two-dimensional geometric graph $G$ in which every unit sphere is either a cycle
graph (interior points) or an interval graph (boundary points). This graph has uniform
dimension $2$ and Euler characteristic $2-b=-1$, where $b=3$ is the number 
of boundary components. It is a discretization of a two-dimensional manifold $M$ with boundary. 
An optimal topology for $G$ comes from a \v{C}ech cover $\B$ of that manifold. In the graph
case, $\B$ could be built with geodesic balls of radius $2$. Once a good cover is found, we
could refine the graph using {\bf barycentric refinements} in which triangle and edges are split.
Such refinements are examples of 2-homeomorphism and of course a homeomorphism in our sense. 
\label{geometric}
}
\end{figure}

\section{Remarks}

{\bf 1)} Two topological graphs $H,G=\phi(H)$ which are isomorphic as graphs are homeomorphic as topological 
graphs with suitable topologies: because any finite simple graph carries the star topology generated by 
stars which are one-dimensional maximal trees in the unit ball.
Choose such a topology $\O$ on $H$ generated by a subbasis $\A$. Then $\B = \{ \phi(A) \; | \; A \in \A \; \}$ 
generates a topology on $G$. 
Two contractible graphs of the same uniform dimension are homeomorphic using the indiscrete topology.
For example, any wheel graph is homeomorphic to a triangle. 
For non-contractible graphs, the indiscrete topology of course is forbidden. An icosahedron
for example needs a topology with $6$ elements which makes it homeomorphic to the octahedron. We can take
unit neighborhood of antipodal triangles, two standard unit balls and two unit balls with an additional 
triangle. \\ 

{\bf 2)} Since $\O$ is a standard topology, we can look at
at product, quotient and relative topologies. But one has to be careful: \\
{\bf a)} Given two topological graphs 
$(G_1,\O_1),(G_2,\O_2)$, the {\bf product topology} $\O_1 \times \O_2$ is a topology on the Cartesian 
product $G_1 \times G_2$ of the graphs,
but we have to make sure that the product topology is also a graph topology and therefore provide a 
subbasis with the right properties. The subbases $\B_1 \times \B_2$ consisting of all product graphs $B_1 \times B_2$
with $B_i \in \B_i$ often works. The product $C_4 \times K_2$ for example is the cube graph. If $\B_1$ is the fine
topology on $C_4$ with $4$ interval graphs and $\B_2 = \{ (1,2) \}$ is the indiscrete topology, then 
$\B=\B_1 \times \B_2$ consists of $4$ sets.  It does not produce a topology however since
$\chi(G_1 \times G_2) = 8-12=-4$ but the nerve graph is isomorphic to $C_4$ and has $\chi(\G)=0$.  \\
{\bf b)} Now lets look the {\bf quotient topology} given by an equivalence relation $\sim$ on $G$ for which 
$H=G/\sim$ is a graph, then $\O/\sim$ defines a topology on $H$ in the set theoretical sense but not 
necessary a graph topology. We assume that $G/\sim$ is a finite simple graph removing possible 
loops or multiple connections from the identification. \\
{\bf c)} Finally, lets look at the {\bf induced topology}. If $H=(V,E)$ is a subgraph of $G$, then
the topology $\H = \{ A \cap V \; | \; A \in \O \; \}$ is in general not a topology which makes $H$ a topological
graph. This can not be avoided: lets take the wheel graph $G=W_4$ which has the embedded circle $H=C_4$. The graph $G$ is
a topological graph with the indiscrete topology. The induced topology on $H$ is the indiscrete topology on the circle
which makes sense as a topology but it is not a graph topology on the circle because $C_4$ is not contractible. \\

{\bf 3)} We have seen that the classical notion of graph homeomorphism \cite{TuckerGross} is
equivalent to the just defined notion if the graphs have no triangles.
The classical $1$-homeomorphisms do not preserve Euler characteristic, nor dimension in general.
A triangle for example has dimension $2$ and Euler characteristic $1$. It is $1$-homeomorphic to 
$C_6$ which has Euler characteristic $0$ and dimension $1$. 
Still, the notion of $1$-homeomorphism is an important concept.
It is in particular essential for {\bf Kuratowski's theorem}e. There are higher dimensional versions
of {\bf $k$-homeomorphisms} which allow for barycentric refinements of maximally $k$-dimensional 
subsimplices. For graphs containing no $K_4$ graphs, such $2$-homemorphisms would be a special case of a 
homeomorphisms in the sense given here.  \\

{\bf 4)} Downplaying membership and focussing on functions is a point of view
\cite{Lawvere} which appears natural when studying the topology of graphs. Since graphs are finite objects,
the corresponding Boolean algebras are naturally complete {\bf Heyting algebras}.
Topological graphs are {\bf locales} which illustrate that a ``pointless topology point of view" 
\cite{Johnstone} can be useful even in finite discrete situations.
The notion of homeomorphism given here is simpler than any use of multi-valued maps, with which we 
have experimented before. Our motivation has been to generalize the Bouwer fixed point theorem \cite{brouwergraph}
to set-valued maps. {\bf Kakutani's fixed point theorem} deals with set-valued maps. The pointless 
topology approach makes the classical Kakutani theorem appear more natural since it allows to avoid 
set-valued functions $T$. We only need that $T(x)$ is contractible for all $x$. 
This assumption for the Kakutani theorem implies that 
contractible sets are mapped into contractible sets. So, if we go the pointless path in the Euclidean 
space and define the topology on a convex, compact subset $X$ of ${\bf R}^n$ 
generated by the set $\B$ of all open contractible sets and $T$ to be an isomorphism of the corresponding 
Heyting algebra  (without specifying points, which is the point of point-less topology), then $T$ has 
a fixed contractible element in the Heyting algebra of arbitrary small diameter. This by compactness implies
that $T$ has a fixed point. We just have sketched a proof that the classical Kakutani theorem follows from 
the corresponding fixed point theorem for point-less topology. We could also reduce to 
the graph case when invoking nonstandard analysis \cite{Nelson77} because in nonstandard analysis, compact 
sets are finite sets and $d$-dimensional geometric graphs and $d$-dimensional manifolds are very similar 
from a topological point of view, especially for fixed point theorems. Again, also in nonstandard analysis,
the pointless topology approach is very natural. For compact manifolds, there is more than the finite set:
we can not just look at the graph of all pairs $(x,y)$ for which $|x-y|$ is infinitesimal. We need to capture
the nature of a $d$-dimensional compact set by looking at a finite nonstandard cover $\B$ for which the nerve
graph is a nonstandard geometric graph. \\

{\bf 5)} The usual notion of {\bf connectedness} is not based on point set topology in graph theory. The reason
is that a graph has a natural metric, the {\bf geodesic distance} which defines a topology on the vertex
set. Note however that technically, any graph is completely disconnected with respect to this metric because 
$B_{1/2}(x) = \{ x \; \}$ so that every subset of a vertex set is both open and closed. In other words,
with respect to the geodesic topology, the graph is {\bf completely disconnected}. Also, whenever we have a finite
point set topology $\O$ on a graph and two sets $U,V \in \O$ whose union is the entire set,
we already have disconnectedness This is often not acceptable. What we understand under
connectedness in graph theory is {\bf path connectedness}, a notion in which edges are part of the picture.
But as the subbasis $\B = \G_1$ shows, one can not use the classical notion of connectedness for
the topology $\O$ generated by $\B$. The notion of connectedness introduced here for topological 
graphs is also natural for general topology: given a topological space $(X,\O)$ with a subbasis $\B$ 
for which every finite intersection of elements in $\B$ is 
connected, then classical connectedness is equivalent to the fact that $\B$ can not be written as a union $\B_1 \cup \B_2$ 
for which any $B_1 \in \B_1$ and $B_2 \in \B_2$ has an empty intersection. Proof. If $X$ is disconnected, 
then $X = U_1 \cup U_2$ with two disjoint nonempty open sets $U_i$. The set of subsets $\B_i = U_i \cap \B$ 
partition $\B$ because each $U_i \cap B$ is either empty or $B$ 
(otherwise $(U_1 \cap B) \cup (U_2 \cap B) = B$
shows that $B$ is not connected). On the other hand, if we can split a subbasis $\B$ of $X$ into two 
sets $\B_1,\B_2$ then $U_i=\bigcup_{B \in \B_i} B$ are open sets which unite to $X$ and which are both 
not empty, showing that $X$ is not connected. \\
(The example of $X=[0,1] \cup [2,3]$ with topology generated by the subbasis 
$\B = \{ [0,p/q] \cap X \; | \; p/q \in Q \}$ shows that connectedness of elements in $\B$ is needed
because $\B$ can not be written as a union of two disjoint sets. 
The example of the Cantor set shows that not every topological space has a subbasis consisting of 
connected sets.) 
We don't have to stress how important the notion of connectedness is in topology: to cite 
\cite{Sato}: {\it "In topology we investigate one aspect of geometrical objects almost exclusively of the others: 
that is whether a given geometrical object is connected or not connected. We classify objects according
to the nature of their connectedness. One focuses on the connectivity, ignoring changes caused by 
stretching or shrinking."} \\

{\bf 6)} Graph homeomorphisms do preserve the dimension ${\rm dim}_G(A)$ for $A \in \B$.
We did not ask that the dimension of all $A \in \O$ remain the same even so this is often the case.
Requiring the dimension to be constant for all elements in $\O$ would be a natural assumption but we do 
not make it to have more flexibility and simplicity. Checking the dimension assumption for a couple of 
elements $\B$ is simpler. For the entire topology, it could be a tough task to establish and produce too
much constraints. Note that in general, the topology $\O$ is the discrete topology. 
Classically, the inductive topological dimension is preserved by homeomorphisms but it is important to
note that inductive dimension is classically a global notion, a number attached to the topological space itself. 
For example: take a classical Euclidean disc and attach one-dimensional hairs. This 
space is two-dimensional classically using the inductive dimension of Menger and Urysohn \cite{HurewiczWallman}.
It contains however open subsets like neighborhoods of points on the hair which are one-dimensional. 
Since we have a local dimension in the discrete, we could define for a metric space $(X,d)$ a 
{\bf local topological dimension} at a point as the limit of 
${\rm dim}(X \cap B_r(x))$ with $r \to 0$ if the limit exists. In the discrete we do not have to worry
about such things and have a local notion of dimension which works. \\

{\bf 7)} Lets look at some classical notions of dimension and see what they mean in the case of graphs: 
To cite \cite{Crilly}: {\it "The concept of dimension, deriving from our understanding of the dimensions of
physical space, is one of the most interesting from a mathematical point of view".} \\
{\bf a)} The classical {\bf Hausdorff dimension} for metric spaces is not topological. It only is invariant 
under homeomorphisms satisfying a Lipshitz condition. 
It can change under homeomorphisms as any two Cantor sets are homeomorphic, but the dimension of a Cantor set
can be pretty arbitrary. If we apply the notion of Hausdorff dimension to graphs verbatim, then
the dimension is zero at every point. It appears not interesting for graphs. The dimension $\dim(G)$ is a
good replacement which shares with the Hausdorff dimension the property that it is not a topological invariant. 
We still need to explore for which compact metric spaces we can find graph approximations such that the 
dimension converges. For the standard Cantor set $X \subset [0,1]$ for example we would have to approximate
the set with graphs which are partly one and partly zero-dimensional. \\
{\bf b)} Classically, the {\bf Lebesgue covering dimension} of a topological space is the
minimal $n$ such that every finite open cover contains a
subcover in which no point is included more than $n+1$ times. It is a topological invariant. 
In other words, the nerve graph of the subcover has degree smaller or equal to $n+1$.
We can define the {\bf Lebesgue covering dimension} of a graph as the maximal dimension 
of a point in the nerve graph of a topology. For the discrete topology generated from $\B=\G_1$, 
the Lebesgue covering dimension is the maximal degree. For a triangularization of a $k$-dimensional
manifold, the Lebesgue covering dimension is $k$. \\
{\bf c)} Classically, the {\bf inductive dimension} is the smallest $n$ such that
every open set $U$ has an open $V$ for which its closure $\overline{V}$ 
is in $U$ has a boundary with inductive dimension $\leq n-1$.
Using the unit sphere, we get a useful notion however as we have seen. \\
{\bf d)} In algebraic topology (i.e. \cite{Hatcher}), a graph is considered a 
{\bf one-dimensional cell complex}. 
Graphs therefore are often considered one-dimensional or treated as discrete analogues of algebraic curves. 
Graphs however naturally have a {\bf CW complex structure} by looking at the complete 
subgraphs in $G$. This is the point of view taken here. \\

{\bf 8)} Classically, two topological spaces $X,Y$ are homotopic, if there are continuous maps $f:X \to Y$
and $g: Y \to X$ for which $S=g \circ f: X \to X$ and $T=f \circ g: Y \to Y$ allow
for continuous maps $F: X \times [0,1] \to X$ and $G:Y \times [0,1]$ 
satisfying  $F(x,0)=S(x), F(x,1)=x$ for all $x \in X$ and $G(y,0) = T(y), G(y,1)=y$ for all $y \in Y$. 
Since graph topology depends heavily on discrete homotopy, it is natural to ask whether one can 
reformulate discrete homotopy to match the classical notion. This is
indeed possible if the homotopy step $X \to Y$ has the property that the inclusion $X \to Y$
and projection $Y \to X$ are continuous. To show this lets focus on a single homotopy step $X \to Y$, 
where $Y$ is a new graph obtained from $X$ by a pyramid construction adding a point $z$. We first extend the topology 
on $X$ to a topology on $Y$, where every $U$ containing an edge in $S(z)$ will be given $z$ 
and define $f: X \to Y$ as the inclusion map and $g: Y \to X$ as a projection map which maps the 
new point $z$ to any of its neighbors, lets call this neighbor $z_0$. 
Now $S: g \circ f: X \to X$ and $T: f \circ g: Y \to Y$ are both continuous: actually, $g \circ f$ is
already the identity map so that it is trivially homotopic to the identity. And $T(x)=x$ for all $x \neq z$
and $T(z)=z_0$. To show that $T=f \circ g$ is homotopic
to the identity map on $Y$, we have to find a continuous map $G$ from the product graph $Y \times K_2$ to $Y$
so that $G(y,0) = T(y)$ and $G(y,1) = y$. These requirements actually define $G$ on the product graph which 
can be visualized as two copies of $Y$. (We ignore here the fact that the product topology is not a graph
topology in general). We only need to verify that $G$ is continuous: as usual in topology 
it is only necessary to show for a subbasis $\B$ of $Y$ that for every $U \in \B$, the set $G^{-1}(U)$
is open. But this is obvious because $G^{-1}(U) = U \times \{0,1 \}$. \\

\begin{figure}[H]
\parbox{13cm}{
\parbox{4cm}{\scalebox{0.1}{\includegraphics{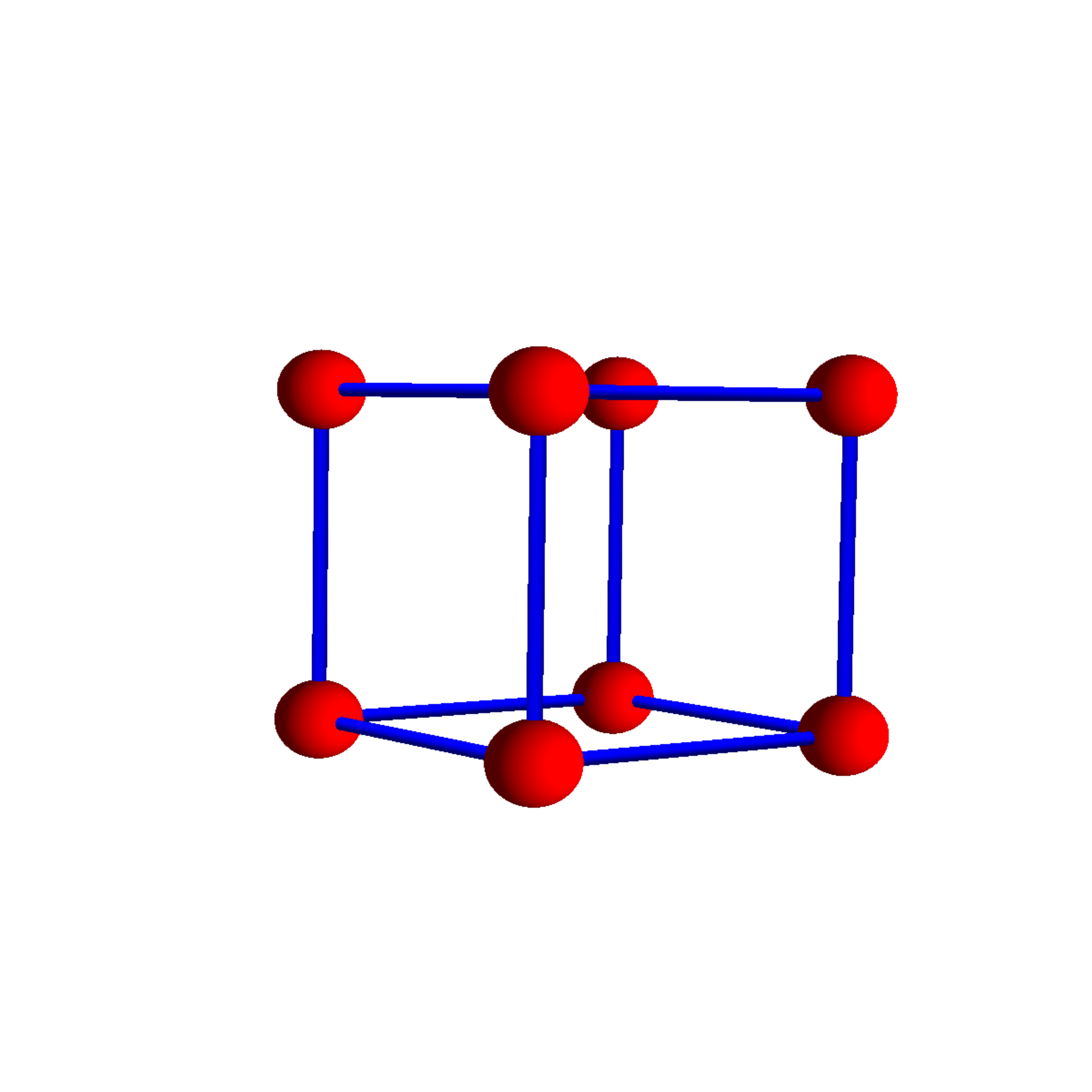}}}
\parbox{4cm}{ $F: X \times K_2 \to X$  }
\parbox{4cm}{\scalebox{0.1}{\includegraphics{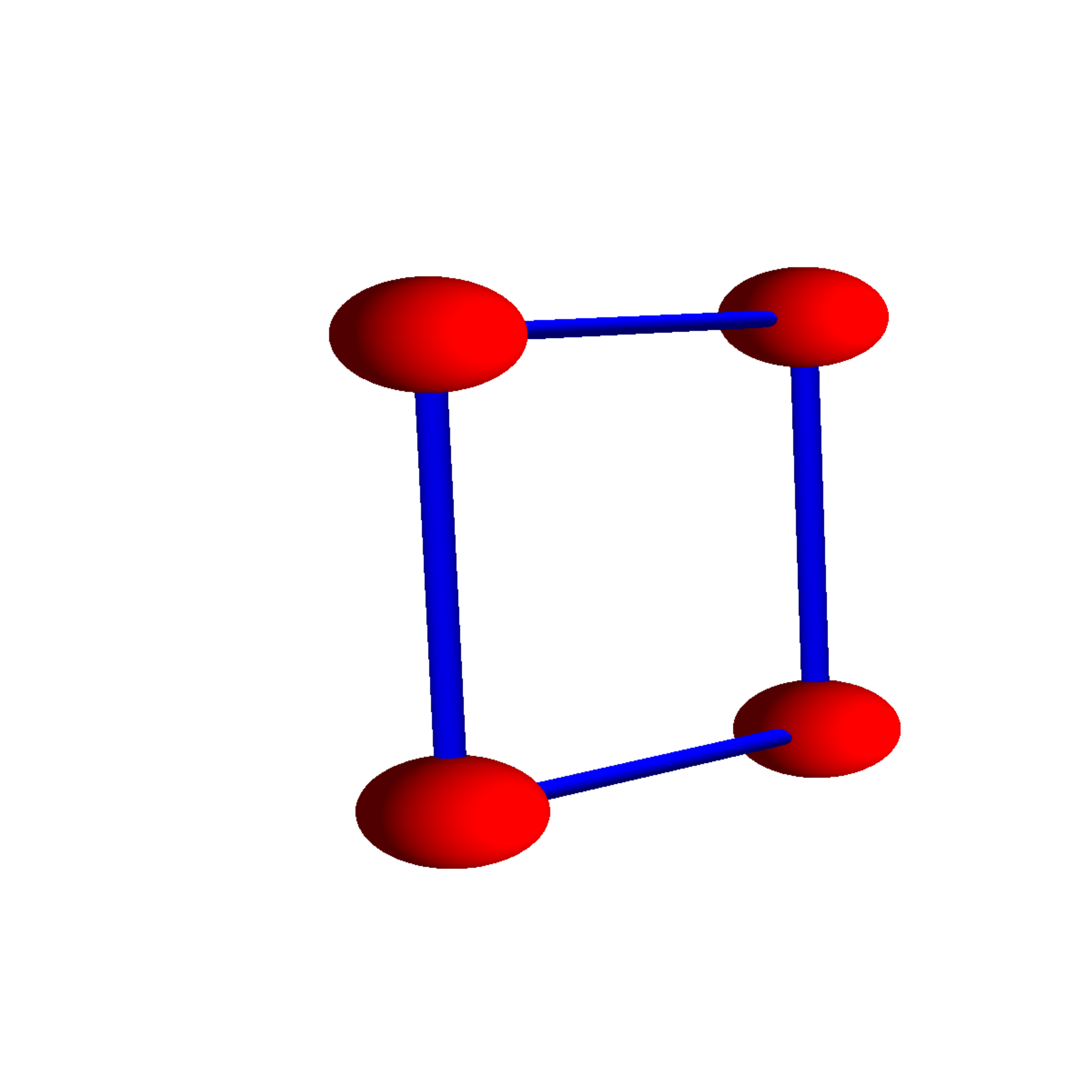}}}
}
\parbox{13cm}{
\parbox{4cm}{\scalebox{0.1}{\includegraphics{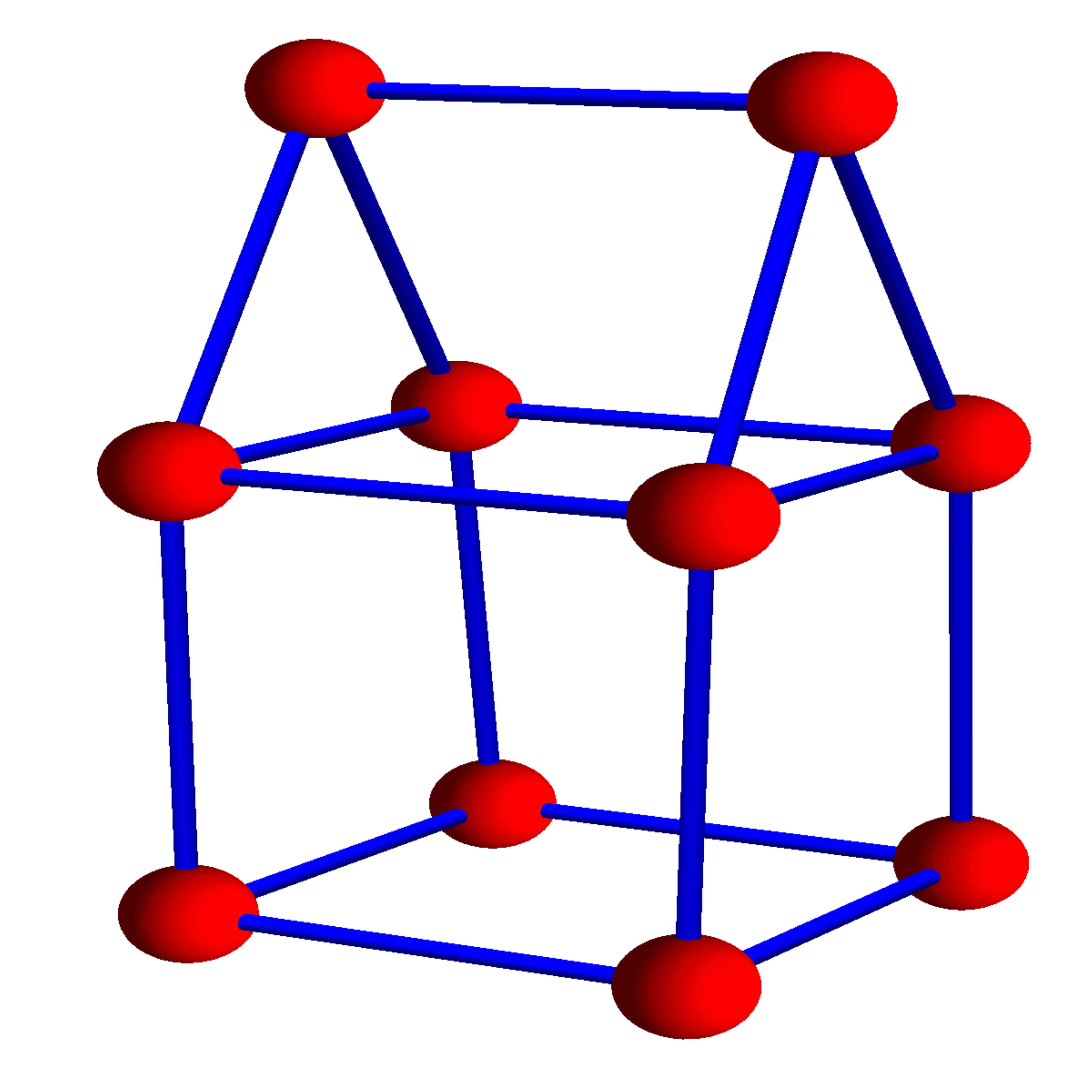}}}
\parbox{4cm}{ $G: Y \times K_2 \to Y$  }
\parbox{4cm}{\scalebox{0.1}{\includegraphics{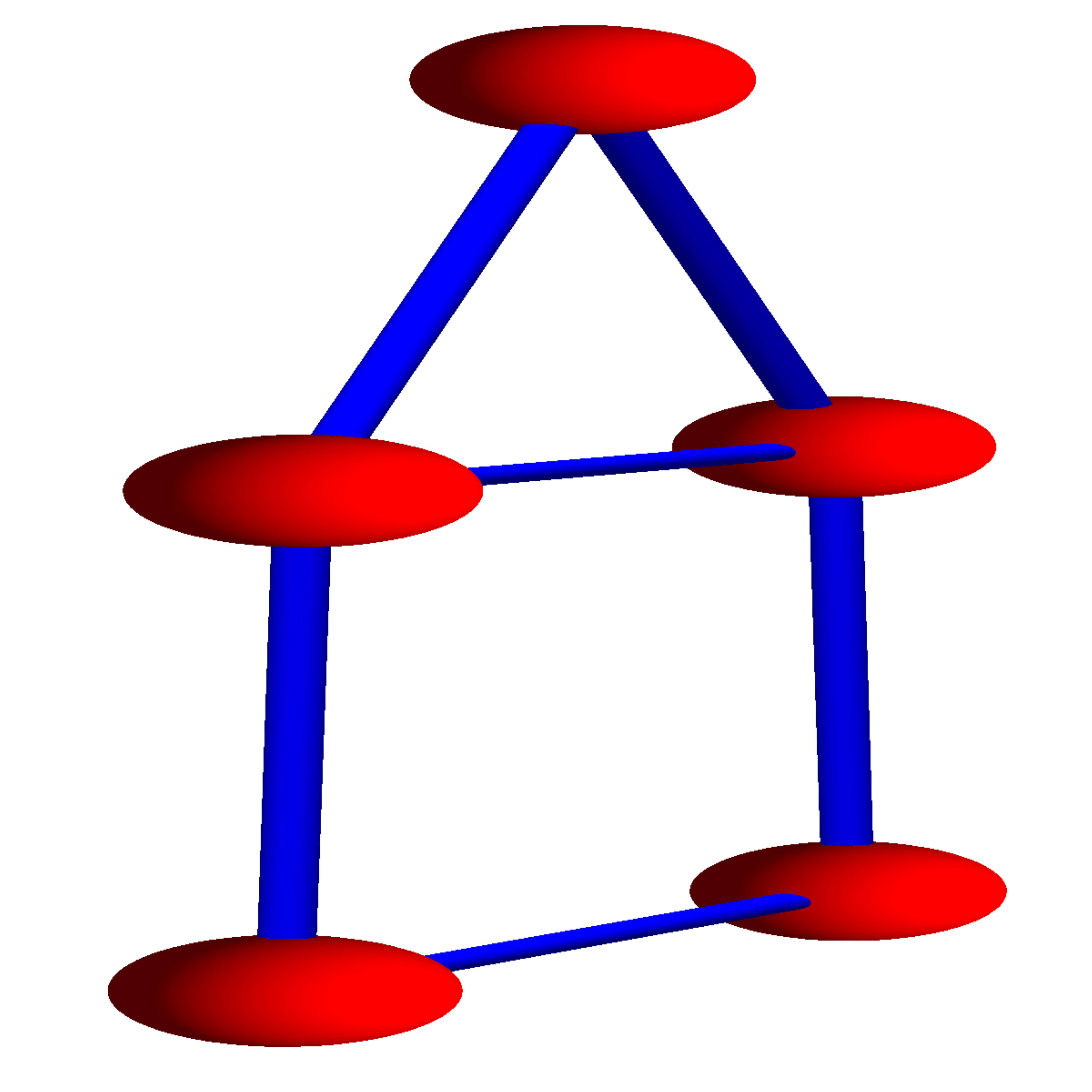}}}
}
\caption{
$Y$ is a homotopy extension of $X$ over a set $U \subset B$ with $B \in \B$
where a new point $z$ has been added. The classical notion of homotopy applies:
there are continuous maps $f:X \to Y, g:Y \to X$
and continuous maps $F:X \times K_2 \to X, G:Y \times K_2 \to Y$
so that $F(x,0)=g(f(x)), F(x,1)=x, G(x,0)=f(g(x)), G(y,1)=y$.
Since $g f(x)=x$ anyway, the first part is trivial. The existence of $G$
follows from the choice of the topology: if we think about $X$ as a subgraph
of $Y$, then $g(f(x)) =x$ for $x \in X$ and $g(f(z)) = z_0$ with $z_0 \in B$. 
Since $Y \times K_2$ carries the product topology also $G$ is continuous.
\label{homotopy}
}
\end{figure}

{\bf 9)} As in the continuum and its adaptations to the discrete \cite{josellisknill} we have to stress that 
contractibility in itself is different from contractibility within an other graph. The boundary circle $C_4$ 
in the wheel graph $W_4$  is not contractible in itself but contractible within $W_4$ because $W_4$ is contractible. 
In the present article, we only deal with contractibility in itself. The distinction is important in 
{\bf Ljusternik-Schnirelmann category}, where the {\bf geometric category} is the minimal number of contractible subgraphs covering $G$. 
The smallest number of in $G$ contractible subgraphs which cover $G$ is called the topological category. 
Both the geometric as well as the topological category are not yet homotopy invariants but we have shown 
the inequality $\tcat(G) \leq \crit(G)$, where $\crit(G)$ is the minimal number of critical points, an injective
function can have on $G$. The number $\cat(G)$ which is the smallest topological category of any graph $H$ 
homotopic to $G$ is a homotopy invariant and $\cat(G) \leq \crit(G)$ is the {\bf discrete Ljusternik-Schnirelmann theorem}
\cite{josellisknill}. \\

{\bf 10)} One can ask why the dimension needs to be invoked at all in the definition of graph topology and homeomorphism. 
Yes, one could look at a subbasis $\B$ consisting of contractible sets and define the nerve graph as all the 
pairs $(A,B)$ for which the intersection is contractible and not empty. A homeomorphsism would be just a bijection 
between topologies. We would still also require the nerve graph to be homeomorphic to $G$. Such a {\bf dimension-agnostic}
setup has serious flaws however. It would make most contractable graphs homeomorphic. Why should a triangularization
of a three-dimensional ball be homeomorphic to a triangularization of a two-dimensional disc? Its not so much the space 
itself which has different topological features but the {\bf boundary}, the set of points for which the unit sphere is 
contractible.  For a disc, the sphere is a circular graph which has a nontrivial homotopy group. For a three-dimensional ball, 
the sphere is a graph which is simply connected. These notions are heavily topological and show that dimension 
must play an important role. It is not only essential for connectivity or simple connectivity, it is also important
for cohomology: if we drill a hole in the middle of a two-dimensional disc, or drill a hold into a three
dimensional disc produces very different topological spaces which any reasonable notion of topology should
honor.  \\

{\bf 11)} Lets look at possible modifications of the definitions and see why we did not do them: 
{\bf a)} the restriction to have finite graphs is not really necessary. Non-compact topological 
spaces can be
modeled with a similar setup. The basis $\B$ is just no more finite then. The notion of topology and 
homeomorphism goes over verbatim. An example is the hexagonal tiling of the
plane which has a natural topology generated by the unit balls $\B$ which are all wheel graphs and two-dimensional. 
This topology is natural also because the curvature is zero everywhere.
{\bf b)} We could ask that open sets in $\B$ either intersect in a contractible set satisfying the 
dimension assumption or then not intersect. We do not see a reason why we should include the second
requirement. It complicates the definition and is not essential. For smaller graphs it would produce
unnecessary constraints. In Figure~(\ref{poster}) we see for example
that some one-dimensional sets in $\B$ intersect without being connected
in the nerve graph. This happens in the two-dimensional components which are
triangles. \\
{\bf c)} We could be more stingy and ask that the intersection of two elements $A,B \in \B$ is 
contractible, independent of the dimension assumption. It is not a good idea because it would
produce exceptions. The set of unit balls on $C_4$ for example, which coincides with the 
discrete topology on $C_4$ would not be a valid topology because the intersection of two
antipodal balls is not contractible. Also the octahedron would not have a natural topology, 
not even the discrete topology. These are not small exceptions: any graph with girth 4 would
have no discrete topology. \\

{\bf 12)} 
When looking at the dynamics of a homeomorphism on a graph, a single topology $\O$ appears 
too limited.  When iterating a homeomorphism, one has to look at a sequence of topologies $\O_n$ which are
equivalent but in which more and more open sets are added, as time moves on.
Similarly as in probability theory, where {\bf martingales} capture stochastic processes $X_k$ by a 
filtration of $\sigma$-algebras $\A_k$ adapted to the random variables in such a
way that $\A_{k+1}$ is generated by $\A_k$ and a random variable $X_k$ which
is independent of $\A_k$, we have to look at a filtration of topologies $\O_k$ 
with subbasis $\B_k$ of $\O_k$ and bijections $\phi_k: \B_k \to \C_k \subset \B_{k+1}$ 
such that $\B_k$ generates $\O_k$ and $\C_k$ generates $\B_{k+1}$. A sequence of 
topologies with subbasis $\B_k$ generating $\B_{k+1}$ and a 
sequence of inclusions $\phi_k$ and an orbit of the homeomorphism is a sequence
of sets $Y_i$, where $Y_{i+1}$ is an atom in $\phi_i V_i$. The dynamics allows to talk about
points $(x_0,x_1,x_2,\dots)$. Similarly, as the orbit of the map $T(x)=2x$ determines the
binary expansion of $x$ and so a filtration of topological spaces,
the graph homeomorphism now defines a filtration of topological space. \\

{\bf 13)} One has looked at classical topological on 
graphs before like in \cite{BFV}. There is almost no overlap. The work in 
\cite{BFV} look at classical topologies on a subclass of countable or finite graphs 
which are {\bf Alexandroff spaces} in the sense that arbitrary intersections of open sets are open. 
The paper studies the notion of homeomorphism because it produces equivalence classes
ou graphs which are easier to distinguish from the complexity point of view. 
We here only look at finite topologies, where every topology is Alexandroff. 
We look here at the concept of contractibility
for a subbase and the concept of dimension. Contractibiliy is essential for us because we aimed
to have Euler characteristic, homotopy structures and cohomology invariant under homeomorphisms. 
That dimension is essential is because notions like connectivity, 
fundamental group, topology of the boundary do in an essential way depend on dimension: 
the boundary of a three two-dimensional ball has a different topology than the boundary 
of a three-dimensional ball. This is especially true in geometric situations which are 
important in applications like computer graphics. We also have seen that it is possible 
conceptually even in the discrete to single out Euclidean structures among metric 
spaces by using contractibility and dimension. \\

\begin{figure}[H]
\scalebox{0.14}{\includegraphics{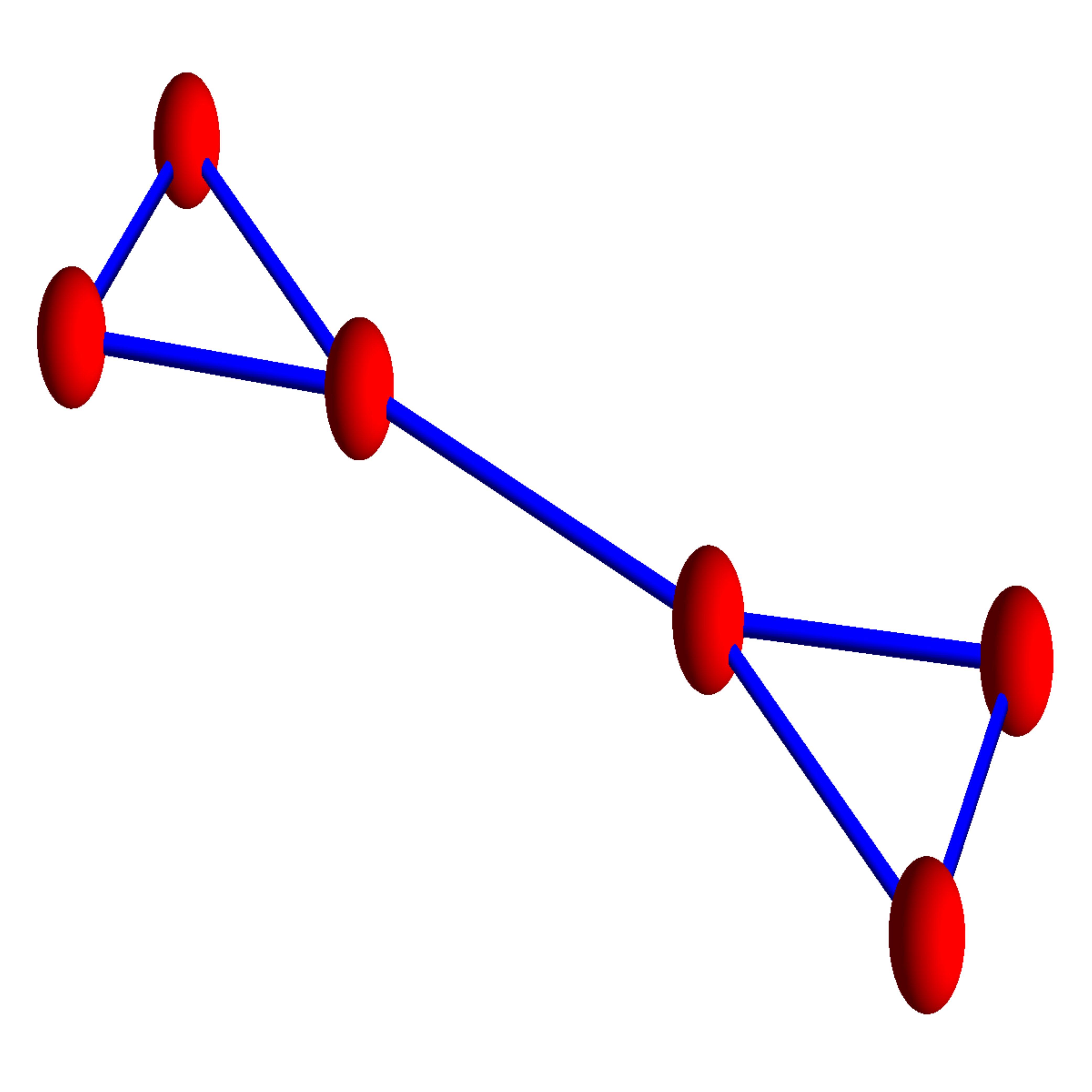}}
\scalebox{0.14}{\includegraphics{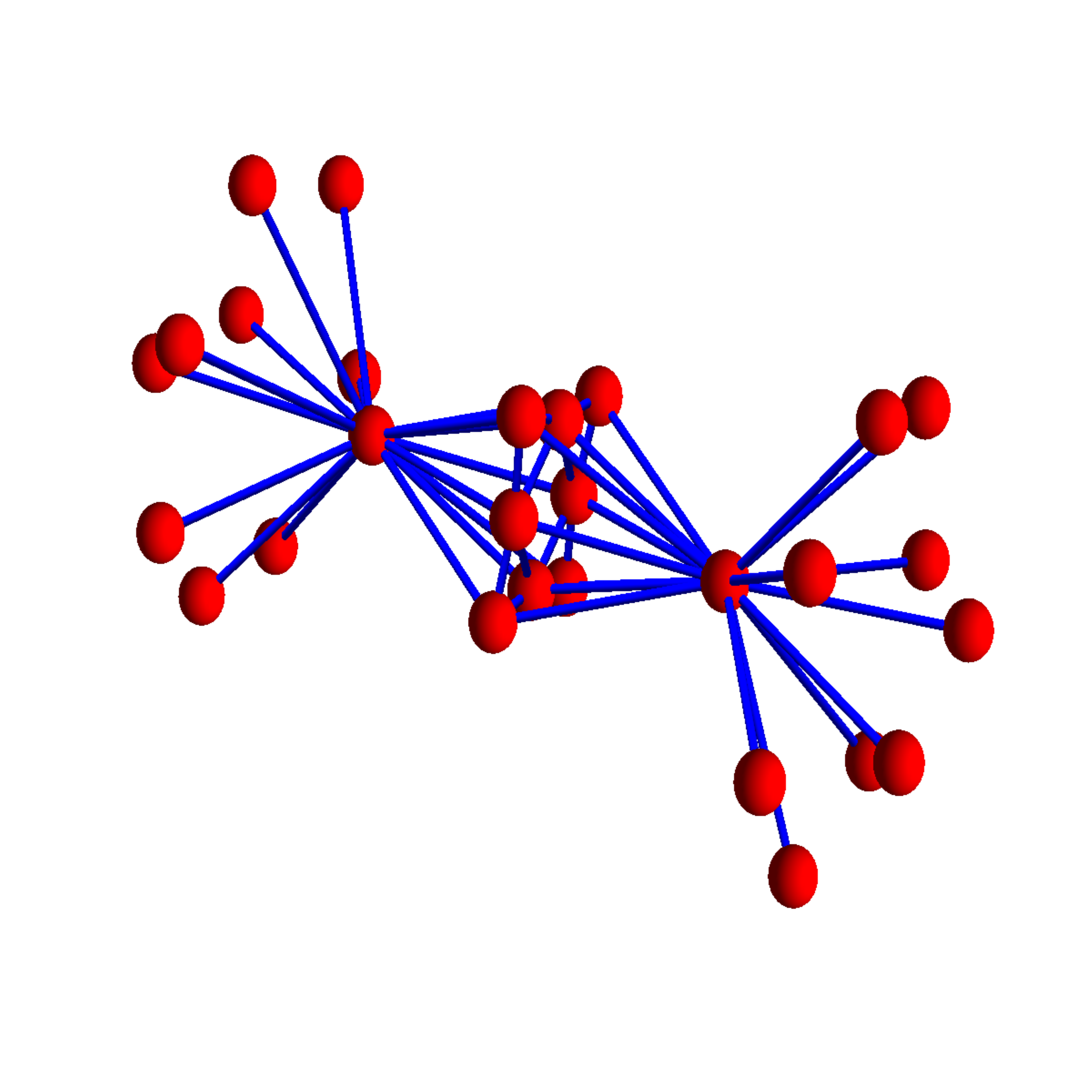}}
\caption{
Two examples, for which the unit ball topology is not a graph topology.
In the left case, $\G$ consists of two separate triangles because the dimension assumption
does not connect $B(x)$ with $B(y)$. In the right case, $\G$ is contractible while $G$ is not. 
This is a case, where the unit balls $B(x),B(y)$ have dimension smaller than $2$ while
the intersection $B(x) \cap B(y)$ has dimension $2$. This forces us to connect $B(x)$
with $B(y)$ in $\G$. Of course there are natural and optimal topologies in both cases: 
in the left case take the two triangles connected by a line graph of length $4$.
In the right case, cover and connect each hair using $1$-dimensional line graphs.
\label{ball}
}
\end{figure}

{\bf 14)} The intersection of two contractible adjacent balls is contractible but the intersection of
two balls of distance $2$ in general does not have this property as the case of 
the cyclic graph $C_4$ shows which has two unit balls intersecting in a disconnected graph. 
In the following, we mean with $\B$ a minimal set of unit balls, discarding multiple copies of
the same graph. For the triangle for example, $\B$ contains only one set, the triangle itself.
Do unit balls $\B$ form a topology? It is often the case. But
two triangles joined by a vertex show a graph $G$ for which the set of unit 
balls generates a nerve graph which is not homotopic to $G$. 
The case of two tetrahedra joined along a triangle shows an example there the intersection
has dimension $2$ while $B(x),B(y)$ have dimension $3$. Add one-dimensional hairs at the
vertices which are not in the intersection can now render the dimension $B(x),B(y)$ arbitrarily close
to $1$, while the dimension of the intersection remains $2$. This shows that the two balls of distance
$2$ must be connected in the nerve graph 
if the dimension of the two balls $B(x)$ and $B(y)$ are both smaller or equal to $1$.  See Figure~(\ref{ball}). 
Lets assume that graph has the property that the dimension assumption is 
true for any adjacent unit balls and false for any balls of distance $2$ form a topology. Then this defines
a graph topology: {\bf Proof:} {\it (ii) Every unit ball $B(x)$ is contractible unconditionally.} 
Use induction with respect to the order $n$ of $B(x)$. We can assume $G=B(x)$.
Take a point $z$ in $S(x)$ and remove it with all connections producing a new graph $H$.
This is a homotopy step $G \to H$ because $H$ is the new $B(x)$.
{\it (iii) If $(x,y) \in E$, then $B(x) \cap B(y)$ is contractible.}
Use induction with respect to the order $n$ of the graph $H=B(x) \cap B(y)$ which
we can assume to be $G$. If $n=3$, then $H$ is a 
triangle $x,y,z$, which is contractible. Assume it has been shown for all graphs 
order $n$. Consider the case $n+1$ and chose a vertex $z$ in $G$ different from $x,y$.
It is connected to a subgraph $H$ of order $n$ containing $x$ and $y$. The graph $G$
with $z$ removed is of the form $B(x) \cap B(y)$ in $H$. By induction assumption, $H$
is contractible so that $G$ as a homotopy extension is contractible too.
(More generally, any finite intersection of adjacent unit balls is contractible:
If $x_1, \dots, x_k$ are the centers of the balls, go through the same proof
showing that $B(x_1) \cap B(x_2)$ is contractible within $H= G \cap \bigcap_{j=3}^k  B(x_j)$.) 
{\it (iv) Two unit balls $B(x),B(y)$ with $d(x,y)=3$ do not intersect.} 
By the triangle inequality.
{\it (v) The graph $G$ is homotopic to $\G$.} 
The nerve graph $\G$ is the same as the graph $G$ and the homotopy 
assumption is satisfied automatically. \\

{\bf 15)} Here is a question we can not answer yet: is it true that if a graph $H$ is planar and equipped with 
an (optimal) topology and $H$ is homeomorphic to $G$ which is also equipped with an (optimal) topology, then $G$ is planar?  
By the Kuratowski theorem, non-planarity is equivalent to have no subgraph which is 1-homeomorphic to $K_5$ nor $K_{3,3}$. 
Lets for example look at a graph which contains $K_5$, then there exists an open set which has dimension at least $5$.
The image of this open set must have dimension at least $5$ too and therefore contain a copy of $K_5$. 
Now lets look at a graph which contains a 1-homeomorphic copy of $K_5$. While it has smaller dimension, we need more
open sets to cover it because it is no more contractible. In the case when all edges are extended, we need at least
10 open sets to cover it. The image of this produces a 1-homeomorphic graph. This still does not cover all the possibilities
yet for $K_5$ and then we also have to deal with the utility graph $K_{3,3}$. 
While $K_{3,3}$ is one-dimensional we can not take the indiscrete topology because $\chi(K_{3,3}) = -3$ shows that $K_{3,3}$ is not
contractible and a topology needs more open sets. The graph $K_5$ with the indiscrete topology requires the image 
graph to have a $K_5$ subgraph. 

\vspace{12pt}
\bibliographystyle{plain}

\end{document}